\documentclass[sn-mathphys-num]{sn-jnl}

\usepackage{subfigure}

\usepackage{amsmath,amssymb,amsfonts}
\usepackage{amsthm}
\usepackage{mathrsfs}
\usepackage{xcolor}
\usepackage{textcomp}
\usepackage{listings}
\usepackage{color}

\newtheorem{theorem}{Theorem}[section]
\newtheorem{lemma}[theorem]{Lemma}
\newtheorem{definition}[theorem]{Definition}

\newtheorem{proposition}[theorem]{Proposition}
\newtheorem{problem}{Problem}
\newtheorem{remark}[theorem]{Remark}
\newtheorem{assumption}{Assumption}
\newtheorem{example}[theorem]{Example}

\DeclareMathAlphabet{\mymathbb}{U}{BOONDOX-ds}{m}{n}

\newcommand{\bx}{{\mathbf{x}}}

\newcommand{\by}{{\mathbf{y}}}
\newcommand{\ba}{{\mathbf{a}}}
\newcommand{\bz}{{\mathbf{z}}}

\newcommand{\bu}{{\mathbf{u}}}

\newcommand{\bq}{{\mathbf{q}}}
\newcommand{\bv}{{\mathbf{v}}}
\newcommand{\bc}{{\mathbf{c}}}

\newcommand{\setdef}[2]{\{#1 : #2\}}

\newcommand{\Kc}{\mathcal{K}}

\newcommand{\Oc}{\mathcal{O}}
\newcommand{\Lc}{\mathcal{L}}
\newcommand{\Nc}{\mathcal{N}}

\newcommand{\Ec}{\mathcal{E}}

\newcommand{\realpos}{\mathbb{R}_{> 0}}

\newcommand{\zero}{\mymathbb{0}}

\newcommand{\real}{\mathbb{R}}

\newcommand{\1}{\mathds{1}}
\newcommand{\Cc}{\mathcal{C}}

\DeclareSymbolFont{bbold}{U}{bbold}{m}{n}
\DeclareSymbolFontAlphabet{\mathbbold}{bbold}

\newcommand{\norm}[1]{\left\lVert#1\right\rVert}

\newcommand\oprocendsymbol{\hbox{$\bullet$}}
\newcommand\oprocend{\relax\ifmmode\else\unskip\hfill\fi\oprocendsymbol}

\renewcommand{\theenumi}{(\roman{enumi})}

\newcommand*{\QEDA}{\hfill\ensuremath{\blacksquare}}%

\newcommand\xqed[1]{%
  \leavevmode\unskip\penalty9999 \hbox{}\nobreak\hfill
  \quad\hbox{#1}}

\newcommand\problemfinal{\xqed{$\triangle$}}

\newcounter{countitems}
\newcounter{nextitemizecount}
\newcommand{\setupcountitems}{%
  \stepcounter{nextitemizecount}%
  \setcounter{countitems}{0}%
  \preto\item{\stepcounter{countitems}}%
}
\makeatletter
\newcommand{\computecountitems}{%
  \edef\@currentlabel{\number\c@countitems}%
  \label{countitems@\number\numexpr\value{nextitemizecount}-1\relax}%
}
\newcommand{\nextitemizecount}{%
  \getrefnumber{countitems@\number\c@nextitemizecount}%
}
\newcommand{\previtemizecount}{%
  \getrefnumber{countitems@\number\numexpr\value{nextitemizecount}-1\relax}%
}
\makeatother    
{\end{itemize}%
\unskip\computecountitems\ifnumcomp{\previtemizecount}{>}{4}{\end{multicols}}{}}

\newcommand{\longthmtitle}[1]{\mbox{}\emph{(#1):}}

\newcommand{\comment}[1]{} 
\allowdisplaybreaks

\renewcommand{\baselinestretch}{1}

\newcommand\blfootnote[1]{%
  \begingroup
  \renewcommand\thefootnote{}\footnote{#1}%
  \addtocounter{footnote}{-1}%
  \endgroup
}

\begin{document}

\title{\bf \large Control Barrier Function-Based Safety Filters: Characterization of Undesired Equilibria, Unbounded Trajectories, and Limit Cycles\thanks{}}

\author*[1]{\fnm{Pol} \sur{Mestres}}\email{pomestre@ucsd.edu}

\author[2]{\fnm{Yiting} \sur{Chen}}\email{yich4684@bu.edu}

\author[2]{\fnm{Emiliano} \sur{Dall'Anese}}\email{edallane@bu.edu}

\author[1]{\fnm{Jorge} \sur{Cort\'{e}s}}\email{cortes@ucsd.edu}

\affil*[1]{\orgdiv{Department of Mechanical and Aerospace Engineering}, \orgname{UC San Diego}}

\affil[2]{\orgdiv{Department of Electrical and Computer Engineering}, \orgname{Boston University}}




\abstract{
This paper focuses on ``safety filters'' designed based on Control Barrier Functions (CBFs): these are modifications of a nominal stabilizing controller  typically utilized in safety-critical control applications to render a given subset of states forward invariant. The paper investigates the dynamical properties of the closed-loop systems, with a focus on characterizing undesirable behaviors that  may emerge due to the use of CBF-based filters. These undesirable behaviors include unbounded trajectories, limit cycles, and undesired equilibria, which can be locally stable and even form a continuum. 
Our analysis offer the following contributions: (i) conditions under which trajectories remain bounded and (ii) conditions under which limit cycles do not exist; (iii) we show that undesired equilibria can be characterized by solving an algebraic equation, and (iv) we provide examples that show that asymptotically stable undesired equilibria can exist for a large class of nominal controllers and design parameters of the safety filter (even for convex safe sets). 
Further, for the specific class of planar systems, (v) we provide explicit formulas for the total number of undesired equilibria and the proportion of saddle points and asymptotically stable equilibria, and (vi) in the case of linear planar systems, we present an exhaustive analysis of their global stability properties. Examples throughout the paper illustrate the results.
}



\maketitle

\blfootnote{* A preliminary version of this paper appeared as~\cite{YC-PM-EDA-JC:24-cdc} at the IEEE Conference on Decision and Control.}

\section{Introduction}\label{sec:introduction}

Cyber-physical systems and autonomous systems -- from individual and connected self-driving vehicles
to complex infrastructures such as smart grids and transportation networks -- must comply with safety and operational constraints, while ensuring a desired level of performance and efficiency. Control Barrier Functions (CBFs) have emerged as a widely used tool for the design of architectural control frameworks that guarantee the forward invariance of a given set of desirable,  ``safe'' states (termed in what follows as \emph{safe set} of the system)~\cite{PW-FA:07,ADA-SC-ME-GN-KS-PT:19, WX-CGC-CB:23}.  In this context, CBF theory is a workhorse  for the design of \emph{safety filters}, which are utilized to (minimally) adjust the input provided by a nominal controller (typically designed to achieve properties such as stability, optimality, or robustness) to ensure forward invariance of the safe set. However,
recent works~\cite{MFR-APA-PT:21,XT-DVD:24,PM-JC:23-csl,YY-SK-BG-NA:23,GN-SM:22,YC-PM-EDA-JC:24-cdc}
%
%
have shown that when the nominal controller is augmented with a safety filter, the resulting closed-loop system may  not preserve  desired properties such as stability and robustness, and it may in fact lead to undesirable behaviors  such that the emergence of undesired equilibria. This paper provides a detailed characterization of the undesired equilibria and delves deep into these undesirable behaviors, presenting new findings that demonstrate how safety filters can give rise to unbounded trajectories and limit cycles. 

%
%

\textbf{Literature review}. CBFs~\cite{PW-FA:07,ADA-SC-ME-GN-KS-PT:19, WX-CGC-CB:23} are a well-established tool to design controllers that render a given set forward invariant. 
Several works have combined CBFs with other control-theoretic tools, such as control Lyapunov functions (CLFs)~\cite{EDS:98}, in order to obtain controllers with both safety and stability guarantees~\cite{ADA-XX-JWG-PT:17,XT-DVD:24,PM-JC:23-csl}, and robustness guarantees~\cite{MJ:18,FC-JJC-BZ-CJT-KS:21}; CBF-based controllers with input constraints have also been developed in, e.g.,~\cite{DRA-DP:21}.
Within this line of works, this paper focuses on \textit{safety filters}~\cite{LW-ADA-ME:17}.  
Safety filters yield a controller that minimally modifies the nominal one while ensuring forward invariance of a safe set. A critical research question is whether the closed-loop system under the safety filter retains the stability and robustness properties of the dynamical system with the nominal controller only.
A first set of results about the dynamical properties of the closed-loop system with safety filters  was provided in~\cite{WSC-DVD:22-tac}, which estimated the region of attraction of the desirable equilibrium (taken to be the origin w.l.o.g.). However, 
questions remain on the asymptotic behavior of the trajectories with initial condition outside the estimated region of attraction. The emergence of undesired equilibria due to CBF-based controllers which are asymptotically stable was noted in~\cite{MFR-APA-PT:21}, and since then they have been studied profusely~\cite{XT-DVD:24,PM-JC:23-csl,YY-SK-BG-NA:23,GN-SM:22}; here, in the controller and filter design, the CBF is assumed to be given throughout the analysis. 
In previous work~\cite{indep-cbf}, we have shown that the undesired equilibria that emerge in safety filters and their stability properties, whatever they may be, are independent of the choice of CBF, for a broad class of CBFs (related by an appropriate equivalence relation). However,~\cite{indep-cbf} does not study what the actual dynamical properties are, which is precisely the focus of this paper.
%
%
Our recent work~\cite{YC-PM-EDA-JC:24-cdc} provides a characterization of undesired equilibria and, for the special case of linear planar systems and ellipsoidal obstacles, finds such undesired equilibria explicitly and characterizes their stability properties. The results show that for this class of systems and safe sets, if the system is underactuated there always exists a single undesired equilibrium, whereas if the system is fully actuated, the number and stability properties of undesired equilibria depend on the choice of nominal controller.
However, the results of~\cite{YC-PM-EDA-JC:24-cdc} are limited to  linear and planar dynamics, and to safe sets that are the complement of an ellipsoid.
%
%

\textbf{Statement of contributions}. We study the dynamical properties of closed-loop systems under CBF-based safety filters, paying special attention to the emergence of undesired behaviors. The contributions are summarized  as follows:
\begin{enumerate}
    \item We characterize the undesired equilibria that emerge in closed-loop systems with  general control-affine dynamical systems, a stabilizing, locally-Lipschitz nominal controller, and a CBF-based safety filter. We show that finding the undesired equilibria is equivalent to solving an algebraic equation.

    \item We provide an example showing that, in general, the set of undesired equilibria can be a continuum. This motivates our next contribution, which consists in providing conditions under which the equilibria are isolated points.

    \item We show that, in general, the trajectories of the closed-loop system can be unbounded. We then show how, by appropriately selecting some of the parameters of the safety filter, and under mild assumptions, one can ensure that the trajectories of the closed-loop system remain bounded. 
    
    \item In the case of planar systems, we show that by suitably tuning the parameters of the safety filter, the closed-loop system does not contain any limit cycles. This implies that all trajectories of the closed-loop system either converge to the origin or to an undesired equilibrium.
    Therefore, the solutions of the algebraic equation for the undesired equilibria define all the possible limits of trajectories of the closed-loop system.
    Since solving this algebraic equation for general systems is complicated, we also provide qualitative results regarding the structure of the set of undesired equilibria. 
    We show that if the safe set is bounded, the number of undesired equilibria is even, and half of them are saddle points, whereas if the unsafe set is bounded, the number of undesired equilibria is odd, equal to $2l-1$ with $l\in\mathbb{Z}_{>0}$, and $l$ of them are saddle points.

    \vspace{.1cm}

    \item We illustrate the existence of undesired equilibria and their stability properties for linear planar systems in a variety of different cases.
    For underactuated systems and safe sets that are parametrizable in polar coordinates, we show that no undesired equilibria exist. We provide different examples in which asymptotically stable undesired equilibria exist, including a fully actuated system with a 
    convex safe set, an underactuated system with a safe set not parameterizable in polar coordinates, and an underactuated and fully actuated systems with a bounded unsafe set.
    We also provide an example with nonconvex unsafe set where asymptotically stable undesired equilibria exist for any choice of stabilizing nominal controller.
    Finally, for the special case where the unsafe set is an ellipse,
    we provide analytical expressions for the undesired equilibria and their stability properties. We show that if the system is underactuated, there exists exactly one undesired equilibria, which is a saddle point, whereas if the system is fully actuated, the behavior is much richer and includes a variety of different cases.
    \end{enumerate}
    Our contributions highlight the intrincate relationship between the system dynamics, the geometry of the safe set, and the existence of undesired equilibria and their stability properties.
    They also serve as a cautionary note to practitioners, for whom we provide a variety of methods to tune (when possible) their controllers to avoid this plethora of undesirable behaviors.

\section{Preliminaries}\label{sec:preliminaries}

In this section we introduce the notation and provide preliminaries on CBFs and safety filters. The reader familiar with these contents can safely skip this section.

\subsection{Notation}

We denote by $\mathbb{Z}_{>0}$, $\mathbb{Z}_{\geq0}$ and $\real$ the set of positive integers, nonnegative integers, and real numbers, respectively.  
We use bold symbols to represent vectors and non-bold symbols to represent scalar quantities. Let $n\in\mathbb{Z}_{>0}$;
$\mathbf{0}_n$ represents the $n$-dimensional zero vector, $\zero_n$ the $n\times n$-dimensional zero matrix and $\textbf{I}_n$ the $n$-dimensional identity matrix.
We also write $[n] = \{ 1, \hdots, n \}$.
Given $\bx\in\real^{n}$, $\norm{\bx}$ denotes its Euclidean norm. Given a matrix $G\in\real^{n\times n}$, $\norm{\bx}_{G} = \sqrt{\bx^T G \bx}$ and $\text{det}(G)$ denotes its determinant. A function 
$\beta:\real\to\real$ is of extended class $\mathcal{K}_{\infty}$ if $\beta(0)=0$,
$\beta$ is strictly increasing and $\lim\limits_{s\to\pm\infty}\beta(s) = \pm\infty$.
Given a set $S\subset\real^n$, we denote by $\text{Int}(S)$ and $\partial S$ the interior and boundary of $S$, respectively. For a continuously differentiable function  $h:\real^{n}\to\real$, $\nabla h(\bx)$ denotes its gradient at $\bx$.  
Consider the system $\dot{\bx} = f(\bx)$, with $f:\real^n\to\real^n$ locally Lipschitz. 
Then, for any initial condition $\bx_0 \in \real^n$ at time $t_0$, there exists a maximal interval of existence $[t_0,t_1)$ such that $\bx(t;\bx_0)$ is the unique solution to $\dot{\bx} = f(\bx)$ on 
$[t_0,t_1)$, cf.~\cite{EDS:98}.
For $f$  continuously differentiable
and $\bx_*$ an equilibrium point of $f$ (i.e., $f(\bx_*)=\mathbf{0}_n$), $\bx_*$ is \emph{degenerate} if the Jacobian of $f$ evaluated at $\bx_*$ has at least one eigenvalue with real part equal to zero; otherwise, we refer to  $\bx_*$ as \emph{hyperbolic}. Given a hyperbolic equilibrium point with $k\in\mathbb{Z}_{>0}$ eigenvalues with negative real part, the Stable Manifold Theorem~\cite[Section 2.7]{LP:00} ensures that there exists an invariant $k$-dimensional manifold $S$ for which all trajectories with initial conditions lying on $S$ converge to $\bx_*$.
The global stable manifold at $\bx_*$ is defined as $W_s(\bx_*)=\bigcup\limits_{ \{t\leq0, \ \bx_0\in S \} }\bx(t;\bx_0)$.
An equilibrium point $\bx_*$ is \textit{isolated} if there exists an open neighborhood $\mathcal{U}$ of $\bx_*$ such that $\bx_*$ is the only equilibrium point in $\mathcal{U}$. 


\subsection{Control barrier functions and safety filters} 

Consider a control-affine system of the form
\begin{align}\label{eq:control-affine-sys}
  \dot{\bx}=f(\bx)+g(\bx)\bu,
\end{align}
where $f:\real^{n}\to\real^{n}$ and $g:\real^{n}\to\real^{n\times m}$
are locally Lipschitz, $\bx\in\real^{n}$ is the state, and
$\bu\in\real^{m}$ is the input. Suppose that a locally Lipschitz nominal controller $k:\real^n\to\real^m$ is designed so that the system 
$\dot \bx = \tilde{f}(\bx) := f(\bx)+g(\bx)k(\bx)$ has a unique equilibrium, which furthermore is globally asymptotically stable. In the remainder of the paper, we assume without loss of generality that such equilibrium is the origin. In the following, we introduce the notion of safety filter and filtered system. To this end, we first recall the  definition of CBF. 

\vspace{.1cm} 

\begin{definition}[Control Barrier
    Function]\label{def:cbf}
  Let $\Cc\subset\real^n$ be a subset of $\real^n$.
  Let $h:\real^{n}\to\real$ be a continuously differentiable function
  such that
  \begin{subequations}
      \begin{align}
          \Cc &= \setdef{\bx\in\real^n}{h(\bx)\geq0}, \\
          \partial\Cc &= \setdef{\bx\in\real^n}{h(\bx)=0}.
      \end{align}
      \label{eq:safe-set}
  \end{subequations}
  The
  function $h$ is a  \textbf{CBF} of $\Cc$ for the
  system~\eqref{eq:control-affine-sys} if there exists an extended class
  $\mathcal{K}_{\infty}$ function $\alpha$ such that, for all $\bx \in \Cc$, there exists $\bu \in\real^{m}$ satisfying $\nabla h(\bx)^\top (f(\bx)+g(\bx)\bu) + \alpha(h(\bx)) \geq 0$. If this inequality holds strictly, we refer to $h$ as a strict CBF.
\end{definition}

\smallskip

The set $\Cc$ corresponds to the set of \emph{safe} states. With this definition of CBF, hereafter we refer to the \emph{filtered} system as:
\begin{equation}\label{eq:general-system-1}
\dot{\bx}= \tilde{f}(\bx) +g(\bx)v(\bx),
\end{equation}
 where the map
 $v:\real^n\to\real^m$ is defined as the unique solution to the following optimization problem:
\begin{align} \label{eq:v-problem} 
v(\bx)&=\arg\underset{ \boldsymbol{\theta} \in \mathbb{R}^{m} }{\min}  \| \boldsymbol{\theta} \|_{G(\bx)}^2 \\
\notag
&\text { s.t. } \nabla h(\bx)^\top (\tilde{f}(\bx)+g(\bx) \boldsymbol{\theta} )+\alpha( h(\bx)) \geq 0, 
\end{align}
%
%
with $G:\real^n\to\real^{m\times m}$  continuously differentiable and positive definite for all $\bx\in\real^n$. Note that~\eqref{eq:v-problem} is a linearly-constrained convex quadratic program (QP), parametrized by the state vector $\bx$. We  assume the following.

\vspace{.1cm}

\begin{assumption}[Origin in the interior of $\mathcal{C}$]\label{as: interior eq}
     The origin is in the interior of the safe set, i.e., $h(\mathbf{0}_n)>0$.
\end{assumption}
%
%
%
%

\vspace{.1cm}

%
%

\vspace{.1cm}

We note that if $h$ is a strict CBF,~\eqref{eq:v-problem} is strictly feasible for all $\bx\in\Cc$.
Moreover, since the objective function in~\eqref{eq:v-problem} is strongly convex, and the constraints are affine, $v(\bx)$ is unique and well-defined for all $\bx\in\Cc$.
Furthermore, if $h$ is a strict CBF, using arguments similar to~\cite[Lemma III.2]{MA-NA-JC:25-tac}, one can show that $v(\bx)$ is locally Lipschitz. 
If $h$ is a strict CBF, it also follows that 
$\frac{\partial h}{\partial \bx}(\bx) \neq \mathbf{0}_n$ for all $\bx\in\partial\Cc$ and therefore from~\cite[Thm.~2]{ADA-SC-ME-GN-KS-PT:19}, it follows that the set $\Cc$ is forward invariant for the system~\eqref{eq:general-system-1}. Because of this feature, $v$ is typically referred to as the \emph{safety filter} associated with~$k$.

\section{Problem Statement}\label{sec:problem-statement}

We consider a control-affine system as in~\eqref{eq:control-affine-sys} and a safe set $\Cc\subset\real^n$ defined as the $0$-superlevel set of a differentiable function $h:\real^n\to\real$.  We assume that $h$ is a strict CBF of $\Cc$; this function, along with the associated extended class $\Kc_{\infty}$ function $\alpha$ in Definition~\ref{def:cbf}, are assumed to be given.  In our setup, we consider the general setting where: (i) a given controller $k:\real^n\to\real^m$ globally asymptotically stabilizes the origin; (ii) a safety filter is added as in~\eqref{eq:general-system-1} to render $\Cc$ forward invariant; and, (iii) Assumption~\ref{as: interior eq} holds. 
%
%
Figure~\ref{fig:diagram-safety-filters} illustrates the setup considered in this paper.
\begin{figure}
  \centering
  {\includegraphics[width=0.9\textwidth]{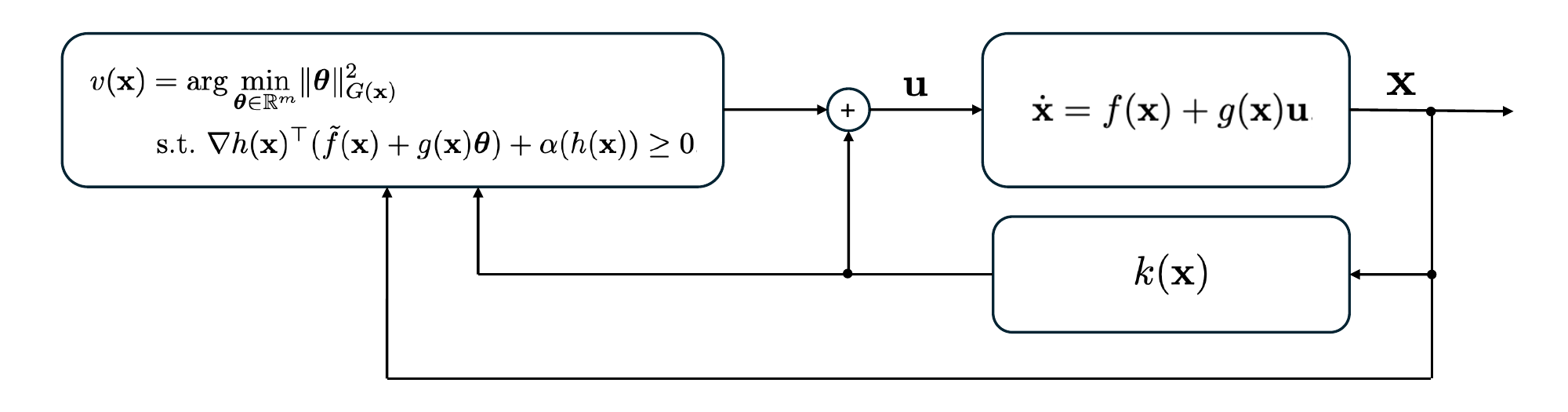}} 
  \caption{Closed-loop system that is the subject of this paper. A nominal controller $k$ is used as input to the safety filter, that finds the controller closest to $k$ that satisfies the CBF condition.}
  \label{fig:diagram-safety-filters}
\end{figure}

Even though the origin is globally asymptotically stable under the nominal controller $k$, and $v$ is designed to minimally modify $k$ while ensuring that $\Cc$ is forward invariant, studying the dynamical behavior of~\eqref{eq:general-system-1} is challenging.
As noted in, e.g.,~\cite{WSC-DVD:22-tac}, the filtered system does not necessarily inherit the global asymptotic stability properties of the controller $k$, and can even have undesired equilibria~\cite{MFR-APA-PT:21,XT-DVD:24,PM-JC:23-csl,YY-SK-BG-NA:23,YC-PM-EDA-JC:24-cdc} (cf. Figure~\ref{fig:4_figures_v3}). 
\begin{figure*}
  \centering
    \subfigure[Unbounded trajectories]{\includegraphics[width=.48\linewidth]{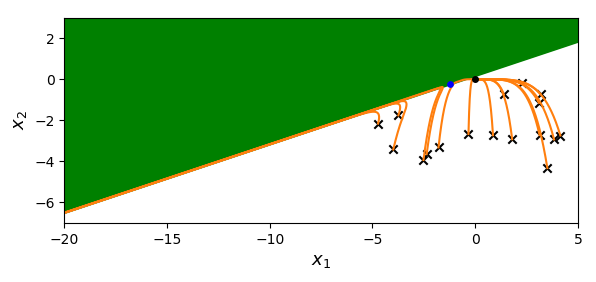}}
   \subfigure[Stable undesired equilibrium]{\includegraphics[width=.48\linewidth]{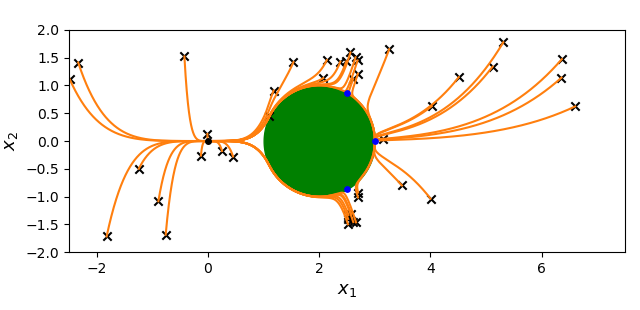}}
  \caption{Control-affine systems with a safety filter with (a) half-plane and (b) circular obstacles. The plots show the trajectories from random initial conditions, the undesired equilibria (colored in blue), and the desired equilibrium (the origin, colored in black).}
  \label{fig:4_figures_v3}
\end{figure*}
Most of these works focus on studying conditions under which such undesired equilibria exist or can be confined to specific regions of interest, but do not study other dynamical properties of the closed-loop system, such as boundedness of trajectories, existence of limit cycles, or regions of attraction. In practical applications, these properties (e.g., ensuring that trajectories are bounded, unexpected limit cycles do not arise, or the region of attraction of undesired equilibria is small) is critical to guarantee a desirable performance of the closed-loop system.
With this premise, the goal of this paper is as follows: 

\smallskip



\begin{problem}\label{problem}
 Given the system~\eqref{eq:control-affine-sys} and the safety filter $v(\bx)$, characterize the boundedness of trajectories, the region of attraction of the origin, existence of undesired equilibria and their regions of attraction, and the existence of limit cycles for the dynamics~\eqref{eq:general-system-1}. 
\end{problem}

\smallskip

%
%

%
%

\section{Undesired Equilibria, Bounded Trajectories, and Limit Cycles}\label{sec:undesired-eq-bounded-trajs-limit-cycles}

In this section, we study the dynamical properties of~\eqref{eq:general-system-1}, including undesired equilibria, (un)boundedness of trajectories, and limit cycles. We start by providing a precise expression for the unique optimal solution $v(\bx)$ of the problem~\eqref{eq:v-problem}. For brevity, define $\eta(\bx) := \nabla h(\bx)^T \tilde{f}(\bx) + \alpha(h(\bx))$. Then,  for any $\bx\in\mathcal{C}$ 
\begin{align}\label{eq:v-expression}
    v(\bx) = \begin{cases}
        \mathbf{0}_m, &\ \text{if} \ \eta(\bx) \geq 0, \\
        \bar{u}(\bx), &\ \text{if} \ \eta(\bx) < 0,
    \end{cases}
\end{align}
where $\bar{u}(\bx):=-\frac{\eta(\bx) G(\bx)^{-1}g(\bx)^\top\nabla h(\bx) }{\| 
 g(\bx)^\top \nabla h(\bx)  \|_{G^{-1}(\bx)}^2}$. We note that since $h$ is a strict CBF, $g(\bx)^\top \nabla h(\bx)\neq \mathbf{0}_m$ if $\eta(\bx) < 0$ and $\bx\in\Cc$. This implies that $\bar{u}$ (and hence $v$) is well defined on~$\Cc$. 

\subsection{Characterization of undesired equilibria}

Our first result leverages expression~\eqref{eq:v-expression} to provide a necessary and sufficient condition for the existence of undesired equilibria of the filtered system~\eqref{eq:general-system-1}. 

\smallskip

%
%
%
\begin{lemma}\longthmtitle{Conditions for undesired equilibria}\label{lem:undesired-eq-characterization}
  Consider system~\eqref{eq:general-system-1}. Let $h$ be a strict CBF and suppose Assumption~\ref{as: interior eq} holds. Let $\bx_* \in \mathbb{R}^n$ be such that $\tilde{f}(\bx_*) \neq \mathbf{0}_n$. Then,  $\bx_*$ is an equilibrium of \eqref{eq:general-system-1}
    if and only if  there exists $\delta<0$ such that
    \vspace{-.1cm}
\begin{subequations}
\label{eq: condition-eq} 
\begin{align}
    ~~~~~~~ & h(\bx_*)=0 ~\text{and} \label{eq: condition-eq1} & \\ 
    ~~~~~~~ & \tilde{f}(\bx_*)=\delta g(\bx_*)G(\bx_*)^{-1}g(\bx_*)^\top\nabla h(\bx_*) \, .  & ~~~~~~~~~ \label{eq: condition-eq2} 
\end{align}
\end{subequations}
Moreover, $\bx_{*}$ is an equilibrium of~\eqref{eq:general-system-1} independently of the choice of $h$ and $\alpha$.
\end{lemma}
\vspace{.1cm}
\begin{proof}  
    Note that if $x_*$ and $\delta < 0$ satisfy~\eqref{eq: condition-eq}, then by multiplying~\eqref{eq: condition-eq2} by $\nabla h(\bx_{*})^\top$ we get 
    $ 
    \|  g(\bx_*)^\top \nabla h(\bx_*)  \|_{G^{-1}(\bx_*)}^2 \delta=\nabla h(\bx_*)^\top \tilde{f}(\bx_*)$. Now, if $g(\bx_*)^\top \nabla h(\bx_*) = 0$, it would follow that $\nabla h(\bx_{*})^\top \tilde{f}(\bx_{*}) = 0$. Since $h(\bx_{*})=0$, this would imply that $\nabla h(\bx_*)^\top ( \tilde{f}(\bx_*) +  g(\bx_*)\bu) + \alpha(h(\bx_*)) = \nabla h(\bx_*)^\top ( f(\bx_*) +  g(\bx_*)\bu) + \alpha(h(\bx_*)) = 0$, for all $\bu\in\real^m$, which contradicts the assumption that $h$ is a strict CBF.
    Therefore, $g(\bx_*)^\top \nabla h(\bx_*) \neq 0$
    and
    $ 
    \delta=\frac{\nabla h(\bx_*)^\top \tilde{f}(\bx_*) }{\|  g(\bx_*)^\top \nabla h(\bx_*)  \|_{G^{-1}(\bx_*)}^2} $.
    %
    %
    Then, it follows that $\eta(\bx_*)<0$ and $\tilde{f}(\bx_*)+g(\bx_*)\bar{u}(\bx_*)= \mathbf{0}_n$. Therefore, $\bx_*$ is an equilibrium of \eqref{eq:general-system-1}. 
    
    Conversely, if $\bx_*$ is an equilibrium of \eqref{eq:general-system-1},  then $\tilde{f}(\bx_*)+g(\bx_*)\bar{u}(\bx_*)=\mathbf{0}_n$. It follows that $0 = \nabla h(\bx_*)^\top ( \tilde{f}(\bx_*)+g(\bx_*)\bar{u}(\bx_*) )= -\alpha (h(\bx_*))$. Since $\alpha(\cdot)$ is an extended class $\mathcal{K}_\infty$ function, it must hold that $ h(\bx_*)=0$. Note also that $\eta(\bx_*) < 0$, since otherwise $\bar{u}(\bx_*) = \mathbf{0}_m$ and hence $\tilde{f}(\bx_*)+g(\bx_*)\bar{u}(\bx_*) = \tilde{f}(\bx_*) = f(\bx_*) + g(\bx_*)k(\bx_*) = \mathbf{0}_n$, which can only hold if $\bx_* = \mathbf{0}_*$, contradicting Assumption~\ref{as: interior eq}.
    Hence, $\eta(\bx_*)<0$
    %
    %
    and one has that $f(\bx_*) =\frac{\eta(\bx_*)    }{\|  g(\bx_*)^\top \nabla h(\bx_*)  \|_{G^{-1}(\bx_*)}^2}  g(\bx_*)G(\bx_*)^{-1}g(\bx_*)^\top\nabla h(\bx_*)$, implying~\eqref{eq: condition-eq2} with $\delta<0$.
    %
    %
    Finally, the fact that $\bx_{*}$ is an equilibrium of~\eqref{eq:general-system-1} independently of $h$ and $\alpha$ follows from~\cite[Corollary 4.5]{indep-cbf}.
\end{proof}

\vspace{.1cm}


Hereafter, given a solution $(\bx_*,\delta_{\bx_*})$ to \eqref{eq: condition-eq}, we refer to  $\delta_{\bx_*}$ as the \emph{indicator} of $\bx_*$.
Lemma~\ref{lem:undesired-eq-characterization} characterizes the undesired equilibria of closed-loop systems obtained from safety filters. We note that related results exist in the literature:~\cite[Theorem 2]{XT-DVD:24} characterizes the undesired equilibria for a CBF-based control design that includes CLF constraints and~\cite[Proposition 5.1]{PM-JC:23-csl} characterizes undesired equilibria for a design that includes such CLF constraints as a penalty term in the objective function. 
%
%
However, both of these designs can introduce undesired equilibria in the interior of the safe set, whereas as shown in Lemma~\ref{lem:undesired-eq-characterization},~\eqref{eq:general-system-1} can only introduce undesired equilibria in the boundary of the safe set. 

Based on Lemma~\ref{lem:undesired-eq-characterization}, we define the sets
\begin{align*}
 \mathcal{E} &:=\{\bx \in \real^n:~\exists~\delta\in\mathbb{R} \ \text{s.t. $(\bx,\delta)$ solves \eqref{eq: condition-eq}} \} ,
 \\
 \hat{\mathcal{E}} &:=\{\bx\in \real^n:~\exists~\delta<0  \ \text{s.t. $(\bx,\delta)$ solves \eqref{eq: condition-eq}}\} .
\end{align*}
We refer to $\mathcal{E}$ and $ \hat{\mathcal{E}}$ as the sets of \emph{potential undesired equilibria} and of \emph{undesired equilibria} of \eqref{eq:general-system-1}, respectively. Note that $\hat{\mathcal{E}} \subset \mathcal{E}$.  By Lemma \ref{lem:undesired-eq-characterization}, determining the equilibrium points of system \eqref{eq:general-system-1} is equivalent to solving \eqref{eq: condition-eq} and checking the sign of $\delta$. 
Under appropriate conditions,~\cite[Proposition 10]{indep-cbf} provides an explicit expression for the Jacobian of $\tilde{f}(\bx) + g(\bx)v(\bx)$ at $\bx\in\hat{\Ec}$ and shows that one of its eigenvalues is $-\alpha^{\prime}(0)$ (provided that $\alpha$ is differentiable), and the rest of the eigenvalues are independent of~$\alpha$.
If $\alpha^\prime(0) > 0$, it follows that the Jacobian evaluated at  $\bx\in\hat{\mathcal{E}}$ always has a negative eigenvalue.


Even if we are able to identify the set of undesired equilibria, characterizing the dynamical properties of the closed-loop system~\eqref{eq:general-system-1} is still challenging. In the following, we provide a variety of results regarding the boundedness of trajectories, the existence of limit cycles, and the stability properties of undesired equilibria. We also investigate the possible asymptotic behaviors of trajectories of the closed-loop system~\eqref{eq:general-system-1}.

\subsection{Boundedness of trajectories}

The following result states a set of general conditions under which the trajectories of~\eqref{eq:general-system-1} are bounded.

\smallskip

\begin{proposition}\longthmtitle{Conditions for boundedness of trajectories}\label{prop:conditions-bounded-trajs}
     Consider system~\eqref{eq:general-system-1} and suppose that $h$ is a strict CBF and Assumption \ref{as: interior eq} holds.
    Let $\partial\Cc$ be bounded.
    %
    %
    Moreover, assume that the extended class $\Kc_{\infty}$ function $\alpha(\cdot)$ is linear, i.e., $\alpha(z) = a z$ for some $a > 0$. Then, for any compact set $\Phi \subset \real^n$ 
    with $\partial\Phi\subset\Cc$, there exists $\tilde{a}_{\Phi}>0$ and a compact set $\tilde{\Phi}$ containing $\Phi$ such that, by taking $a >\tilde{a}_{\Phi}$, 
    $\tilde{\Phi}$ is forward invariant under~\eqref{eq:general-system-1}.
    As a consequence, any trajectory of~\eqref{eq:general-system-1} with initial condition in $\Phi\cap\Cc$ is bounded (because it remains in $\tilde{\Phi}$ at all times).
\end{proposition}
%
%
%
%
\begin{proof}
    Since $\partial\Cc$ is bounded, either $\Cc$ is bounded or $\real^n\backslash\Cc$ is bounded.
    
    Case~1: $\Cc$ bounded. 
    %
    %
    In this case, $\Cc$ is actually compact, and the result holds because $\Cc$ is forward invariant under~\eqref{eq:general-system-1}: for any compact $\Phi$ with $\partial\Phi\subset\Cc$,  necessarily $\Phi \subset \Cc$ and we can take $\tilde{\Phi}=\Cc$ and any extended class $\Kc_{\infty}$ function.

    Case~2: $\real^n\backslash\Cc$ is bounded. 
    Since the system $\dot{\bx}=f(\bx)+g(\bx)k(\bx)$ renders the origin  globally asymptotically stable, by~\cite[Theorem 4.17]{HK:02}, there exists a radially unbounded Lyapunov function $V:\real^n\to\real_{\geq0}$.
    Let $\Gamma$ be a Lyapunov sublevel set of $V$ containing $(\real^n\backslash\Cc) \cup \Phi$ and such that $h(\bx)>0$ for all $\bx\in\partial\Gamma$. Note that such $\Gamma$ exists because $\real^n\backslash\Cc$ is bounded and $V$ is radially unbounded.
    Now, since $h(\bx)>0$ for all $\bx\in\partial\Gamma$, there exists 
    $\tilde{a}_\Phi$ sufficiently large such that $\eta(\bx)=\nabla h(\bx)^T (f(\bx) + g(\bx)k(\bx)) + \tilde{a}_\Phi h(\bx) \geq 0$ for all $\bx\in\partial\Gamma$.
    Let $\tilde{\Phi}=\Gamma$.
    Since $\Gamma$ is a Lyapunov sublevel set, this implies that $\Gamma$ is forward invariant for $\dot{\bx} = f(\bx) + g(\bx)k(\bx)$.
    By taking $\alpha(z) = a z$ with $a > \tilde{a}_{\Phi}$, and
    since $\nabla h(\bx)^T (f(\bx) + g(\bx)k(\bx)) + \alpha (h(\bx)) \geq 0$ for all $\bx\in\partial\Gamma$,
    the safety filter is inactive in $\partial\Gamma$ and therefore $\Gamma$ is also forward invariant under~\eqref{eq:general-system-1}.
\end{proof}

Next we show that if the assumptions of Proposition~\ref{prop:conditions-bounded-trajs} do not hold, the trajectories of~\eqref{eq:general-system-1} might not be bounded (as illustrated in Figure~\ref{fig:4_figures_v3}(a)). The following result provides technical conditions under which for linear systems and affine CBFs,~\eqref{eq:general-system-1} has unbounded trajectories.

\smallskip
\begin{proposition}\longthmtitle{Unbounded trajectories}\label{prop:unbounded-trajs-unsafe-sets}
    Let $A\in\real^{n\times n}$, $B\in\real^{n\times m}$, $\ba\in\real^n$, $b\in\real$ and consider the LTI system $\dot{\bx}=A\bx+B\bu$, the function
    $h:\real^n\to\real$ given by $h(\bx) = \ba^T \bx - b$ and the set $\Cc := \setdef{\bx\in\real^n}{h(\bx)\geq0}$.
    Let $h$ be a strict CBF and suppose Assumption \ref{as: interior eq} holds.
    Further assume that there exists $\bc\in\mathbb{R}^n$, $\zeta_1 > 0$, and $\zeta_2\geq0$ satisfying $\bc\neq \mathbf{0}_n$, $\bc^T B=\mathbf{0}_m^T$, and $\bc^T A=\zeta_1 \bc^T +\zeta_2 \ba^T$.  Then, for any locally Lipschitz controller $\hat{\bu}:\real^n\to\real^m$, and any initial condition $\bx_0$ in $\setdef{\bx\in\real^n}{\ba^T \bx\geq b,~\zeta_1 \bc^T \bx + \zeta_2 b > 0}$, the solution of $\dot{\bx} = A \bx + B \hat{\bu}(\bx)$ with initial condition at $\bx_0$ satisfies $\lim_{t\to+\infty} \|\bx(t;\bx_0)\|=+\infty$.
\end{proposition}
\begin{proof}
We note that
\begin{align*}
    \frac{d}{dt}(\bc^T \bx) \! = \! \bc^T (A\bx \! + \! B\hat{\bu}(\bx)) \! = \! (\zeta_1 \bc \! + \! \zeta_2 \ba)^T \bx \geq \zeta_1 \bc^T \bx \! + \! \zeta_2 b \, .
\end{align*}
This implies that the set $\setdef{\bx\in\real^n}{\zeta_1 \bc^T \bx+\zeta_2 b \geq 0}$ is forward invariant. Additionally, if
$\zeta_1 \bc^T \bx_0 + \zeta_2 b > 0$,
then $\zeta_1 \bc^T \bx(t;\bx_0) + \zeta_2 b \geq \zeta_1 \bc^T \bx_0 + \zeta_2 b$
and therefore $\bc^T \bx(t;\bx_0) \geq \bc^T \bx_0$ for all $t\geq 0$.
%
%
It follows that
\begin{align*}
    \bc^T \bx(t;\bx_0)
    &= \bc^T \bx_0 + \int_0^{+\infty}  \frac{d}{dt} \bc^T \bx(t;\bx_0) ~ dt  \geq  \bc^T \bx_0 +\int_0^{+\infty}  (\zeta_1 \bc^T \bx_0 + \zeta_2 b) dt =+\infty,
\end{align*}
which implies that $\lim_{t\to+\infty} \|\bx(t;\bx_0)\|=+\infty$.
\end{proof}
%
%


\smallskip

\begin{remark}\longthmtitle{Underactuated systems always have unbounded solutions for some safe set}\label{rem:underactuated-unbounded-trajs}
    {\rm In the setting of Proposition~\ref{prop:unbounded-trajs-unsafe-sets}, if $m<n$, $\ker(B)\neq0$ and therefore there exists $\bc\neq\mathbf{0}_n$ such that $\bc^T B = \mathbf{0}_m^T$. Then, for any  $\zeta_1>0$ and $\zeta_2> 0$, by letting $\ba=\frac{1}{\zeta_2}(A^T \bc - \zeta_1 \bc) $ we satisfy the hypothesis of Proposition~\ref{prop:unbounded-trajs-unsafe-sets}.
    %
    %
    Hence, for any underactuated linear system, there exists a safe set for which any controller induces unbounded solutions.
    On the other hand, for fully actuated systems, 
    there does not exist $\bc\in\real^n$ satisfying $\bc^T B \neq \mathbf{0}_m^T$ and therefore the conditions of Proposition~\ref{prop:unbounded-trajs-unsafe-sets} are never met.
    \hfill $\Box$}
\end{remark}


Proposition~\ref{prop:unbounded-trajs-unsafe-sets} and Remark~\ref{rem:underactuated-unbounded-trajs} show that in general, one cannot guarantee that the trajectories of~\eqref{eq:general-system-1} are bounded. 

\subsection{Limit cycles}

Here we turn our attention to limit cycles. The following result ensures that, by taking a linear extended class $\Kc_{\infty}$ function $\alpha(x) = a x$, $a>0$, with sufficiently large slope, closed-loop planar systems~\eqref{eq:general-system-1} do not have limit cycles.

\smallskip

\begin{proposition}\longthmtitle{No limit cycles in planar systems}\label{prop:no-limit-cycles-general}
    Consider system~\eqref{eq:general-system-1} with $n=2$. Let $h$ be a strict CBF with a linear extended class $\Kc_{\infty}$ function $\alpha(\cdot)$, i.e., $\alpha(z) = a z$ for some $a > 0$. Suppose Assumption \ref{as: interior eq} holds and that $\real^2\backslash\Cc$ is comprised of a finite number of connected components.
    Then, there exists $\hat{a}>0$ sufficiently large such that, by taking $a>\hat{a}$, the closed-loop system does not contain any limit cycles in~$\Cc$. Moreover, all bounded trajectories with initial condition in~$\Cc$
    %
    %
    converge to an equilibrium point.
\end{proposition}
\begin{proof}
     We start by noting that, if there exists a limit cycle with at least one point in $\Cc$, then it is contained in $\Cc$ (because the set is forward invariant under~\eqref{eq:general-system-1}).
    Let $\gamma$ be a limit cycle contained in $\Cc$. Note that $\gamma$ corresponds to a maximal trajectory of~\eqref{eq:general-system-1} and cannot contain an equilibrium point.
    We distinguish four different cases:
    \begin{enumerate}
        \item\label{it:limit-cycle-proof-first} If $\gamma$ does not encircle the origin and does not encircle any connected component of $\real^n\backslash\Cc$,
        we reach a contradiction because limit cycles must encircle an equilibrium point~\cite[Corollary 6.26]{JM-book:07}, and all equilibrium points other than the origin of~\eqref{eq:general-system-1} are in $\partial\Cc$ by Lemma~\ref{lem:undesired-eq-characterization}.
        %
        %
        \item\label{it:limit-cycle-proof-second} Suppose that $\gamma$ encircles the origin, but does not encircle any connected component of $\real^2\backslash\Cc$.
        Note that since the origin is globally asymptotically stable for the nominal system, and the origin is contained in $\text{Int}(\Cc)$, there exists $\bq_1\in\partial\Cc$ such that the solution $\by(t;\bq_1)$ of the nominal system satisfies $\by(t;\bq_1)\in \text{Int}(\Cc)$ for all $t>0$. 
        Since $\bq_1\in\partial\Cc$, $\nabla h(\bx)^T (f(\bx)+g(\bx)k(\bx))\rvert_{\bx=\bq_1} > 0$ and since $\nabla h$, $f$, $g$ and $k$ are continuous there exists a neighborhood $\Nc_{\bq_1}$ of $\bq_1$ such that $\nabla h(\bx)^T (f(\bx)+g(\bx)k(\bx))\rvert_{\bx=\bz} > 0$ for all $\bz\in\Nc_{\bq_1}$.
        Now, note that there exists $d_1>0$ such that $h( \by(t;\bq_1) ) > d_1$ for all $t>0$ such that $\by(t;\bq_1)\notin \Nc_{\bq_1}$.
        This means that there exists $a_{1}>0$ such that $\nabla h(\by(t;\bq_1))^T (f(\by(t;\bq_1))+g(\by(t;\bq_1))k(\by(t;\bq_1))) + a_{1} h(\by(t;\bq_1)) > 0$ for all $t\geq0$. Therefore, $\by(t;\bq_1)$ is also a trajectory of~\eqref{eq:general-system-1} by taking $\alpha$ to be an extended class $\Kc_{\infty}$ function with slope greater than $a_{1}$, which means that it intersects with~$\gamma$, violating the existence and uniqueness of solutions of~\eqref{eq:general-system-1}.
        Hence, by taking $a>a_1$, we ensure such $\gamma$ cannot exist.
        \item\label{it:limit-cycle-proof-third} Suppose that $\gamma$ encircles one or more connected components in $\real^2\backslash\Cc$, but not the origin, cf. Figure~\ref{fig:limit-cycle-proof-sketch}.
        Let $S$ (resp., $\bar{S}$) be the union of the connected components encircled (resp., not encircled) by $\gamma$.
        Since the origin is globally asymptotically stable for the nominal system, there exists $\bq_2$ in the boundary of $S$ so that the solution $\by(t;\bq_2)$ of the nominal system 
        satisfies one of the following:
        \begin{enumerate}
            \item there exists $t_1<0$ with $\tilde{\bq}_2=\by(t_1;\bq_2)\in \partial\bar{S}$ and $\by(t;\bq_2)\in\text{Int}(\Cc)$ for all $t\in(t_1,0)$. In this case, there exists $t_1^\prime$ such that $\by(t_1^\prime;\bq_2)\in\gamma$, and
            there exists $a_{S,2}>0$ such that 
            $\{ \by(t;\bq_2) \}_{t\in[t_1,t_1^\prime]}$ is a trajectory of~\eqref{eq:general-system-1} by taking a linear extended class $\Kc_{\infty}$ function with slope greater than $a_{S,2}$.
            This violates the existence and uniqueness of solutions of~\eqref{eq:general-system-1}, because $\{ \by(t;\bq_2) \}_{t\in[t_1,t_1^\prime]}$ is a solution and intersects $\gamma$. 
            \item $\by(t;\bq_2)\in \Cc$ for all $t<0$. 
            By Lemma~\ref{lem:trajectories-of-gas-system}, $\lim\limits_{t\to-\infty} \norm{ \by(t;\bq_2) } = \infty$. This means that there exists
            $\bar{t}_1<0$ such that
            $\by(\bar{t}_1;\bq_2)\in\gamma$;
            %
            %
            %
            This means that 
            there exists $\tilde{a}_{S,2}>0$ such that 
            $\{ \by(t;\bq_2) \}_{t\in[\bar{t}_1,\frac{\bar{t}_1}{2}] }$ is a trajectory of~\eqref{eq:general-system-1} by taking a linear extended class $\Kc_{\infty}$ function with slope greater than $\tilde{a}_{S,2}$.
            However, this solution intersects with $\gamma$, violating the existence and uniqueness of solutions of~\eqref{eq:general-system-1}.
        \end{enumerate} 
        %
        Hence by taking a linear extended class $\Kc_{\infty}$ function with slope greater than $a_{S,2}$ and $\tilde{a}_{S,2}$, we ensure such $\gamma$ cannot exist.
        \item\label{it:limit-cycle-proof-fourth} 
        Suppose that $\gamma$ encircles one or more connected components in $\real^2\backslash\Cc$ and the origin. Let $S^\prime$ (resp., $\bar{S}^\prime$) be the subset of the connected components of $\real^2\backslash\Cc$ encircled (resp., not encircled) by~$\gamma$. Again, since the origin is globally asymptotically stable for the nominal system, there exists $\bq_3$ in the boundary of $S^{\prime}$ so that the solution $\by(t;\bq_3)$ of the nominal system satisfies one of the following:
        \begin{enumerate}
            \item there exists $t_2 < 0$ with $\by(t_2;\bq_3)\in \partial\bar{S}^\prime$ and $\by(t;\bq_3)\in\text{Int}(\Cc)$ for all $t\in(t_2,0)$;
            \item $\by(t;\bq_3)\in\Cc$ for all $t < 0$.
        \end{enumerate}
        By following an argument analogous to case~(iii), there exists $\breve{a}_{S^{\prime}}>0$ sufficiently large such by taking a linear extended class $\Kc_{\infty}$ function with slope greater than $\breve{a}_{S^{\prime}}$, we ensure such $\gamma$ cannot exist.
        %
        %
    \end{enumerate}

    Note that the values of 
    $a_{S,2}$ and $\tilde{a}_{S,2}$ defined in~\ref{it:limit-cycle-proof-third}
    depend on the set of connected components of $\real^2\backslash\Cc$ encircled by the limit cycle. Since there is a finite number of bounded connected components, there exists $a^*$ such that $a^* > a_{S,2}$ and $a^* > \tilde{a}_{S,2}$ for all possible sets of connected components $S$ of $\real^2\backslash\Cc$.
    %
    %
    Similarly, the value $\breve{a}_{S^\prime}$ defined in~\ref{it:limit-cycle-proof-fourth} depends on $S^\prime$, but there exists $\breve{a}^*$ such that $\breve{a}^* > \breve{a}_{S^{\prime}}$ for all possible $S^{\prime}$.
    By taking $a>\hat{a}:=\max\{ a^*, \breve{a}^* \}$, it follows that the closed-loop system does not contain any limit cycles in~$\Cc$. Finally, since no limit cycles exist in $\Cc$ for $a>\hat{a}$,
    %
    %
    by the Poincar\'e-Bendixson Theorem~\cite[Chapter 7, Thm. 4.1]{PH:02} all bounded trajectories with initial condition in  $\Cc$ converge to an equilibrium point.
\end{proof}
%
%

\begin{figure*}
  \centering
    \subfigure[Subcase (a) of item~\ref{it:limit-cycle-proof-third}]{\includegraphics[width=.48\linewidth]{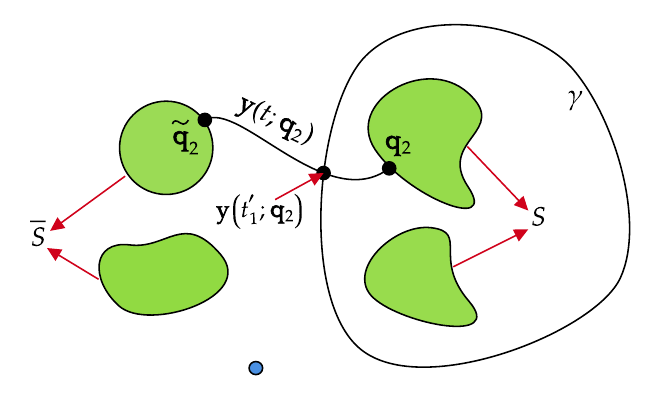}}
   \subfigure[Subcase (b) of item~\ref{it:limit-cycle-proof-third}]{\includegraphics[width=.48\linewidth]{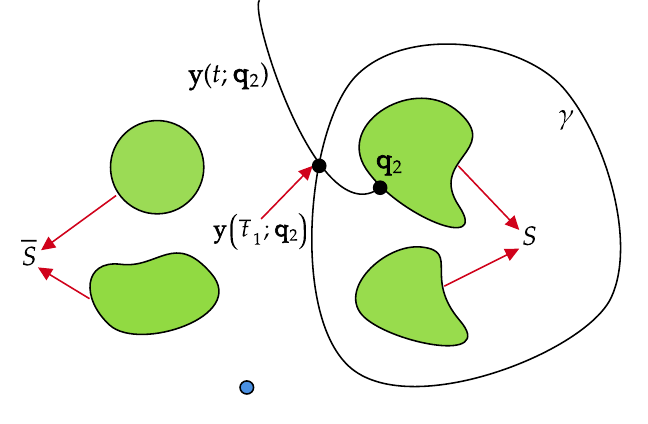}}
  \caption{Sketch of the setting considered in item~\ref{it:limit-cycle-proof-third} of the proof of Proposition~\ref{prop:no-limit-cycles-general}. The connected components comprising $\real^2\backslash\Cc$ are depicted in green, whereas the origin is represented by the blue dot.}
  \label{fig:limit-cycle-proof-sketch}
\end{figure*}

The following example shows that limit cycles can exist for systems of the form~\eqref{eq:general-system-1} in dimension $n\geq 3$.
%
%

\smallskip

\begin{example}\longthmtitle{Existence of limit cycles in higher dimensions}\label{ex:existence-limit-cycle-higher-dimensions}
{\rm
Consider the safe set $\Cc:=\setdef{\bx\in\real^3}{ \|\bx-[x_c,0,0]^T \|^2-r^2\geq 0 }$.
Let $B\in\real^{3\times 3}$ invertible, 
$G(\bx)=B^\top B$, $K=\zero_{3}$ and nominal controller $k(\bx) = K\bx \equiv\mathbf{0}_3$ .
Next, let $0<p_1<\frac{r}{r+x_c}p_2$, $p_2>0$, $p_3>0$, and define
\begin{align*}
    A:=\begin{bmatrix}
        -p_1 & 0 &0 \\
        0 & -p_2  & p_3\\
        0 & -p_3 & -p_2
    \end{bmatrix}.
\end{align*}
Consider the closed-loop system~\eqref{eq:general-system-1} obtained with $f(\bx) = A\bx$, $g(\bx) = B$, $h(\bx) = \|\bx-[x_c,0,0]^T \|^2-r^2$ and $k(\bx) \equiv 0$.

Define also $\hat{p}:=\frac{x_c p_1}{p_2-p_1}$ and $\hat{q}:=\sqrt{r^2-\hat{p}^2}$, then $0<\hat{p}<r$ and $p_1=\frac{\hat{p} }{\hat{p}+x_c}p_2$. Consider
\begin{align*}
    \bx_0 := \begin{bmatrix}
        \hat{p} \\
        \hat{q} \\
        \hat{q}
    \end{bmatrix}, \quad
   \hat{\bx}(t;\bx_0):=\begin{bmatrix}
        \hat{x}_1(t) \\ \hat{x}_2(t) \\ \hat{x}_3(t)
    \end{bmatrix}=\begin{bmatrix}
        x_c+\hat{p} \\ \hat{q} \sin (p_3 t)\\
        \hat{q} \cos (p_3 t)
    \end{bmatrix}.
\end{align*}

\begin{figure}[htb]
  \centering
  {\includegraphics[width=0.7\linewidth]{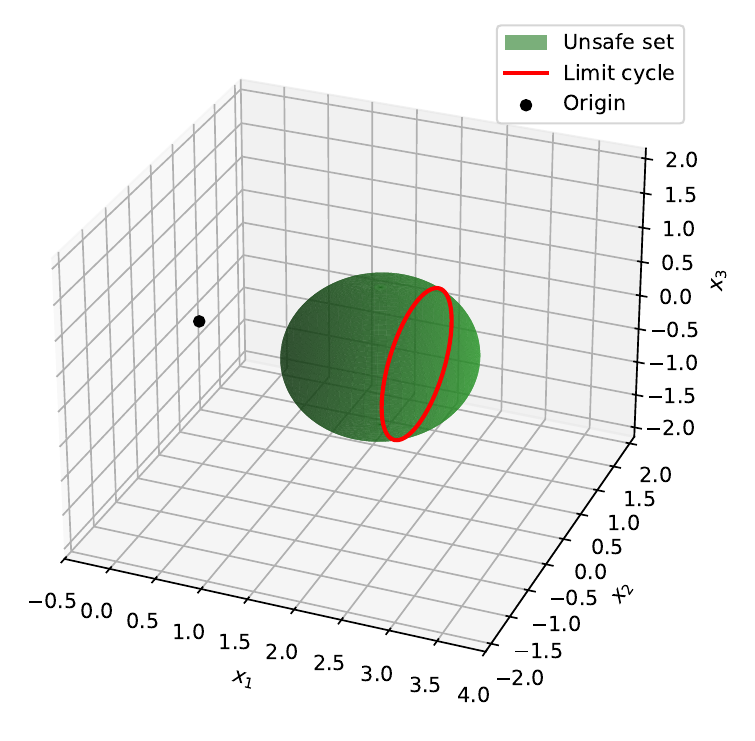}} 
  \caption{Depiction of the setting considered in Example~\ref{ex:existence-limit-cycle-higher-dimensions}, with parameters $x_c = 2$, $r=1$, $p_1 = 1$, $p_2=6$, $p_3=1$. The obstacle is depicted in green, the limit cycle $\hat{\bx}$ in red, and the origin in black.}
  \label{fig:limit-cycle-3d}
\end{figure}

Note that for any $t$,
\begin{align*}
    &\nabla h(\hat{\bx}(t;\bx_0))^T (A+BK)\hat{\bx}(t;\bx_0)= -2 \! \! \begin{bmatrix}
        \hat{p} \\
        \hat{q} \sin (p_3 t) \\
        \hat{q} \cos (p_3 t)
    \end{bmatrix}^T \! \!
    \begin{bmatrix}
        \frac{\hat{p}}{\hat{p}+x_c}p_2 \! & 0 \! &0 \\
        0 \! & \! \! p_2 \!  & \! \!  \ \ -p_3\\
        0 \! & \! \! p_3 \! & \! \! p_2
    \end{bmatrix} \! \hat{\bx}(t;\bx_0) \! = \! -2 r^2 p_2 \! < \! 0.
\end{align*}
Moreover, let 
\begin{align*}
    \kappa(t) = \frac{\nabla h(\hat{\bx}(t;\bx_0))^T (A+BK)\hat{\bx}(t;\bx_0)+\alpha(h(\hat{\bx}(t;\bx_0)))}{\nabla h(\hat{\bx}(t;\bx_0))^T B G^{-1} B^T \nabla h(\hat{\bx}(t;\bx_0))}
\end{align*}
and note that since $K=\zero_3$,
\begin{align*}
    &(A+BK)\hat{\bx}(t;\bx_0)-\kappa(t) B G^{-1} B^T \nabla h(\hat{\bx}(t;\bx_0))\\
    =&A\hat{\bx}(t;\bx_0)-\frac{\nabla h(\hat{\bx}(t;\bx_0))^T A\hat{\bx}(t;\bx_0) }{\nabla h(\hat{\bx}(t;\bx_0))^T \nabla h(\hat{\bx}(t;\bx_0))}  \nabla h(\hat{\bx}(t;\bx_0))\\
    =&\begin{bmatrix}
         -\hat{p} p_2 \\ -  p_2\hat{q} \sin (p_3 t) + p_3 \hat{q} \cos (p_3 t) \\
        -p_2 \hat{q} \sin (p_3 t) - p_3\hat{q} \cos (p_3 t) 
    \end{bmatrix}+p_2\begin{bmatrix}
       \hat{p}_1 \\ \hat{q} \sin (p_3 t)\\
        \hat{q} \cos (p_3 t)
    \end{bmatrix}\\
    =&\begin{bmatrix}
      0  \\ p_3\hat{q} \cos (p_3 t)\\
        -p_3 \hat{q} \sin (p_3 t)
    \end{bmatrix}=\frac{d}{dt}\begin{bmatrix}
        x_c+\hat{p} \\ \hat{q} \sin (p_3 t)\\
        \hat{q} \cos (p_3 t)
    \end{bmatrix} =\frac{d}{dt}\begin{bmatrix}
        \hat{x}_1(t) \\ \hat{x}_2(t) \\ \hat{x}_3(t)
    \end{bmatrix}
\end{align*}
Hence $\hat{\bx}(t;\bx_0)$ is a valid trajectory of the closed-loop system~\eqref{eq:general-system-1} and it is a limit cycle.
Note also that since $\hat{\bx}(t)\in\partial\Cc$ for all $t\geq0$, 
by~\cite[Corollary 4.5]{indep-cbf} this trajectory exists for any choice of $h$ and $\alpha$.
Figure~\ref{fig:limit-cycle-3d} depicts the limit cycle $\hat{\bx}$.
\problemfinal
}
\end{example}

\subsection{Structure of the set of undesired equilibria}
In this section, we investigate the set of undesired equilibria of~\eqref{eq:general-system-1}. The following example shows that in general, this set can be a continuum.
%
%


\smallskip

\begin{figure}[htb]
  \centering
  {\includegraphics[width=0.7\textwidth]{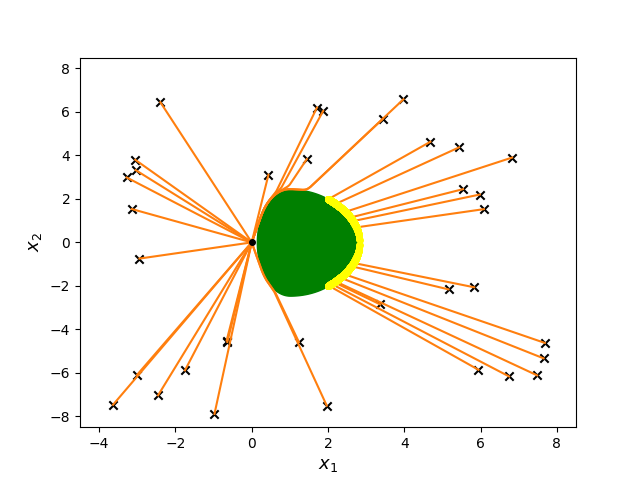}} 
  \caption{Plot of different trajectories of~\eqref{eq:v-linear-system}. Trajectories are depicted in orange. The unsafe set is colored in green. Black crosses denote initial conditions, the black dot denotes the origin, and the yellow region denotes a continuum of undesired equilibria.}
  \label{fig:continuum-eq-underactuated}
\end{figure}

\smallskip
\begin{example}\longthmtitle{Continuum of undesired equilibria}\label{ex:continuum-eq}
{\rm
Let $n=2$, $m=1$ and consider~\eqref{eq:general-system-1} with
\begin{align*}
    f(\bx) = \begin{pmatrix}
        0 & -1 \\
        1 & -2
    \end{pmatrix} \bx, 
    \ \
    g(\bx) = B = \begin{pmatrix}
        0 \\
        1
    \end{pmatrix}, \ \ 
    h(\bx)=\begin{cases}
        (x_1+1)^2+x_2, \quad x_1\geq -1,\\
        x_2, \quad \quad \qquad \qquad -2<x_1<-1,\\
        (x_1+2)^2+x_2, \quad x_1\leq -2,
    \end{cases}
\end{align*}
and $\alpha(s)=10s$.
Note that $h$ is continuously differentiable. Moreover,
\begin{align*}
    &\nabla h(\bx)^T B
    =\begin{cases}
       2\begin{bmatrix}
           x_1+1 & 1
       \end{bmatrix}\begin{bmatrix}
         0 \\1
       \end{bmatrix}\neq 0, \quad x_1\geq -1, \\
        2\begin{bmatrix}
           0 & 1
       \end{bmatrix} \begin{bmatrix}
         0 \\1
       \end{bmatrix}\neq 0, \quad \qquad -2<x_1<-1, \\
         2\begin{bmatrix}
           x_1+2 & 1
       \end{bmatrix}\begin{bmatrix}
         0 \\1
       \end{bmatrix}\neq 0, \quad x_1\leq -2,
    \end{cases}
\end{align*}
Therefore, $h$ is a strict CBF.
Next, we show that the set $\{t(1,0): -2\leq t\leq -1 \}$ is contained in the set of undesired equilibria for any linear stabilizing controller $k(\bx)=k_1 x_1 + k_2 x_2$, $G:\real^2\to\real$ and extended class $\mathcal{K}_{\infty}$ function $\alpha$.
Since $k$ is a stabilizing controller, it follows that $f(\bx) + Bk(\bx)$ is Hurwitz and therefore $-2+k_2<0$, $1+k_1>0$.
For any $\sigma \in[-2,-1]$, the point $\bx_{\sigma} = (\sigma,0)\in\real^2$ satisfies~\eqref{eq: condition-eq} with associated indicator equal to $\frac{G(\bx_{\sigma})}{(1+k_1)\sigma}<0$. Hence, the set $\{\sigma(1,0): -2\leq \sigma\leq -1 \}$ is contained in the set of undesired equilibria for any linear stabilizing controller $k(\bx)=k_1 x_1 + k_2 x_2$, $G:\real^2\to\real$ and extended class $\mathcal{K}_{\infty}$ function $\alpha$. Figure~\ref{fig:continuum-eq-underactuated} shows some of the trajectories for the corresponding closed-loop system~\eqref{eq:general-system-1}. Since the undesired equilibria are not isolated, the study of their stability properties requires using the notion of \textit{semistability}~\cite{SPB-DSB:03}.
} \problemfinal
\end{example}

\smallskip

The following result provides conditions under which a continuum of undesired equilibria of~\eqref{eq:general-system-1} does not exist.

\vspace{.1cm}

\begin{lemma}\longthmtitle{Sufficient conditions for isolated equilibria}\label{lem:suff-cond-isolated-eq}
    An undesired equilibrium~$\bx_*$ is isolated if the Jacobian of~\eqref{eq:general-system-1} evaluated at~$\bx_*$ does not have imaginary eigenvalues. 
    If $\partial\Cc$ is bounded, then each undesired equilibria of~\eqref{eq:general-system-1} is isolated if and only if 
    $|\hat{\Ec}| < \infty$.
\end{lemma}
\begin{proof}
   Given an undesired equilibrium~$\bx_*$, if the Jacobian of~\eqref{eq:general-system-1} evaluated at $\bx_*$ does not have imaginary eigenvalue, then there exists a neighborhood of $\bx_*$ such that the linearization of~\eqref{eq:general-system-1} around $\bx_*$ does not contain any equilibrium point other than itself. By the Hartman-Grobman Theorem~\cite[Section 2.8]{LP:00}, there also exists a neighborhood of $\bx_*$ for which~\eqref{eq:general-system-1} does not contain any undesired equilibrium and hence $\bx_*$ is isolated.

   Consider the case when $\partial\Cc$ is bounded. Clearly, if $|\hat{\Ec}|<\infty$ (i.e., the number of undesired equilibria is finite), then each of the undesired equilibria is isolated. Conversely,
   if the number of undesired equilibria is infinite, consider an infinite sequence of undesired equilibria $\{\bx_{*,i}\}_{i=1}^{+\infty}$. Since $\partial\Cc$ is compact, there exists a convergent subsequence $\{\bx_{*,i_k}\}_{k=1}^{+\infty}$ such that $\lim_{k\to\infty} \bx_{*,i_k}=\bq_* $. 
    Since~\eqref{eq:general-system-1} is continuous under the assumption that $h$ is a strict CBF (cf.~\cite[Lemma III.2]{MA-NA-JC:25-tac}), $0 = \lim_{k\to\infty} f(\bx_{*,i_k}) + g(\bx_{*,i_k})v(\bx_{*,i_k}) = f(\bx_*) + g(\bq_*) v(\bx_*)$ and hence $\bx_*$ is an equilibrium, which is non-isolated.
\end{proof}




While Lemma~\ref{lem:suff-cond-isolated-eq} provides sufficient conditions for the existence of isolated undesired equilibria, finding their number is, in general, challenging. However, this is possible for the special case of planar systems, as we show next.

\smallskip
\begin{proposition}\longthmtitle{Number and stability properties of undesired equilibria for a planar system with bounded obstacle}\label{prop:number-of-undes-equilibria-stab-properties-finite-set-of-bounded-obstacles}
    Let $h$ be a strict CBF and suppose Assumption~\ref{as: interior eq} holds. Let  
    $\real^2\backslash\Cc$ be a bounded connected set. Consider~\eqref{eq:general-system-1} with $n=2$ and assume its undesired equilibria $\hat{\Ec} = \{ \bx_*^{(i)} \}_{i=1}^{k} \subset \partial\Cc$
    %
    %
    are either asymptotically stable or saddle points. Then, $k$ is odd, and $\frac{k+1}{2}$ equilibria are saddle points and $\frac{k-1}{2}$ are asymptotically stable.
    %
    %
\end{proposition}
\begin{proof}
%
%
Let $\Lc \subset [k]$ be the index set of undesired equilibria that are saddle points. We show that $|\Lc|=\frac{k+1}{2}$. Figure~\ref{fig:number-of-equilibria-bounded-obstacle-proof-sketch} serves as visual aid for the different elements employed in the proof.
\begin{figure}[htb]
  \centering    
  {\includegraphics[width=0.7\textwidth]{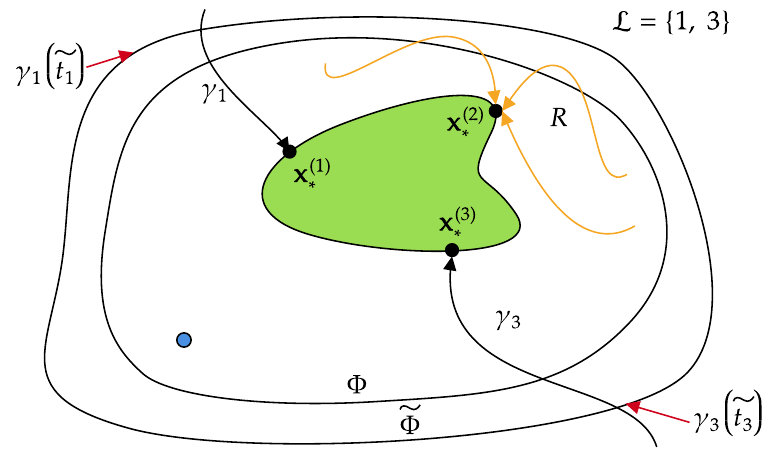}} 
\caption{Sketch of the setting considered in the proof of Proposition~\ref{prop:number-of-undes-equilibria-stab-properties-finite-set-of-bounded-obstacles}. The unsafe set is depicted in green, whereas the origin is represented by the blue dot.}\label{fig:number-of-equilibria-bounded-obstacle-proof-sketch}
\end{figure}
%
%
   Let $\Phi$ be a compact set containing
   the origin and $\real^2\backslash\Cc$ in its interior. This implies that $\partial\Phi\subset\Cc$.
   By Propositions~\ref{prop:conditions-bounded-trajs}
   and~\ref{prop:no-limit-cycles-general}, there exists $a_{\Phi}>0$ and a compact set $\tilde{\Phi}$ containing $\Phi$
   %
   %
   such that~\eqref{eq:general-system-1}, with extended class $\Kc_{\infty}$ function with slope greater than $a_{\Phi}$, makes $\tilde{\Phi}$ forward invariant and $\tilde{\Phi}\cap\Cc$ does not contain any limit cycles.
   Since~\cite[Proposition 11]{indep-cbf} ensures that the stability properties of undesired equilibria are independent of $\alpha$, we can assume without loss of generality that $\alpha$ takes this form.
   %
   %
    Now, for $j\in\Lc$, let $\gamma_j$ be a subset of the one-dimensional local stable manifold of~\eqref{eq:general-system-1} at $\bx_*^{(j)}$ 
    such that:
    \begin{enumerate}
        \item $\gamma_j$ corresponds to a maximal trajectory of~\eqref{eq:general-system-1};
        \item there exists $T>0$ with $\gamma_j(t)\in\text{Int}(\Cc)$ for all $t>T$.
    \end{enumerate}
    Note that such $\gamma_j$ always exists because the stable manifold is a union of trajectories and the stable manifold of~\eqref{eq:general-system-1} at $\bx_*^{(j)}$ is tangent to $\nabla h(\bx_*^{(j)})$.
    The fact that the stable manifold of~\eqref{eq:general-system-1} at $\bx_*^{(j)}$ is tangent to $\nabla h(\bx_*^{(j)})$ follows from the fact that $\nabla h(\bx_*^{(j)})$ is an eigenvector of the Jacobian of~\eqref{eq:general-system-1} evaluated at $\bx_*^{(j)}$ with negative eigenvalue (cf.~\cite[Proposition 6.2]{YC-PM-JC-EDA:24-auto}), and the Stable Manifold Theorem~\cite[Section 2]{LP:00}.
    %
    %
    By Lemma~\ref{lem:existence-stable-manifold-exiting-set-no-limit-cycles}, 
    since $\tilde{\Phi}$ is forward invariant and $\tilde{\Phi}\cap\Cc$ does not contain limit cycles,
    there exists at least one $j\in\Lc$ and $\tilde{t}_j\in\real$ such that
    $\gamma_j(\tilde{t}_j) \in\partial \tilde{\Phi}$ and $\gamma_j(t)\notin \tilde{\Phi}$ for all $t < t_j$.
    
    Moreover by Lemma~\ref{lem:stable-manif-convergence-not-tangent}(i), $\gamma_j$ cannot be tangent to $\Cc$ for any $j\in\Lc$.
    Therefore $\{ \gamma_j \}_{j\in\Lc}$ 
    divide $\tilde{\Phi}\cap\Cc$
    %
    %
    into $|\Lc|$ regions (note that this would not be the case if there was not at least one $j\in\Lc$ and $\tilde{t}_j\in\real$ such that
    $\gamma_j(\tilde{t}_j) \in\partial \tilde{\Phi}$ and $\gamma_j(t)\notin \tilde{\Phi}$ for all $t < t_j$), each of which is an open connected set. Let $R$ be any one of such regions.
    We now show that $\text{Cl}(R)$ contains exactly one asymptotically stable equilibrium point.
    
    Indeed, first suppose that there is no asymptotically stable equilibrium point in~$\text{Cl}(R)$.
    Let $\breve{\gamma}$ be any trajectory with initial condition in $\text{Cl}(R)$. By the Poincar\'e-Bendixson Theorem~\cite[Chapter 7, Thm. 4.1]{PH:02}, since $\tilde{\Phi}$ is forward invariant, $\breve{\gamma}$ converges to either a limit cycle or an equilibrium point. 
    It cannot converge to a limit cycle because $\tilde{\Phi}\cap\Cc$ does not contain any limit cycles.
    %
    %
    Moreover, if $\bar{\bx}_*$ and $\tilde{\bx}_*$ are the equilibrium points
    %
    %
    in $\{ \bx_*^{(j)} \}_{j=1}^k$ whose stable manifolds define the boundary of $R$, $\breve{\gamma}$ cannot converge to an equilibrium other than $\bar{\bx}_*$ or $\tilde{\bx}_*$, since otherwise $\breve{\gamma}$ would intersect the stable manifolds of $\bar{\bx}_*$ or $\tilde{\bx}_*$, which would contradict the uniqueness of solutions of~\eqref{eq:general-system-1}. But  it can also not converge to $\bar{\bx}_*$ or $\tilde{\bx}_*$: if, for example, $\breve{\gamma}$ converged to $\bar{\bx}_*$, there would be two different trajectories converging to $\bar{\bx}_*$ with different tangent vectors at~$\bar{\bx}_*$, which would contradict the fact that $\bar{\bx}_*$ is a saddle point.
    Therefore, $\text{Cl}(R)$ contains at least one asymptotically stable equilibrium.
%
%

%
%
    Next suppose that there are multiple asymptotically stable undesired equilibria in $\text{Cl}(R)$.
    By~\cite[Theorem 8.1]{HK:02}, the boundary of 
    the regions of attraction of asymptotically stable equilibria is formed by trajectories. Let $\tilde{\gamma}$ be one such trajectory with initial condition in $R$. 
    Note that since $R$ is contained in $\tilde{\Phi}$ and $\tilde{\Phi}$ is forward invariant, $\tilde{\gamma}$ is bounded. 
    %
    %
    By the Poincar\'e-Bendixson Theorem~\cite[Chapter 7, Thm. 4.1]{PH:02}, this trajectory must converge to an equilibrium point or a limit cycle. 
    %
    %
    It can not converge to a limit cycle because 
     $\tilde{\Phi}\cap\Cc$ does not contain any limit cycles.
    %
    %
    It can not converge to an equilibrium point because it is not in any region of attraction of an asymptotically stable equilibrium, there are no saddle points in $R$,
    and it can not converge to $\bar{\bx}_*$ (resp. $\tilde{\bx}_*$), because otherwise $\bar{\bx}_*$ (resp. $\tilde{\bx}_*$) would have two trajectories converging to it with different tangent vectors, 
    which would contradict the fact that $\bar{\bx}_*$ (resp. $\tilde{\bx}_*$) is a saddle point.
    
    Therefore, we conclude that in each of the $|\Lc|$ regions formed by $\{ \gamma_j \}_{j\in\Lc}$, there is exactly one asymptotically stable equilibrium in their boundary.
    Since there are $k-|\Lc|$ other undesired equilibria, and since the origin is asymptotically stable,
    this means that $|\Lc| = k - |\Lc| + 1$. Hence, $|\Lc| = \frac{k+1}{2}$.
    Note also that by~\cite[Proposition 10]{indep-cbf}, the stability properties of undesired equilibria are independent of the choice of extended class $\Kc_{\infty}$ function~$\alpha$. 
    Hence, even though in our arguments we have chosen a specific extended class $\Kc_{\infty}$ function $\alpha$, the statement holds for any such~$\alpha$.
\end{proof}

Under the assumptions of Proposition~\ref{prop:number-of-undes-equilibria-stab-properties-finite-set-of-bounded-obstacles}, 
since $\real^2\backslash\Cc$ is bounded,~\cite[Proposition 3]{DEK:87-icra} 
implies that there does not exist a safe globally asymptotically stabilizing controller.
If no limit cycles exist (for example, under the conditions of Proposition~\ref{prop:no-limit-cycles-general}), by the 
Poincar\'e-Bendixson Theorem~\cite[Chapter 7, Thm. 4.1]{PH:02}, this
%
%
implies that there must exist at least one undesired equilibrium, and hence $k\geq 1$.
Note also that as shown in~\cite[Proposition 6.2]{indep-cbf}, under a large class of CBFs, the stability properties of the undesired equilibria remain the same.

Next we give a result similar to Proposition~\ref{prop:number-of-undes-equilibria-stab-properties-finite-set-of-bounded-obstacles} for the case when the safe set $\Cc$ is compact and connected.

\smallskip

\begin{proposition}\longthmtitle{Number of undesired equilibria for compact connected safe set}\label{prop:number-undesired-eq-compact-connected-safe-set}
    Let $h$ be a strict CBF and suppose Assumption~\ref{as: interior eq} holds.
    Let $\Cc$ be a compact connected set.
    %
    %
    Consider~\eqref{eq:general-system-1} with $n=2$ and assume its undesired equilibria $\hat{\Ec} = \{ \bx_*^{(i)} \}_{i=1}^k \subset \partial\Cc$
        %
    %
    are either asymptotically stable or saddle points.
    %
    %
    Then, $k$ is even, and $\frac{k}{2}$ equilibria are saddle points and $\frac{k}{2}$ equilibria are asymptotically stable.
    %
    %
\end{proposition}
%
\begin{proof}
    The proof follows a similar argument to that of Proposition~\ref{prop:number-of-undes-equilibria-stab-properties-finite-set-of-bounded-obstacles}. 
    Let $\Lc\subset[k]$ be the set of indices of undesired saddle points. We show that $|\Lc| = \frac{k}{2}$.
    For $j\in\Lc$, let $\gamma_j$ be a subset of the one-dimensional local stable manifold of~\eqref{eq:general-system-1} at $\bx_*^{ (j) }$ such that
    \begin{enumerate}
        \item $\gamma_j$ corresponds to a maximal trajectory of \eqref{eq:general-system-1};
        \item there exists $T>0$ with $\gamma_j(t)\in\text{Int}(\Cc)$ for all $t>T$.
    \end{enumerate}
    Note that for each $j\in\Lc$, either there exists $\tilde{t}_j\in\real$ such that $\gamma_j(\tilde{t}_j)\in\partial\Cc$ and $\gamma_j(t) \notin\Cc$ for $t<\tilde{t}_j$ or $\lim\limits_{t\to-\infty}\gamma_j(t) = \bx_*^{\prime}$, with $\bx_*^{\prime}$ an undesired equilibrium that is a saddle point.
    Indeed, otherwise, since $\Cc$ is bounded, by the Poincar\'e-Bendixson Theorem~\cite[Chapter 7, Thm. 4.1]{PH:02} $\lim\limits_{t\to-\infty}\gamma_j(t)$ would converge to a limit cycle or another equilibrium point. However, by Proposition~\ref{prop:no-limit-cycles-general}, there exists $\hat{a}>0$ such that~\eqref{eq:general-system-1} with linear extended class $\Kc_{\infty}$ function with slope greater than $\hat{a}$ does not have any limit cycles in $\Cc$. Since~\cite[Proposition 11]{indep-cbf} ensures that the stability properties of undesired equilibria are independent of $\alpha$, we can assume without loss of generality that $\lim\limits_{t\to-\infty}\gamma_j(t)$ does not converge to a limit cycle. Moreover, $\lim\limits_{t\to-\infty}\gamma_j(t)$ cannot converge to an asymptotically stable equilibrium, since it belongs to the one-dimensional local stable manifold of $\bx_*^{ (j) }$.
    
    Now we note that, since for all $j\in\Lc$, either there exists $\tilde{t}_j$ such that $\gamma_j(\tilde{t}_j)\in\partial\Cc$ and $\gamma_j(t) \notin\Cc$ for $t<\tilde{t}_j$ or $\lim\limits_{t\to-\infty}\gamma_j(t) = \bx_*^\prime$, with $\bx_*^\prime$ an undesired equilibrium that is a saddle point, the trajectories
    $\{ \gamma_j \}_{j\in\Lc}$ divide $\Cc$ into $|\Lc|+1$ connected sets in~$\Cc$.
    By an argument analogous to the one in the proof of Proposition~\ref{prop:number-of-undes-equilibria-stab-properties-finite-set-of-bounded-obstacles},
    in each of those sets there must exist exactly one asymptotically stable equilibrium. Since the origin is asymptotically stable under~\eqref{eq:general-system-1}, this implies that the number of undesired equilibria that are saddle points is equal to the number of undesired equilibria that are asymptotically stable, proving $|\Lc|=\frac{k}{2}$.
\end{proof}

Note that by~\cite[Proposition 11]{indep-cbf}, the Jacobian of~\eqref{eq:general-system-1} evaluated at an undesired equilibrium has at least one negative eigenvalue. Therefore, the assumption in Propositions~\ref{prop:number-of-undes-equilibria-stab-properties-finite-set-of-bounded-obstacles} and~\ref{prop:number-undesired-eq-compact-connected-safe-set} that all undesired equilibria are either asymptotically stable or saddle points is satisfied if at any undesired equilibrium the other eigenvalue of the Jacobian is nonzero. If the other eigenvalue is zero and the equilibrium is degenerate, 
the point is asymptotically stable if the trajectories in its central manifold converge to it and it is a saddle point if the trajectories in its central manifold diverge from it.

We also note that Propositions~\ref{prop:number-of-undes-equilibria-stab-properties-finite-set-of-bounded-obstacles} and~\ref{prop:number-undesired-eq-compact-connected-safe-set} provide information about the number and the stability properties of undesired equilibria even in the case where the algebraic equations~\eqref{eq: condition-eq} defining the undesired equilibria
are difficult to solve. We also point out that both results require  $\partial\Cc$ to be bounded and hence, by Lemma~\ref{lem:suff-cond-isolated-eq}, the assumption that the number of equilibria is finite is equivalent to each of them being isolated.

We finalize this section by introducing a class of safe sets that do not introduce undesired equilibria.

\begin{proposition}\longthmtitle{Class of safe sets with global asymptotic stability of origin}\label{prop:class-safe-sets-no-undesired-eq}
    Let $V:\real^n:\to\real$ be a  global
    Lyapunov function for the nominal system $\dot{\bx} = f(\bx) + g(\bx)k(\bx)$.
    Let $c>0$ and suppose $h(\bx) := c - V(\bx)$ is a strict CBF of $\Cc = \setdef{ \bx\in\real^n }{ h(\bx) \geq 0 }$.
    Then, the origin is asymptotically stable with region of attraction containing $\Cc$. In particular,~\eqref{eq:general-system-1} does not contain any undesired equilibria.
\end{proposition}
\begin{proof}
  Let us first show that the safety filter is inactive at all points in $\Cc$.
  Indeed, for all $\bx\in\Cc$, 
  $\eta(\bx) = \nabla h(\bx)^\top (f(\bx)+g(\bx)k(\bx)) + \alpha(h(\bx)) = -\nabla V(\bx)^\top (f(\bx)+g(\bx)k(\bx)) + \alpha(h(\bx)) \geq 0$.
  Therefore, $v(\bx) = \mathbf{0}_m$ for all $\bx\in\Cc$. This implies that for all $\bx\in\Cc$, $\nabla V(\bx)^\top (f(\bx)+g(\bx)k(\bx)+g(\bx)v(\bx)) < 0$ for all $\bx\in\Cc\backslash\{ \mathbf{0}_n \}$, and therefore all trajectories with initial condition in $\Cc\backslash\{ \mathbf{0}_n \}$ converge to the origin, i.e., the origin is asymptotically stable with region of attraction containing $\Cc$. In particular, this implies that no undesired equilibria exist 
  (i.e., $\hat{\Ec} = \emptyset$), since otherwise trajectories with initial condition in such undesired equilibria do not converge to the origin.
\end{proof}

\section{Dynamical Properties of Safety Filters for Linear Planar Systems}\label{sec:dynamical-properties-safety-filters}
%
%


In this section, we focus on linear planar systems. Due to their simpler structure, solving~\eqref{eq: condition-eq} leads to additional results and insights, compared to the general treatment presented in the previous section. Consider the LTI planar system
\begin{align}\label{eq:2d-linear-system}
        \dot{\bx} = A\bx + B\bu,
\end{align}
where $\bx = [x_1,x_2]^\top \in\real^2$, $\bu\in\real^m$, with $m\in\{1,2\}$, $A\in\real^{2\times 2}$, and  $B\in\real^{2 \times m}$ having full column rank. We make the following assumption.

\smallskip

\begin{assumption}[Stabilizability]\label{as:A-B-stabilizable}
    The system~\eqref{eq:2d-linear-system} is stabilizable. Moreover, $\bu = - K \bx$, $K \in\real^{2\times m}$, is a stabilizing controller such that $\tilde{A} = A-BK$ is Hurwitz.  \hfill $\Box$
\end{assumption}

\vspace{.1cm}

In this setup, the system \eqref{eq:general-system-1} is then customized as follows.
\begin{align}\label{eq:v-linear-system}
    \dot{\bx} = F(\bx) := (A-BK)\bx + B v(\bx),
\end{align}
where the safety filter is given by
\begin{align}\label{eq:v-linear-expression}
v(\bx) = \begin{cases}
            0, &\ \text{if} \ \eta(\bx) \geq 0, \\
            -\frac{\eta(\bx) G(\bx)^{-1} B^T\nabla h(\bx) }{ 
        \norm{B^T \nabla h(\bx)}_{G^{-1}(\bx)}^2 }, &\ \text{if} \ \eta(\bx) <  0,
        \end{cases}
\end{align}
with $\eta(\bx) := \nabla h(\bx)^T (A-BK)\bx + \alpha(h(\bx))$.

As shown in~\cite[Proposition 6.2]{indep-cbf}, the stability properties of the undesired equilibria in the different results of this section hold for a large class of choices of the CBF~$h$ and the function~$\alpha$.

\subsection{Bounded safe set}

Here we discuss various results and examples for the case where the safe set is compact and contains the origin.
We start by showing that for linear, planar underactuated systems and safe sets that are parametrizable in polar coordinates by a continuously differentiable function, the system~\eqref{eq:v-linear-system} does not have undesired equilibria.

\smallskip

\begin{proposition}\longthmtitle{No undesired equilibria for underactuated planar systems and safe sets parametrizable in polar coordinates}\label{prop:safety-filters-work-for-safe-sets-star-shaped}
    Consider~\eqref{eq:2d-linear-system}
     with $m=1$ (i.e., the system is underactuated), and suppose that Assumptions~\ref{as: interior eq} and~\ref{as:A-B-stabilizable} hold.
    Let $r:\real\to\real_{>0}$ be a continuously differentiable, $2\pi$-periodic function,
    %
    %
    and such that
    \begin{align}\label{eq:boundary-parametrization}
        \partial\Cc = \setdef{(r(\theta)\cos(\theta), r(\theta)\sin(\theta) ) }{\theta \in [0,2\pi]},
    \end{align}
    and let $h$ be a strict CBF of $\Cc$.
    Then,~\eqref{eq:v-linear-system} does not have any undesired equilibria for any $K\in\real^2$ and $\alpha \in \realpos$.
\end{proposition}
\begin{proof}
    For convenience, denote
    \begin{align*}
      A=\begin{bmatrix}
            a_{11} & a_{12} \\
            a_{21} & a_{22}
        \end{bmatrix},~
    B=\begin{bmatrix}
            b_1 \\
            b_2
    \end{bmatrix}, \ K = \begin{bmatrix}
        k_1 \ k_2
    \end{bmatrix},
    \end{align*}
    in~\eqref{eq:2d-linear-system} and let
    $\beta=a_{11}b_2 - b_1 a_{21}$, $\gamma = a_{22} b_1 - b_2 a_{12}$.
    Further define $R:\real^2\backslash\{ \mathbf{0}_2 \} \to\real$ as
    \begin{align*}
        R(\by) = \begin{cases}
            r(\frac{\pi}{2})^2 \ &\text{if} \ y_1 = 0, \ y_2 > 0,\\
            r(\frac{3\pi}{2})^2 \ &\text{if} \ y_1 = 0, \ y_2 < 0,\\
            r(\arctan(\frac{y_2}{y_1}))^2 \ &\text{otherwise},\\
        \end{cases}
    \end{align*}
    We recall the following three facts:
    \begin{enumerate}
        \item By~\cite[Corollary 4.5]{indep-cbf}, undesired equilibria are independent of the choice of CBF;
        \item By Lemma~\ref{prop:if-one-strict-cbf-exists-all-cbfs-are-strict}, since $h$ is a strict CBF of $\Cc$, any CBF of $\Cc$ is a strict CBF of $\Cc$;
        \item By~\cite[Theorem 3]{ADA-SC-ME-GN-KS-PT:19}, since $\Cc$ is safe, any continuously differentiable function whose $0$-superlevel is $\Cc$ is a CBF of $\Cc$.
    \end{enumerate}
    These facts imply that, without loss of generality, we can assume that $h$ is such that in a neighborhood $\Nc_c$ of $\partial\Cc$, $h(\by) = -\norm{\by}^2 + R(\by)$
    for all $\by\in\Nc_{c}$, and that such $h$ is a strict CBF.
    Note that for all $\bx\in\partial\Cc$ with $x_1\neq0$, we have
    \begin{align}\label{eq:gradient-polar-coordinates}
        \nabla h(\bx) = \begin{pmatrix}
            -2x_1 + 2r( \arctan(\frac{x_2}{x_1}) ) r^{\prime}(\arctan(\frac{x_2}{x_1}))\frac{-x_2}{x_1^2 + x_2^2} \\
            -2x_2 + 2r( \arctan(\frac{x_2}{x_1}) ) r^{\prime}(\arctan(\frac{x_2}{x_1}))\frac{x_1}{x_1^2 + x_2^2}
        \end{pmatrix}.
    \end{align}
    Since $\nabla h$ is continuous, if $x_1=0$ and $x_2>0$, we have
    \begin{align}\label{eq:gradient-polar-coordinates-2}
        \nabla h(0,x_2) = \begin{pmatrix}
            -2x_1 + 2r( \arctan(\frac{\pi}{2})) ) r^{\prime}(\arctan(\frac{\pi}{2}))\frac{-1}{x_2} \\
            -2x_2
        \end{pmatrix},
    \end{align}
    whereas if  $x_1=0$ and $x_2 < 0$,
    \begin{align}\label{eq:gradient-polar-coordinates-3}
        \nabla h(0,x_2) = \begin{pmatrix}
            -2x_1 + 2r( \arctan(\frac{3\pi}{2})) ) r^{\prime}(\arctan(\frac{3\pi}{2}))\frac{-1}{x_2} \\
            -2x_2
        \end{pmatrix}.
    \end{align}
    Therefore,~\eqref{eq:gradient-polar-coordinates} is valid also for $x_1 =0$ by defining $\arctan(\frac{x_2}{x_1}) = \frac{\pi}{2}$ if $x_1 = 0$ and $x_2 > 0$, and 
    $\arctan(\frac{x_2}{x_1}) = \frac{3\pi}{2}$ if $x_1 = 0$ and $x_2 < 0$.
    From~\eqref{eq: condition-eq}, 
    the undesired equilibria lie in $\partial\Cc$, and therefore are of the form $( r(\theta^*)\cos(\theta^*), r(\theta^*)\sin(\theta^*) )$.
    Furthermore, from~\eqref{eq: condition-eq} and~\eqref{eq:gradient-polar-coordinates}, we have
    \begin{align*}
        \begin{pmatrix}
            a_{11} +b_1 k_1 & a_{12}+b_1k_2 \\
            a_{21}+b_2k_1 & a_{22}+b_2k_2
        \end{pmatrix}
        \begin{pmatrix}
            r(\theta^*)\cos(\theta^*) \\
            r(\theta^*)\sin(\theta^*)
        \end{pmatrix}
        = \delta
        \begin{pmatrix}
            b_1 \\
            b_2
        \end{pmatrix}
        \begin{pmatrix}
            b_1 \\
            b_2
        \end{pmatrix}^T
        \nabla h
        \begin{pmatrix}
        r(\theta^*)\cos(\theta^*) \\
        r(\theta^*)\sin(\theta^*)
        \end{pmatrix}.
    \end{align*}
    It follows that $\theta^*$ satisfies
    \begin{align*}
        b_2 \begin{pmatrix}
            a_{11}+b_1 k_1 & a_{12} + b_1 k_2 
        \end{pmatrix}
        \begin{pmatrix}
            r(\theta^*)\cos(\theta^*) \\
            r(\theta^*)\sin(\theta^*)
        \end{pmatrix}
        =
        b_1 \begin{pmatrix}
            a_{21}+b_2 k_1 & a_{22} + b_2 k_2 
        \end{pmatrix}
        \begin{pmatrix}
            r(\theta^*)\cos(\theta^*) \\
            r(\theta^*)\sin(\theta^*)
        \end{pmatrix},
    \end{align*}
    which implies that $\theta^*$ satisfies $\gamma\sin(\theta^*) = \beta\cos(\theta^*)$.
    Note that there exist exactly two values $\theta^*$ in $[0, 2\pi]$ solving $\gamma\sin(\theta^*) = \beta\cos(\theta^*)$.
    %
    %
    Hence, there exist exactly two potential undesired equilibria (i.e., $\Ec$ has cardinality $2$).
    Let $\bx_*^{(1)}$, $\bx_*^{(2)}$ be such undesired equilibria, with $\theta_1^*$ and $\theta_2^*$ being their associated $\theta^*$ values.
    From Lemma~\ref{lem:undesired-eq-characterization}, $\bx_*^{(1)}$ is an undesired equilibria if and only if $\eta(\bx_*^{(1)}) < 0$, or
    equivalently, 
    \begin{align}\label{eq:active-filter-condition-underactuated}
        \notag
        \frac{\partial h}{\partial x_1}(\bx_*^{(1)}) \Big( (a_{11}-b_1 k_1) \cos(\theta_1^*) + (a_{12}-b_1 k_2) \sin(\theta_1^*)  \Big) + \\
        \frac{\partial h}{\partial x_2}(\bx_*^{(1)}) \Big( (a_{21}-b_2 k_1) \cos(\theta_1^*) + (a_{22}-b_2 k_2) \sin(\theta_1^*)  \Big) < 0.
    \end{align}
    If $\gamma\neq0$, 
    $\sin(\theta_1^*) = \frac{\beta}{\gamma}\cos(\theta_1^*)$,
    %
    %
    and~\eqref{eq:active-filter-condition-underactuated} is equivalent to
    \begin{align*}
        &\frac{\partial h}{\partial x_1}(\bx_*^{(1)}) \frac{\cos(\theta_1^*)}{\gamma} \Big( \gamma (a_{11}-b_1 k_1) + \beta (a_{12}-b_1 k_2)  \Big) + \\
        &\frac{\partial h}{\partial x_2}(\bx_*^{(1)}) \frac{\cos(\theta_1^*)}{\gamma} \Big( \gamma (a_{21}-b_2 k_1) + \beta (a_{22}-b_2 k_2)  \Big) < 0.
    \end{align*}
    On the other hand, if $\gamma=0$, $\cos(\theta_1^*)=0$, and~\eqref{eq:active-filter-condition-underactuated} is equivalent to
    \begin{align*}
        \sin(\theta_1^*)\Big( \frac{\partial h}{\partial x_1}(\bx_*^{(1)}) (a_{12}-b_1 k_2) +
        \frac{\partial h}{\partial x_2}(\bx_*^{(1)}) (a_{22}-b_2 k_2) \Big) < 0.
    \end{align*}
    %
    %
    Note also that $\gamma = \beta = 0$ leads to no undesired equilibria.
    Indeed, using the definitions of $\beta$ and $\gamma$,
    \begin{subequations}
    \begin{align}
        \gamma (a_{11} \! - \! b_1 k_1) \! + \! \beta (a_{12} - b_1 k_2) \!  &= \! b_1 (\text{det}(A) \! - \! \gamma k_1 - \! \beta k_2), \\
        \gamma (a_{21} \! - \! b_2 k_1) \! + \! \beta (a_{22} - b_2 k_2) \! &= \! b_2 (\text{det}(A) \! - \gamma k_1 \! - \beta k_2 ),
    \end{align}
    \label{eq:gamma-beta-relation}
    \end{subequations}
    which implies that $b_1 \text{det}(A) = b_2 \text{det}(A) = 0$ if $\gamma = \beta = 0$.
    If $\text{det}(A) = 0$, then $\text{det}(A+BK) = \text{det}(A)-\gamma k_1 -\beta k_2 =0$ for all $K$, which contradicts Assumption~\ref{as:A-B-stabilizable}.
    Hence, $b_1 = b_2 = 0$, in which case~\eqref{eq: condition-eq} implies that $(A-BK)\bx_*^{(1)} = (A-BK)\bx_*^{(2)} = \mathbf{0}_2$, which can only hold if $\bx_*^{(1)} = \bx_*^{(2)} = \mathbf{0}_2$ and they are not undesired equilibrium. Hence the rest of the proof focuses on the case $\beta^2 + \gamma^2 > 0$.
    %
    %
    By the Routh-Hurwitz criterion~\cite{AH:95}, since $A-BK$ is Hurwitz and is a $2\times2$ matrix, $\text{det}(A-BK) = \text{det}(A) - \gamma k_1 - \beta k_2  > 0$.
    %
    %
    Hence, if $\gamma\neq0$, by using~\eqref{eq:gamma-beta-relation},~\eqref{eq:active-filter-condition-underactuated} is equivalent to
    \begin{align*}
        \frac{\cos(\theta_1^*)}{\gamma}\Big( \frac{\partial h}{\partial x_1}(\bx_*^{(1)}) b_1 + \frac{\partial h}{\partial x_2}(\bx_*^{(1)}) b_2 \Big) < 0.
    \end{align*}
    %
    %
    On the other hand, if $\gamma = 0$, by using~\eqref{eq:gamma-beta-relation},~\eqref{eq:active-filter-condition-underactuated} is equivalent to
    \begin{align*}
        \sin(\theta_1^*)\Big( \frac{\partial h}{\partial x_1}(\bx_*^{(1)}) \frac{b_1}{\beta} + \frac{\partial h}{\partial x_2}(\bx_*^{(1)}) \frac{b_2}{\beta}
        \Big) < 0.
    \end{align*}
    %
    %
    Using~\eqref{eq:gradient-polar-coordinates}, we get that
    if $\gamma\neq0$,~\eqref{eq:active-filter-condition-underactuated} is equivalent to
    \begin{align}\label{eq:active-filter-condition-2}
        \Big(\frac{\cos(\theta_1^*)}{\gamma} \Big)^2 \! \Big( \! -r(\theta_1^*)(b_1\gamma + b_2\beta) \! + \! r^{\prime}(\theta_1^*)(-b_1\beta + b_2\gamma)  \Big) \! < \! 0,
    \end{align}
    %
    %
    whereas if $\gamma = 0$, $\cos(\theta_1^*)=0$, which means that $\sin^2(\theta_1^*)=1$,
    %
    %
    and again using~\eqref{eq:gradient-polar-coordinates},~\eqref{eq:active-filter-condition-underactuated} is equivalent to
    \begin{align}\label{eq:active-filter-condition-3}
        \frac{1}{\beta}\Big( -r^{\prime}(\theta_1^*)b_1 - r(\theta_1^*)b_2 \sin^2(\theta_1^*)  \Big)
        =
        \frac{1}{\beta}\Big( -r^{\prime}(\theta_1^*)b_1 - r(\theta_1^*)b_2 \Big)
        \! < \! 0,
    \end{align}
    Note that similar expressions to~\eqref{eq:active-filter-condition-2} and~\eqref{eq:active-filter-condition-3} hold for $\theta_2^*$.
    In particular, this means that if $\gamma\neq 0$, the sign of 
    $-r(\theta_1^*)(b_1\gamma + b_2\beta) + r^{\prime}(\theta_1^*)(-b_1\beta + b_2\gamma)$ and $-r(\theta_2^*)(b_1\gamma + b_2\beta) + r^{\prime}(\theta_2^*)(-b_1\beta + b_2\gamma)$ is the same, and if $\gamma = 0$, the sign of $-r^{\prime}(\theta_1^*)b_1 - r(\theta_1^*)b_2$ and $-r^{\prime}(\theta_2^*)b_1 - r(\theta_2^*)b_2$ is the same.

    Next, we compute the eigenvalues of the Jacobian of~\eqref{eq:general-system-1} at $\bx_*^{(1)}$.
    By~\cite[Proposition 11]{indep-cbf}, one of its eigenvalues is $-\alpha^{\prime}(0)$. By using the expression of the Jacobian of~\eqref{eq:general-system-1} at $\bx_*^{(1)}$ also provided in~\cite[Proposition 11]{indep-cbf},
    and using the fact that the trace is the sum of the eigenvalues,
    we obtain that the other eigenvalue is
    \begin{align*}
        \lambda_{\bx_{\star}^{(1)}} := \frac{\beta \frac{\partial h}{\partial x_2}(\bx_*^{(1)}) + \gamma \frac{\partial h}{\partial x_1}(\bx_*^{(1)}) }{ b_1 \frac{\partial h}{\partial x_1}(\bx_*^{(1)}) + b_2 \frac{\partial h}{\partial x_2}(\bx_*^{(2)}) } .
    \end{align*}
    Using~\eqref{eq:gradient-polar-coordinates}, we get that if $\gamma \neq 0$,
    \begin{align}\label{eq:other-eigenvalue}
        \lambda_{\bx_{\star}^{(1)}} = \frac{ r(\theta_1^*)(\beta^2 + \gamma^2) }{ r(\theta_1^*)(b_1\gamma + b_2 \beta) - r^{\prime}(\theta_1^*)(-b_1 \beta + b_2 \gamma) },
    \end{align}
    whereas if $\gamma = 0$,
    \begin{align}\label{eq:other-eigenvalue-3}
        \lambda_{\bx_{\star}^{(1)}} = \frac{\beta r(\theta_1^*) }{ r^{\prime}(\theta_1^*)b_1 + r(\theta_1^*)b_2  }.
    \end{align}
    Now, since if $\gamma \neq 0$, the sign of 
    $-r(\theta_1^*)(b_1\gamma + b_2\beta) + r^{\prime}(\theta_1^*)(-b_1\beta + b_2\gamma)$ and $-r(\theta_2^*)(b_1\gamma + b_2\beta) + r^{\prime}(\theta_2^*)(-b_1\beta + b_2\gamma)$ is the same, and if $\gamma = 0$ the sign of 
    $r^{\prime}(\theta_1^*)b_1 + r(\theta_1^*)b_2$ and $r^{\prime}(\theta_2^*)b_1 + r(\theta_2^*)b_2$ is the same, 
    $\bx_*^{(1)}$ and $\bx_*^{(2)}$ have the same stability properties.
    Note also that 
    since $r$ is strictly positive by definition,
    $r(\theta_1^*)>0$,
    %
    %
    and since we are discussing the case $\beta^2 + \gamma^2 > 0$,
    $\bx_*^{(1)}$ and $\bx_*^{(2)}$ cannot be degenerate undesired equilibria.
    %
    %
    However, according to Proposition~\ref{prop:number-undesired-eq-compact-connected-safe-set}, since $\Cc$ is compact, connected, and contains the origin, if there exist two undesired equilibria, one must be a saddle point and the other one must be asymptotically stable, which is a contradiction with the fact that $\bx_*^{(1)}$ and $\bx_*^{(2)}$ have the same stability properties.
    This implies that $\bx_*^{(1)}$ and $\bx_*^{(2)}$ cannot be undesired equilibria and therefore~\eqref{eq:v-linear-system} does not have any, {\color{blue}
    i.e., $\hat{\Ec} = \emptyset$.}
    %
\end{proof}

Figure~\ref{fig:polar-parametrization-examples} illustrates different examples of safe sets satisfying the assumptions in Proposition~\ref{prop:safety-filters-work-for-safe-sets-star-shaped}. 
The following example shows that the conclusions of Proposition~\ref{prop:safety-filters-work-for-safe-sets-star-shaped}
do not hold if~\eqref{eq:2d-linear-system} is fully actuated.

\begin{figure}[htb]
  \centering 
  {\includegraphics[width=0.32\textwidth]{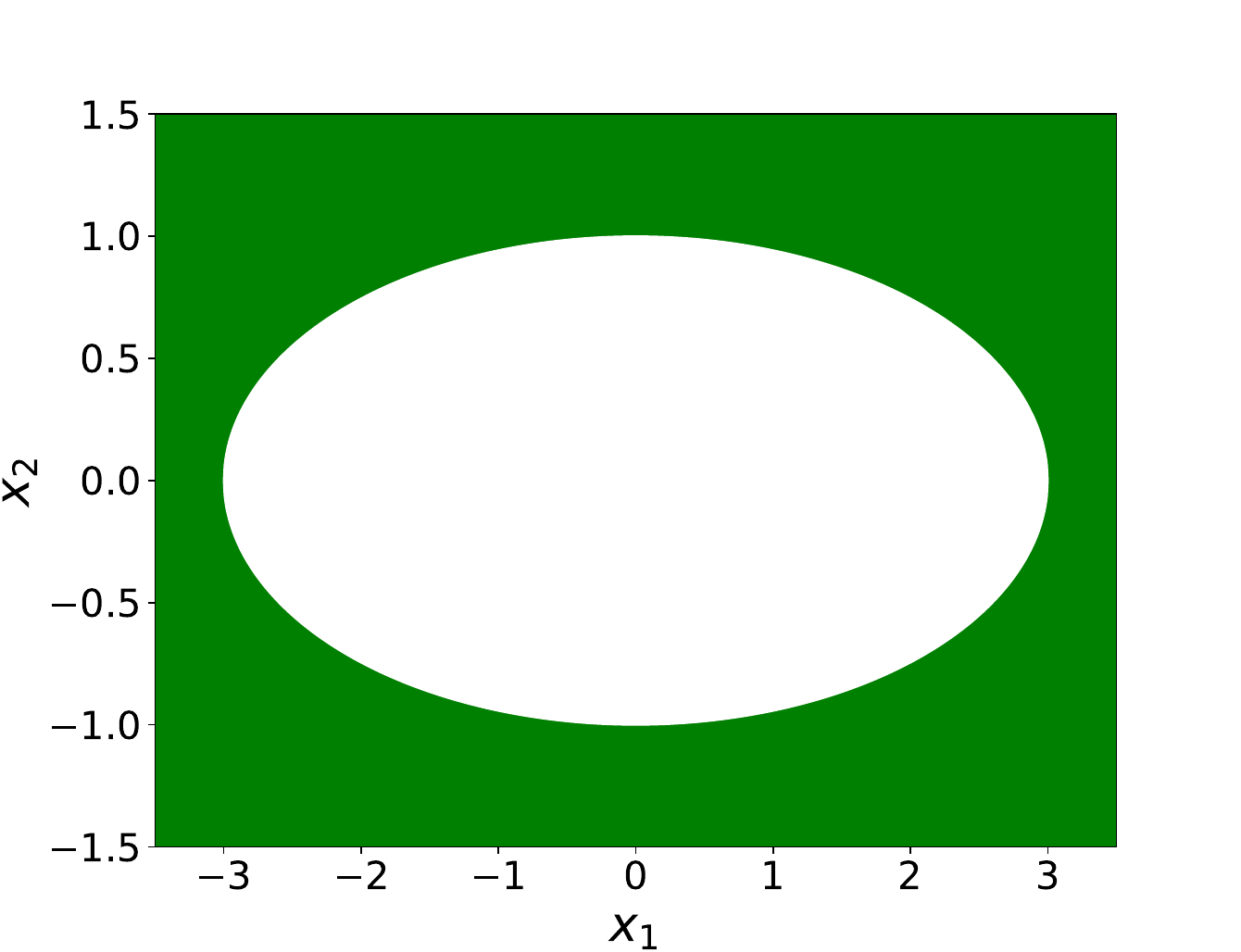}} 
  {\includegraphics[width=0.32\textwidth]{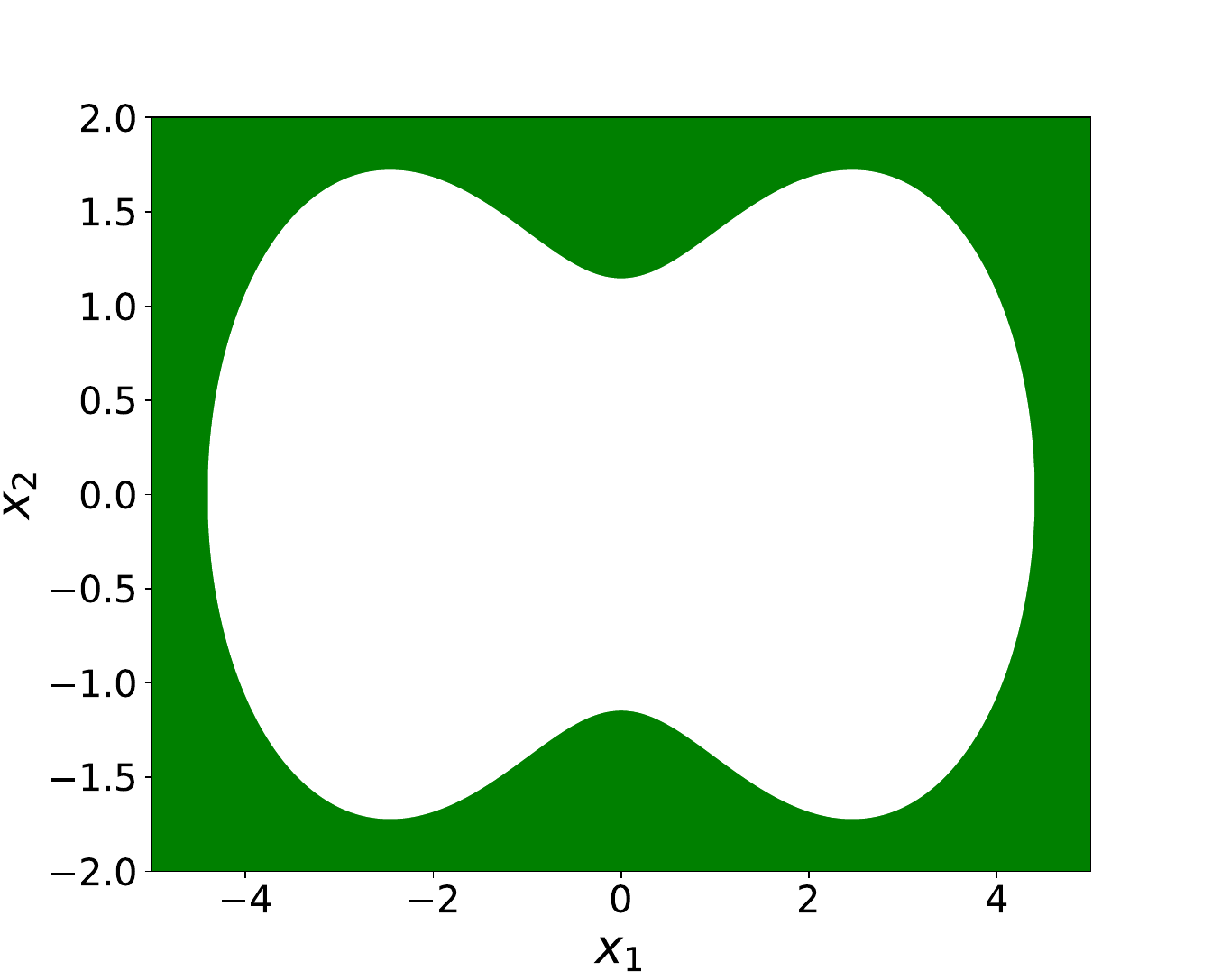}}
  {\includegraphics[width=0.32\textwidth]{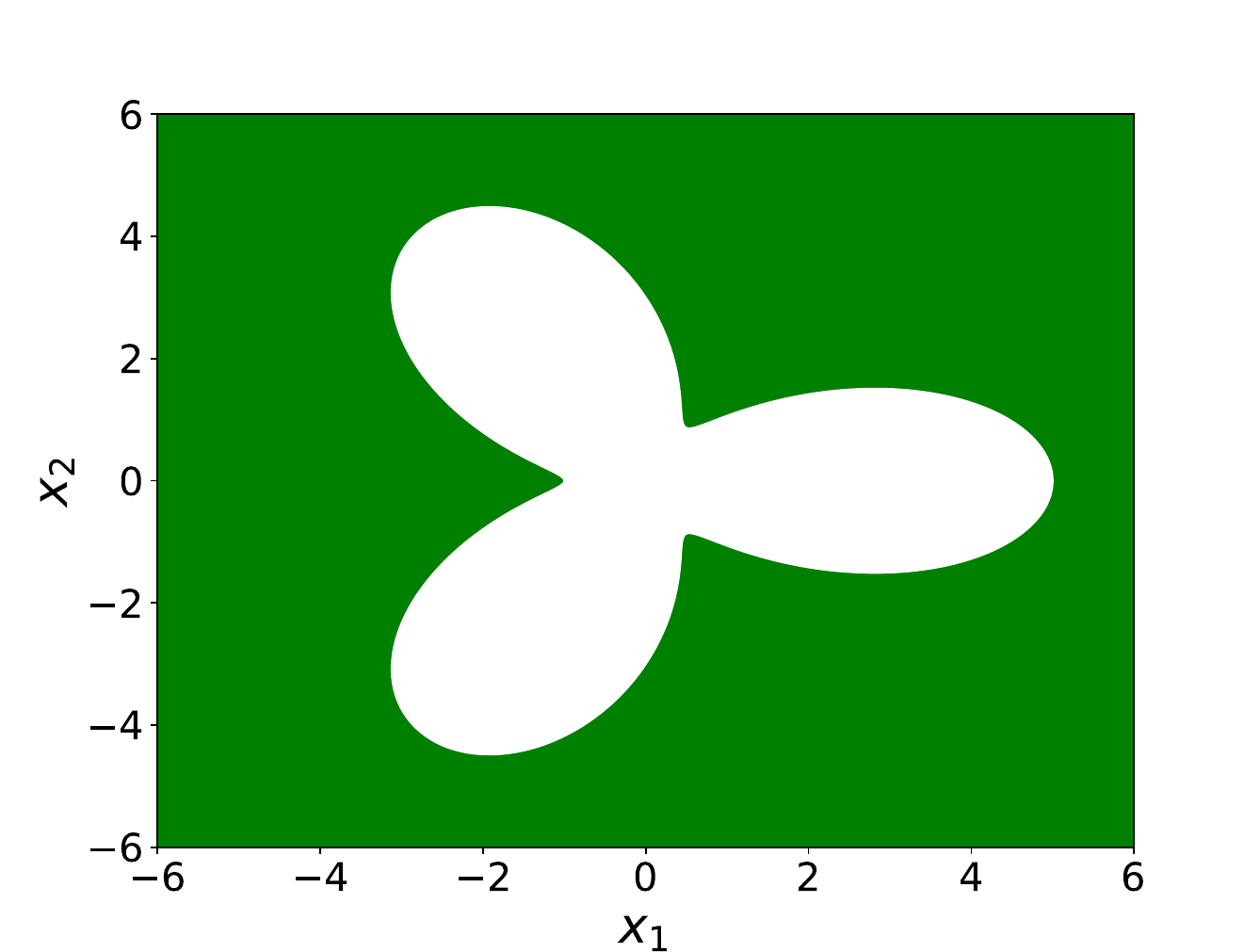}} 
  \caption{Examples of safe sets satisfying the assumptions of Proposition~\ref{prop:safety-filters-work-for-safe-sets-star-shaped} The unsafe region is colored in green and the safe set in white. For the left figure, $r(\theta) = \frac{1}{ \sqrt{3\cos^2(\theta) + \sin^2(\theta)} }$, for the center figure, $r(\theta) = 3\sqrt{ \cos(2\theta) + \sqrt{ (1.03)^4 - \sin^2(\theta) } }$, and for the right figure, $r(\theta) = 3 + 2\cos(3\theta)$.}
  \label{fig:polar-parametrization-examples}
   \vspace{-.5cm}
\end{figure}


\smallskip

\begin{example}\longthmtitle{Undesired equilibrium in convex and bounded safe set, fully actuated system}\label{ex:undesired-eq-convex-bounded}
%
%
{\rm
Consider the planar single integrator, i.e.,~\eqref{eq:v-linear-system} with $A=\zero_2$, $B=\textbf{I}_2$, and let $K = \begin{bmatrix}
    5 & -8 \\
    2 & -3
\end{bmatrix}$, $G(\bx)=\textbf{I}_2$,
$h(\bx) = 10-x_1^2-(x_2-2)^2$, and $\alpha(s) = 50s$.
Note that $-K$ is Hurwitz.
By numerically solving the conditions in~\eqref{eq: condition-eq}, it follows that $(3,3)$ and $(3.161, 2.123)$ are undesired equilibria of~\eqref{eq:v-linear-system}. Moreover, using the expression of the Jacobian in~\cite[Proposition 11]{indep-cbf}, we deduce that $(3,3)$ is asymptotically stable and $(3.161, 2.123)$ is a saddle point. This is illustrated in Figure~\ref{fig:counterexample-fully-actuated}.
}
\problemfinal
\end{example}

\begin{figure}[htb]
  \centering
  {\includegraphics[width=0.7\textwidth]{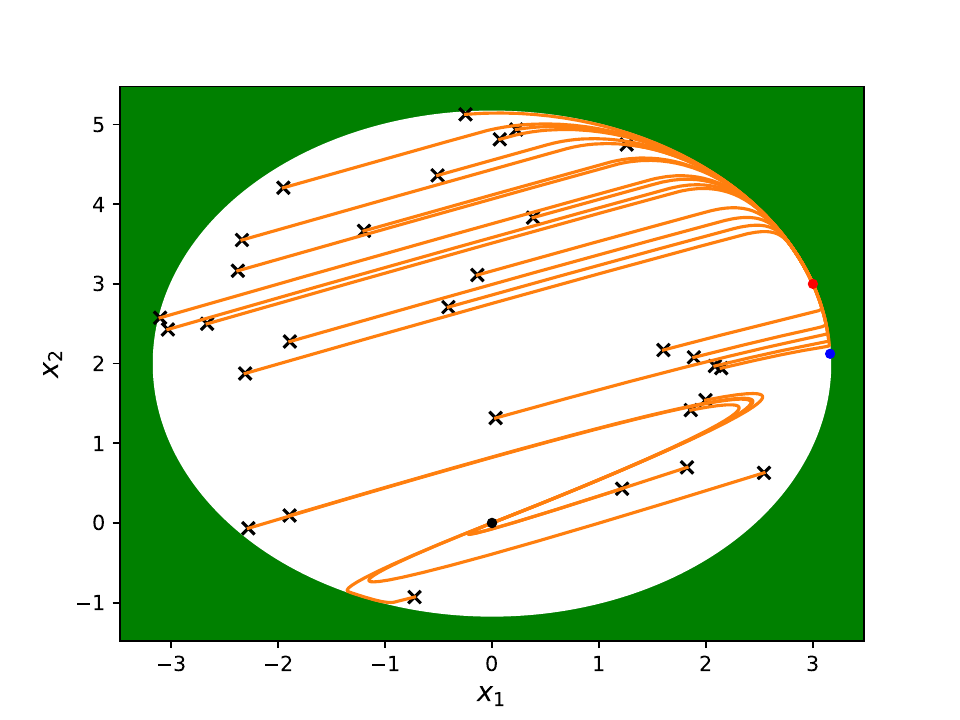}} 
  \caption{Plot of different trajectories (in orange) for Example~\ref{ex:undesired-eq-convex-bounded}. The unsafe set is colored in green. Black crosses denote initial conditions, the black dot denotes the origin, the red dot denotes an asymptotically stable undesired equilibrium, and the blue dot denotes an undesired equilibrium that is a saddle point.}
  \label{fig:counterexample-fully-actuated}
\end{figure}

We also note that a critical assumption in Proposition~\ref{prop:safety-filters-work-for-safe-sets-star-shaped} is that the  boundary of the safe set can be parametrized as in~\eqref{eq:boundary-parametrization}. In the following example, we show that if such a parametrization does not exist, the closed-loop system~\eqref{eq:v-linear-system} can  have undesired equilibria even if the system is underactuated and $h$ is a strict CBF. 

\smallskip

\begin{example}\longthmtitle{Undesired equilibrium in sets not parametrizable in polar coordinates}\label{ex:undesired-no-polar-coords}
{\rm
    Consider~\eqref{eq:v-linear-system} with
    \begin{align*}
        &A = \begin{bmatrix}
            a_{11} & a_{12} \\
            a_{21} & a_{22}
        \end{bmatrix}, \ 
        B = \begin{bmatrix}
            b_1 \\ 
            b_2
        \end{bmatrix}, \
        K = \begin{bmatrix}
            k_1 \\
            k_2
        \end{bmatrix}^T, \\
        &h(\bx) \!=\! b^4 \!-\! \norm{\bx-\begin{bmatrix}
        a-c_1 \\
        c_2
        \end{bmatrix}}^2 \norm{\bx+\begin{bmatrix}
        a+c_1 \\
        -c_2
        \end{bmatrix}}^2, \ \alpha(s) \!=\! 50s,
    \end{align*}
    where $a_{11} = 1.878$, $a_{12} = -6.247$, $a_{21} = -3.189$, $a_{22} = 6.731$, $b_1 = 4.166$, $b_2 = -8.172$, $k_1 = 1.495$, $k_2 = -1.515$,
    $a = 6.587$, $b = 6.591$, $c_1 = -5$, $c_2=0$.
    The boundary of $\Cc$ cannot be described as in~\eqref{eq:boundary-parametrization}.
    Note that since $B$ is a column vector,~\eqref{eq:v-linear-system} is independent of $G$.
    By numerically solving the conditions in~\eqref{eq: condition-eq},
    it follows that $(5.431,0.487)$ and $(4.651, 0.417)$ are undesired equilibria of~\eqref{eq:v-linear-system}. Moreover, using the expression of the Jacobian in~\cite[Proposition 10]{indep-cbf}, it follows that $(5.431,0.487)$ is asymptotically stable and $(4.651, 0.417)$ is a saddle point. This is illustrated in Figure~\ref{fig:cassini-safe-set-zoomout}.
    \begin{figure}[htb]
      \centering
      {\includegraphics[width=0.7\textwidth]{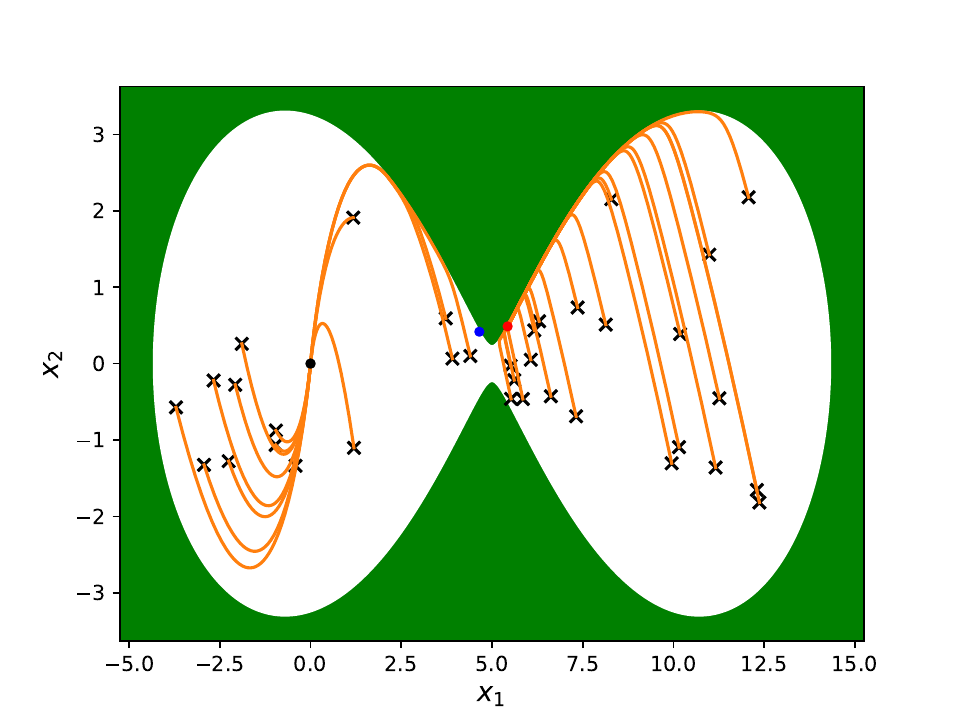}} 
      \caption{Plot of different trajectories (in orange) for Example~\ref{ex:undesired-no-polar-coords}. The unsafe set is colored in green. Black crosses denote initial conditions, the black dot denotes the origin, the red dot denotes an asymptotically stable undesired equilibrium, and the blue dot denotes an undesired equilibrium that is a saddle point.}
      \label{fig:cassini-safe-set-zoomout}
    \end{figure}
    \problemfinal
    } 
\end{example}


\smallskip

The stability properties of undesired equilibria remain the same for a large class of CBFs, cf.~\cite[Proposition 6.2]{indep-cbf}, and therefore the asymptotically stable undesired equilibria defined in 
Examples~\ref{ex:undesired-eq-convex-bounded} and~\ref{ex:undesired-no-polar-coords} exist for a large class of CBFs, not just the ones defined therein.


\subsection{Bounded unsafe set}

Here we study the case where the unsafe set is bounded. In this case, recall (cf. our discussion after Proposition~\ref{prop:number-of-undes-equilibria-stab-properties-finite-set-of-bounded-obstacles}) that~\cite[Proposition 3]{DEK:87-icra} implies that there does not exist a safe globally asymptotically stabilizing controller.
If no limit cycles exist (for example, under the conditions of Proposition~\ref{prop:no-limit-cycles-general}), by the 
Poincar\'e-Bendixson Theorem~\cite[Chapter 7, Thm. 4.1]{PH:02}, this implies that there must exist at least one undesired equilibrium.
We next provide two  examples that illustrate how  system~\eqref{eq:v-linear-system} for linear planar plants can give rise in this case to asymptotically stable undesired equilibria, similarly to Examples~\ref{ex:undesired-eq-convex-bounded} and~\ref{ex:undesired-no-polar-coords}.


\smallskip

\begin{example}\longthmtitle{Nonconvex obstacle with asymptotically stable undesired equilibria}\label{ex:nonconvex-obstacle-asymptotically-stable-undesired-eq}
{\rm
Consider the single integrator in the plane, i.e.,~\eqref{eq:v-linear-system} with $A = \zero_2$, $B=\textbf{I}_2$, and
\begin{align*}
    &K=\begin{bmatrix}
    -k_1 & 0 \\
    0 & -k_2
    \end{bmatrix},
    \ 
    h(\bx) \!=\! -b^4 \!+\! \norm{\bx \! - \! \begin{bmatrix}
        a \\
        c_2
    \end{bmatrix}}^2 \norm{\bx \! + \! \begin{bmatrix}
    a \\
    -c_2
    \end{bmatrix}}^2, \ \alpha(s) \!=\! 10s,
\end{align*}
with $k_1>0$, $k_2>0$, $a>0$, $b>0$, $c_2>0$ and $\frac{b}{a}\in(1,\sqrt{2})$.
It follows that $\bx_*=(0,\sqrt{-a^2+b^2}+c_2)$ and $\delta_{\bx_*} = -\frac{c_2 + \sqrt{-a^2+b^2}}{2\sqrt{-a^2+b^2} b^2 } < 0$ solve~\eqref{eq: condition-eq} and therefore $\bx_*\in\hat{\Ec}$ is an undesired equilibrium.
Moreover, by leveraging~\cite[Proposition 11]{indep-cbf}, the Jacobian of~\eqref{eq:v-linear-system} evaluated at $\bx_*$ is equal to
\begin{align*}
    J({\bx_*}) = \begin{bmatrix}
        k_2 (b^2 - 2a^2)\frac{c_2+\sqrt{-a^2+b^2}}{\sqrt{-a^2+b^2}b^2} & 0 \\
        0 & -\alpha^{\prime}(0)
    \end{bmatrix}.
\end{align*}
Note that $J({\bx_*})$ is a diagonal matrix and since $\alpha^{\prime}(0) > 0$, $k_2>0$, $c_2 > 0$, and $\frac{b}{a}<\sqrt{2}$, all eigenvalues of $J({\bx_*})$ are negative and hence $\bx_*$ is asymptotically stable.
Figure~\ref{fig:cassini-obstacle-zoomin} shows different trajectories of the closed-loop system obtained with $k_1 = 1$, $k_2 = 1$, $a = 3$, $b=1.05a$, $c_2 = 4$.

\begin{figure}[htb]
  \centering 
  {\includegraphics[height=0.35\textwidth]{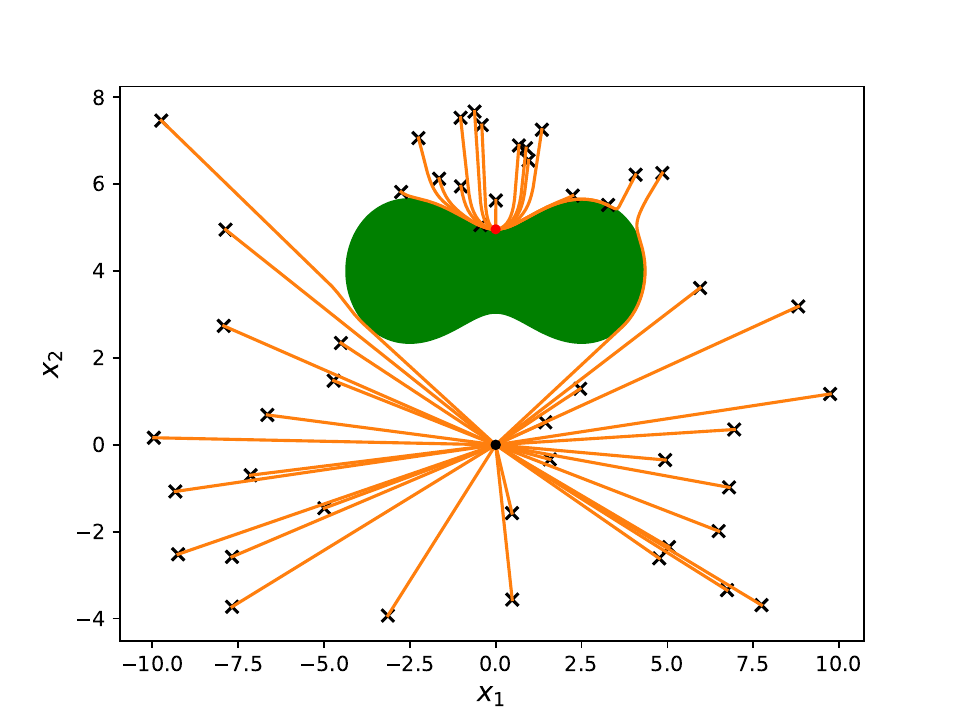}} 
  {\includegraphics[height=0.35\textwidth]{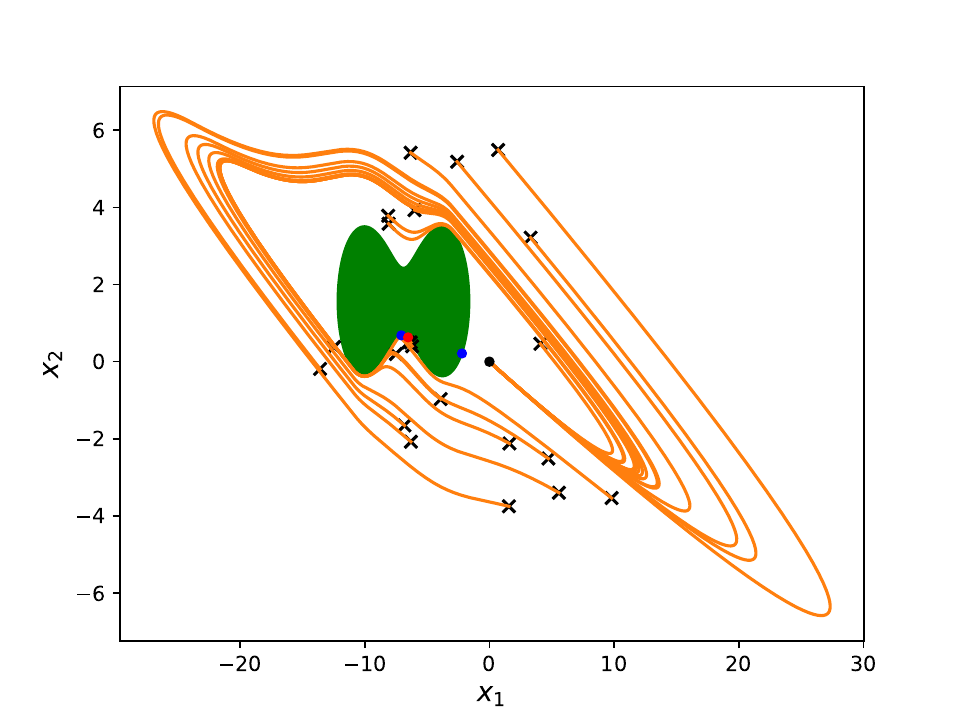}} 
  \caption{Plot of different trajectories of~\eqref{eq:v-linear-system} for the single integrator system (left) and the underactuated system in (right) in Example~\ref{ex:nonconvex-obstacle-asymptotically-stable-undesired-eq}.
  The unsafe set is colored in green. Trajectories are depicted in orange.
  Black crosses denote initial conditions, the black dot denotes the origin, and the red dot denotes the asymptotically stable undesired equilibrium.}
  \label{fig:cassini-obstacle-zoomin}
   \vspace{-.5cm}
\end{figure}
%
%

Next we present a similar example for a linear underactuated system. Consider~\eqref{eq:v-linear-system} with
\begin{align*}
    &A = \begin{bmatrix}
        a_{11} & a_{12} \\
        a_{21} & a_{22}
    \end{bmatrix}, \ 
    B = \begin{bmatrix}
        b_1 \\ b_2
    \end{bmatrix}, \ 
    K = \begin{bmatrix}
        k_1 \\
        k_2
    \end{bmatrix}, \\
    &h(\bx) \!=\! -b^4 \!+\! \norm{\bx \! - \! \begin{bmatrix}
        a-c_1 \\
        c_2
        \end{bmatrix}}^2 \norm{\bx \! + \! \begin{bmatrix}
        a+c_1 \\
        -c_2
        \end{bmatrix}}^2, \ \alpha(s) \!=\! 10s,
\end{align*}
where $a_{11} = 0.268$, $a_{12} = 2.866$, $a_{21} = 0.151$, $a_{22} = 1.526$, $b_1 = 0.350$, $b_2 = -0.151$, $k_1=-41.72$, $k_2=-174.8$, $a = 3.684$, $b = 3.785$, $c_1 = 6.891$, $c_2 = 1.565$. Note that since $B$ is a column vector,~\eqref{eq:v-linear-system} is independent of~$G$. 
By numerically solving the conditions in~\eqref{eq: condition-eq}, we obtain 
three undesired equilibria, two of which, located at $(-7.062,  0.682)$ and $(-2.197,  0.212)$, are saddle points and another, located at $(-6.519, 0.629)$, which is asymptotically stable. Figure~\ref{fig:cassini-obstacle-zoomin} shows that some of the trajectories of~\eqref{eq:v-linear-system} converge to $(-6.519, 0.629)$,
while some others converge to the origin.
\problemfinal
}
\end{example}

%
%

Example~\ref{ex:nonconvex-obstacle-asymptotically-stable-undesired-eq} shows that the closed-loop system~\eqref{eq:v-linear-system} can have asymptotically stable undesired equilibria. A natural question is to ask whether changing the nominal controller would make these equilibria disappear. The following example answers this question in the negative, showing that asymptotically stable undesired equilibria can exist for all nominal stabilizing controllers.

\begin{example}\longthmtitle{Asymptotically stable equilibria can exist for all linear nominal controllers}\label{ex:asymp-eq-exist-all-nominal-controllers}
{\rm
    Consider a linear underactuated system:
    \begin{align*}
        &A = \begin{bmatrix}
            a_{11} & a_{12} \\
            a_{21} & a_{22}
        \end{bmatrix}, \ 
            B = \begin{bmatrix}
            b_1 \\ b_2
        \end{bmatrix}, \ 
            K = \begin{bmatrix}
            k_1 \\
            k_2
        \end{bmatrix}.
    \end{align*}
    Given $c_1 > 0$, $c_2 > 0$, $r_1 > 0$, $r = c_1 - r_1$ and $R = c_1 + r_1$, suppose that $R < c_2$ and $r < c_1$. 
    Consider also the sets
    \begin{align*}
        \Oc_1 &:= \setdef{(x_1,x_2)\in\real^2}{ h_1(x):=-(x_1+c_1)^2 - (x_2-c_2)^2 + r_1^2 > 0, \quad h_2(x) = x_2 - c_2 > 0 }, \\
        \Oc_2 &:= \setdef{(x_1,x_2)\in\real^2}{ h_3(x):=-x_1^2 - (x_2-c_2)^2 + (c_1+r_1)^2 > 0, \\
        &\quad \quad \quad h_4(x) = -x_1^2 - (x_2-c_2)^2 + (c_1-r_1)^2 < 0, \quad h_5(x) = c_2-x_2 > 0}, \\
        \Oc_3 &:= \setdef{(x_1,x_2)\in\real^2}{ h_6(x) = r_1^2 -(x_1-c_1)^2-x_2^2 > 0, \quad h_7(x) = x_2-c_2 > 0 },
    \end{align*}
    and define the safe set $\Cc = \real^2 \backslash (\Oc_1\cup\Oc_2\cup\Oc_3)$, and $h:\real^2\to\real$ be such that $\Cc = \setdef{x\in\real^2}{h(x) \geq 0}$, $\partial\Cc = \setdef{x\in\real^2}{h(x) = 0}$.
    By letting $a_{11} = -0.169$, $a_{21} = -1.989$, $a_{12} = -3$, $a_{22} = -1.4$, $b_1 = -2.355$, $b_2 = -1.707$, $c_1 = 6.1$, $c_2 = 10.027$, $r_1 = 2.129$, it follows that $\bx_* = (3.157, 7.619)$ is an asymptotically stable undesired equilibrium for any nominal stabilizing controller $\bu = -K\bx$.
    We provide a more detailed justification about this in Example~\ref{ex:asymp-eq-exist-all-nominal-controllers-continued} of the Appendix.
    We note also that our argument does not preclude the existence of nonlinear nominal stabilizing controllers for which there does not exist any asymptotically stable undesired equilibrium.
    \begin{figure}[htb]
      \centering 
      {\includegraphics[width=0.32\textwidth]{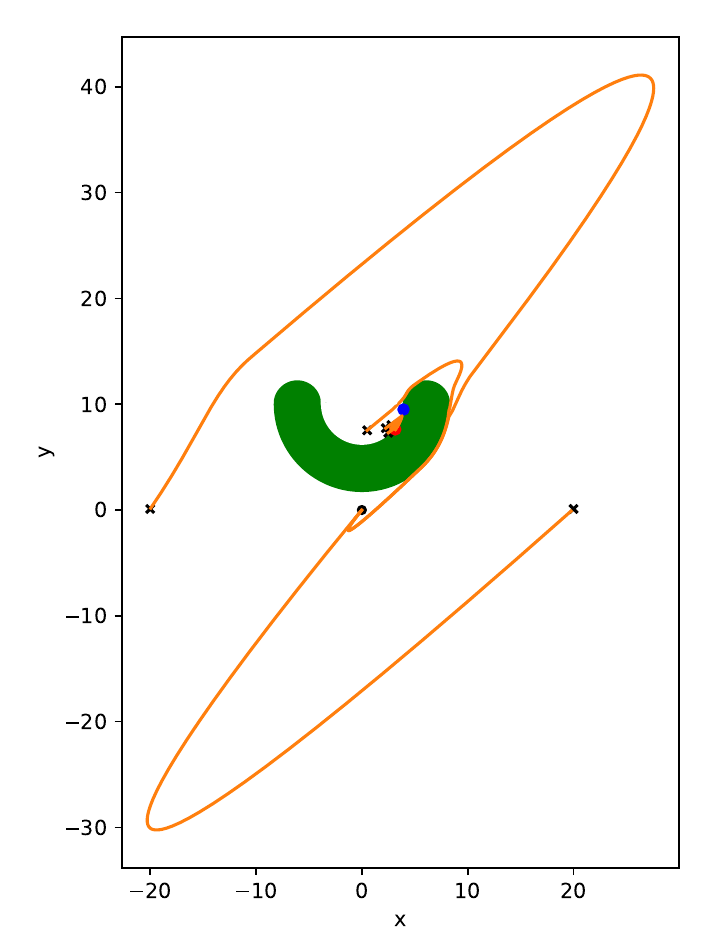}} 
      {\includegraphics[width=0.63\textwidth]{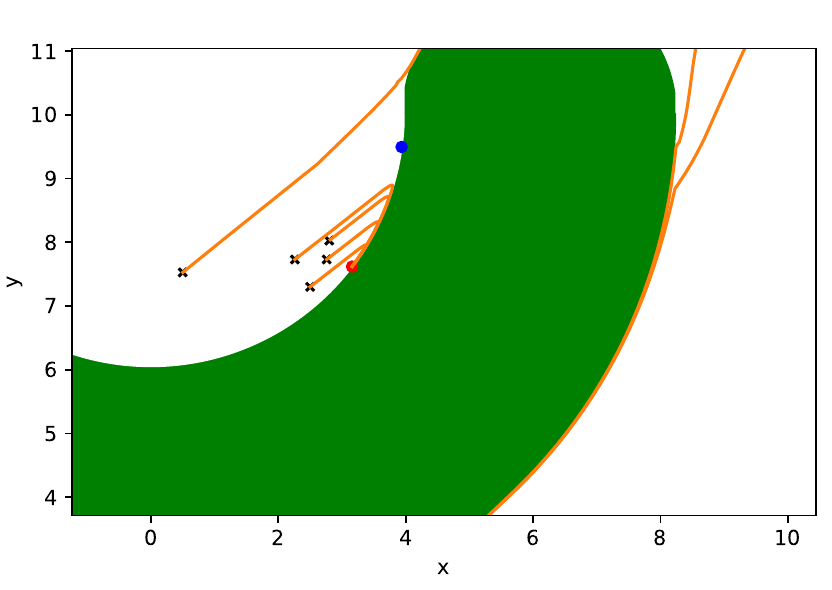}} 
      \caption{Plot of different trajectories of~\eqref{eq:v-linear-system} for the underactuated system in Example~\ref{ex:nonconvex-obstacle-asymptotically-stable-undesired-eq}. 
      The unsafe set is colored in green. Trajectories are depicted in orange.
      Black crosses denote initial conditions, the black dot denotes the origin, the red dot denotes an asymptotically stable undesired equilibrium, and the blue dot denotes an undesired equilibrium that is a saddle point. The plot on the left is a zoomed-in version of the plot on the right.}
      \label{fig:nonconvex-unioncircles}
    \end{figure}
    \problemfinal
    }
\end{example}

\subsection{Ellipsoidal unsafe sets}\label{sec:ellipsoidal-unsafe-sets}

Despite the title of this section, in the following we focus on studying the dynamical properties of safety filters for LTI systems and \emph{circular} obstacles. Accordingly, we consider the circular unsafe set:
\begin{align*}
    \Cc = \setdef{ \bx \in\real^2}{ h(\bx) = \norm{\bx-\bx_c}^2 - r^2 \geq 0},
\end{align*}
with $\bx_c = (x_{c,1}, x_{c,2}) \in\real^2$ the center. This is justified by the following result, which shows that the undesired equilibria of~\eqref{eq:v-linear-system} and their stability properties for general ellipsoidal obstacles are equivalent to those of a system with circular obstacles.

\smallskip

\begin{proposition}\longthmtitle{Safety filters with ellipsoidal and circular obstacles have the same dynamical properties}\label{prop:ellipsoidal-circular-same-properties}
    Let $\bx_c\in\real^2$, $P\in\real^{2 \times 2}$ positive definite, $h(\bx) = (\bx-\bx_c)^T P (\bx-\bx_c) - 1$, and $\Cc:=\setdef{\bx\in\real^n}{h(\bx)\geq0}$. 
    Suppose that $h$ is a strict CBF and Assumption~\ref{as: interior eq} holds.
    Further suppose that $P=E^T E$, with $E\in\real^{2\times 2}$ also positive definite, and define $\hat{\bx}_c = E \bx_c$, 
    $\hat{h}(\hat{\bx}) = (\hat{\bx}-\hat{\bx}_c)^T (\hat{\bx}-\hat{\bx}_c) - 1$ and $\hat{\Cc} = \setdef{\bx\in\real^n}{\hat{h}(\bx)\geq0}$. Moreover, let $\hat{A} = EAE^{-1}$, $\hat{B}=E B$, $\hat{G}(\hat{\bx}) = G(E^{-1}\hat{\bx})$ and $\hat{\eta}(\hat{\bx}) = \nabla \hat{h}(\hat{\bx})^T (\hat{A}-\hat{B}K E^{-1})\hat{\bx}+\alpha(\hat{h}(\hat{\bx}))$.
    Consider the system
    \begin{align}\label{eq:vhat-linear-system}
        \dot{\hat{\bx}} = \hat{F}(\hat{\bx})  :=(\hat{A}-\hat{B}KE^{-1})\hat{\bx} +  \hat{B} \hat{v}(\hat{\bx}),
    \end{align}
    where 
    \begin{align}\label{eq:vhat-linear-expression}
    \hat{v}(\hat{\bx}) = \begin{cases}
                0, &\ \text{if} \ \hat{\eta}(\hat{\bx}) \geq 0, \\
                -\frac{\hat{\eta}(\hat{\bx}) \hat{G}(\hat{\bx})^{-1}(\hat{\bx}) \hat{B}^T \nabla \hat{h}(\hat{\bx}) }{ \norm{
            \hat{B}^T \nabla \hat{h}(\hat{\bx})}_{\hat{G}(\hat{\bx})^{-1}}^2 }, &\ \text{if} \ \hat{\eta}(\hat{\bx}) < 0
            \end{cases}
    \end{align}
   
    Then, 
    \begin{enumerate}
        \item\label{it:ellip-first} $\hat{\Cc}$ is forward invariant under system~\eqref{eq:vhat-linear-system} and $\Cc$ is forward invariant under system~\eqref{eq:v-linear-system};
        \item\label{it:ellip-lipschitz} system~\eqref{eq:vhat-linear-system} is locally Lipschitz and system~\eqref{eq:v-linear-system} is locally Lipschitz;
        \item\label{it:ellip-second} $(A,B)$ is stabilizable if and only if $(\hat{A},\hat{B})$ is stabilizable;
        \item\label{it:ellip-third} $\hat{\bx}_*\in\real^2$ is an undesired equilibrium of~\eqref{eq:vhat-linear-system} if and only if $\bx_* := E^{-1} \hat{\bx}_*$ is an undesired equilibrium of~\eqref{eq:v-linear-system};
        \item\label{it:ellip-fourth} 
        the Jacobian of $\hat{F}$ at $\hat{\bx}_*$ and the Jacobian of $F$ at $\bx_*$ are similar. 
    \end{enumerate}
\end{proposition}
\begin{proof}
Re: \emph{\ref{it:ellip-first}}, by construction, system~\eqref{eq:vhat-linear-system} satisfies
    $\nabla \hat{h}(\hat{\bx})^T \hat{F}(\hat{\bx}) + \alpha(\hat{h}(\hat{\bx})) \geq 0$ for all $\hat{\bx}\in\hat{\Cc}$; similarly, system~\eqref{eq:v-linear-system} satisfies $\nabla h(\bx)^T F(\bx) + \alpha(h(\bx)) \geq 0$ for all $\bx\in\Cc$. These two conditions imply, respectively, that $\hat{\Cc}$ is forward invariant under system~\eqref{eq:vhat-linear-system} and $\Cc$ is forward invariant under system~\eqref{eq:v-linear-system}~\cite[Theorem~2]{ADA-SC-ME-GN-KS-PT:19}.

    Re: \emph{\ref{it:ellip-lipschitz}}, let us prove that for all $\hat{\bx}\in\hat{\Cc}$, there exists $\hat{\bu}\in\real^m$ such that $\nabla \hat{h}(\hat{\bx})^T ( \hat{A}\hat{\bx}+\hat{B}\bu ) + \alpha(h(\hat{\bx})) > 0$, which implies by~\cite[Lemma III.2]{MA-NA-JC:25-tac} that~\eqref{eq:vhat-linear-system} is locally Lipschitz.
    Indeed, suppose that $(EB)^\top (\hat{\bx}-\hat{\bx}_c)=0$. 
    Since for any $\hat{\bx}\in\real^n$, there exists $\bx\in\real^n$ such that $\hat{\bx} = E\bx$ and $h$ is a strict CBF, we have 
    $0=(EB)^\top (\hat{\bx}-\hat{\bx}_c)=B^\top E^\top E(\bx-\bx_c)=BP(\bx-\bx_c)=0$, which means that $(\hat{\bx}-\hat{\bx}_c)^\top  E A E^{-1} \hat{\bx}=(\bx-\bx_c)^\top E^\top E A\bx=  (\bx-\bx_c)^\top P^\top A\bx>0$. Hence for all $\hat{\bx}\in\hat{\Cc}$, there exists $\hat{\bu}\in\real^m$ such that $\nabla \hat{h}(\hat{\bx})^T ( \hat{A}\hat{\bx}+\hat{B}\bu ) + \alpha(h(\hat{\bx})) > 0$, from which it follows that~\eqref{eq:vhat-linear-system} is locally Lipschitz.

    Re: \emph{\ref{it:ellip-second}}, this follows from the observation that if $A-BK$ is Hurwitz then $\hat{A}-\hat{B}K E^{-1} = E(A-BK) E^{-1}$ is also Hurwitz. 
    This is because $E$ is nonsingular and similar matrices have the same eigenvalues~\cite[Corollary 1.3.4]{RAH-CRJ:12}.
    %
    
    Re: \emph{\ref{it:ellip-third}}, note that $\hat{\bx}_*$ satisfies the conditions in Lemma~\ref{lem:undesired-eq-characterization} for~\eqref{eq:vhat-linear-system} if and only if $\bx_*$ satisfies the conditions in Lemma~\ref{lem:undesired-eq-characterization} for~\eqref{eq:v-linear-system}.

    Re: \emph{\ref{it:ellip-fourth}}, we note that $F(E^{-1}\hat{\bx}) = E^{-1} \hat{F}(\hat{\bx})$ for any $\hat{\bx}\in\real^2$.
    Since the safety filter is active at undesired equilibria, $\eta(\bx_*)<0$.
    Now, let $J$ be the Jacobian of~\eqref{eq:v-linear-system} at $\bx_*$, and let $\hat{J}$ be the Jacobian of~\eqref{eq:vhat-linear-system} at $\hat{\bx}_*$.
    By the chain rule, $\hat{J} = E J E^{-1}$, which implies that $J$ and $\hat{J}$ are similar.
\end{proof}

\subsubsection{Underactuated LTI planar systems}

In the under-actuated case, we write
\begin{align}\label{eq:linear-2d-underactuated-system}
  A=\begin{bmatrix}
        a_{11} & a_{12} \\
        a_{21} & a_{22}
    \end{bmatrix},~
    B=
    \begin{bmatrix}
        b_1 \\
        b_2
\end{bmatrix},~K=\begin{bmatrix}
        k_{1} \\ k_{2}
    \end{bmatrix}^T.
\end{align}
Throughout the section, we let $\beta=a_{11}b_2 - b_1 a_{21}$, $\gamma = a_{22} b_1 - b_2 a_{12}$, and $T_3 = -\gamma x_{c,2} + \beta x_{c,1}$. We note also that since in this case $G$ is a scalar,~\eqref{eq:v-linear-system} is independent of~$G$. The following results give conditions on $h$ and system~\eqref{eq:linear-2d-underactuated-system} that ensure that Assumption~\ref{as: interior eq} holds and $h$ is a strict CBF.

\smallskip

\begin{lemma}[Conditions for Assumption~\ref{as: interior eq}]\label{lem:interior}
    Assumption~\ref{as: interior eq} holds if and only if
$\norm{\bx_c}^2 > r^2$.
\end{lemma}
%
%

\smallskip

The proof of this result is straightforward. 

\smallskip

\begin{proposition}[Conditions for $h$ to be a strict CBF]\label{prop:condition-h-strict-cbf-underactuated-lti-planar}
\label{prop:cbf}
    Let $\alpha_0>0$, $T_1 := b_2 \beta + b_1 \gamma + \frac{1}{2} \alpha_0 (b_2^2 + b_1^2)$, and $T_2 := (\beta x_{c,1} - \gamma x_{c,2})^2 + 2\alpha_0 r^2 T_1$.
    Suppose that $r>0$, $b_1^2 + b_2^2 > 0$, $T_1>0$, and
    \begin{align*}
        \frac{r}{ \sqrt{b_2^2 + b_1^2} } > \frac{  |T_3| + \sqrt{ T_2 } }{ 2 T_1 }.
    \end{align*}
    Then, $h$ is a strict CBF with the linear extended class $\Kc_{\infty}$ function $\alpha(s)=\alpha_0 s$.
\end{proposition}
\begin{proof}
    We need to ensure that all $\bx\in\real^2$ such that $h(\bx)\geq0$ and $B^T (\bx-\bx_c)=0$, satisfy 
    $2(\bx-\bx_c)A\bx + \alpha(h(\bx)) > 0$.
    First suppose $b_1\neq0$.
    Equivalently, we need to ensure that
    \begin{align}\label{eq:quadratic-equation-cbf}
        \notag
        &(x_2-x_{c,2})^2 \Big( (a_{11}+\frac{\alpha_0}{2})\frac{b_2^2}{b_1^2}-\frac{b_2}{b_1}(a_{12}+a_{21})+ a_{22} + \frac{\alpha_0}{2}  \Big) \\
        &+\!(x_2-x_{c,2}) \Big( \!  a_{22}x_{c,2}-\frac{b_2}{b_1}a_{11}x_{c,1} \! - \frac{b_2}{b_1}a_{12}x_{c,2} \! + \! a_{21}x_{c,1}   \Big) -\frac{1}{2}\alpha_0 r^2 \! > \! 0
    \end{align}
    whenever $(x_2-x_{c,2})^2 \geq  r^2 / ( (b_2^2/b_1^2) + 1)$.
    This follows by assumption. The condition $T_1>0$ ensures that the coefficient of $x_2-x_{c,2}$ of~\eqref{eq:quadratic-equation-cbf} is positive, and the condition $T_2>0$ ensures that the discriminant of~\eqref{eq:quadratic-equation-cbf} is positive. Now, by calculating the roots of the quadratic equation in $x_2-x_{c,2}$ we observe that the rest of the conditions in the statement ensure that~\eqref{eq:quadratic-equation-cbf} holds whenever $(x_2-x_{c,2})^2 \geq  r^2 /((b_2^2/b_1^2) + 1)$. 
    The case $b_1=0$ follows by a similar argument.
\end{proof}
%
%

\smallskip

The following result shows that for circular obstacles and linear planar underactuated systems,~\eqref{eq: condition-eq} can be solved explicitly, characterizing the undesired equilibria of the closed-loop system.

\smallskip

\begin{proposition}\longthmtitle{Equilibria in Underactuated Systems with Circular Obstacles}\label{prop:undesired-eq-n-2}
    Suppose that Assumptions \ref{as: interior eq}, \ref{as:A-B-stabilizable} and the conditions in Proposition~\ref{prop:cbf} hold. Define  $\bx_{*,+} := ( \gamma z_{+}, \beta z_{+} )$ and $\bx_{*,-} := ( \gamma z_{-}, \beta z_{-} )$, where
    \begin{align*}
        &z_{+} = \frac{ \gamma x_{c,1}+ \beta x_{c,2} + \sqrt{ r^2 (\gamma^2 + \beta^2) - T_3^2 } }{ \gamma^2 + \beta^2 }, \\
        &z_{-} = \frac{ \gamma x_{c,1}+ \beta x_{c,2} - \sqrt{ r^2 (\gamma^2 + \beta^2) - T_3^2 } }{ \gamma^2 + \beta^2 }.
    \end{align*}
    Then, $\gamma x_{c,1} + \beta x_{c,2} \neq 0$, and
    \begin{enumerate}
        \item if $\gamma x_{c,1} + \beta x_{c,2} < 0$, $\hat{\Ec} = \{ \bx_{*,+} \}$ is the only undesired equilibrium of the closed-loop system~\eqref{eq:v-linear-system}. Moreover, $\bx_{*,+}$ is a saddle point;
        \item if $\gamma x_{c,1} + \beta x_{c,2} > 0$, $\hat{\Ec} = \{ \bx_{*,-} \}$ is the only undesired equilibrium of the closed-loop system~\eqref{eq:v-linear-system}. Moreover, $\bx_{*,-}$ is a saddle point.
    \end{enumerate}
\end{proposition}
%
%
\begin{proof}
    The fact that $\gamma x_{c,1}+\beta x_{c,2} \neq 0$ is shown in Lemma~\ref{lem:discriminant-is-positive-if-cbf}. The same result also implies that the expressions for $\bx_{*,+}$ and $\bx_{*,-}$ are well defined (note that if $\gamma^2 + \beta^2 = 0$, then Assumption \ref{as:A-B-stabilizable} would not hold).
    Moreover, it follows from Lemma~\ref{lem:undesired-eq-characterization}
    that $\Ec = \{ \bx_{*,+},\bx_{*,-} \}$ 
    is the set of potential undesired equilibria for system~\eqref{eq:v-linear-system} when the system data is of the form~\eqref{eq:linear-2d-underactuated-system}.
    In order to ensure that $\bx_{*,+}$ is an undesired equilibrium of the closed-loop system, the condition $(\bx-\bx_c)^T (A-BK)\bx \rvert_{\bx = \bx_{*,+}} < 0$
    should hold.
    By using the expression of $\bx_{*,+}$, the condition is equivalent to
    \begin{align}\label{eq:extra-condition-pplus-2}
        z_{+} T_4 \Big( b_1(\gamma z_{+}-x_{c,1}) + b_2(\beta z_{+}-x_{c,2}) \Big) < 0,
    \end{align}
    where $T_4 = a_{11}a_{22}-a_{12}a_{21} - k_1\gamma - k_2\beta$.
    Since $A-BK$ is Hurwitz, $a_{11}a_{22} - a_{12}a_{21} - \gamma k_1 - \beta k_2 > 0$.
    This implies that $T_4 > 0$ and therefore~\eqref{eq:extra-condition-pplus-2} is equivalent to
    \begin{align}\label{eq:extra-condition-pplus-3}
        z_{+}(b_1 (\gamma z_{+}-x_{c,1}) + b_2 (\beta z_{+}-x_{c,2}) ) < 0.
    \end{align}
    Now, let us show that $b_1 (\gamma z_{+}-x_{c,1}) + b_2 (\beta z_{+}-x_{c,2}) > 0$. Indeed, this is equivalent to
    \begin{align*}
        T_3(\gamma b_2 - \beta b_1) + (\gamma b_1 + \beta b_2) \sqrt{ r^2 (\gamma^2 + \beta^2) - T_3^2 } > 0,
    \end{align*}
    and since $\gamma b_1 + \beta b_2 > 0$ as argued in the proof of Lemma~\ref{lem:discriminant-is-positive-if-cbf},
    this could only not hold if $T_3(\gamma b_2 - \beta b_1) < 0$ and $(\gamma b_1 + \beta b_2)^2 ( r^2 (\gamma^2 + \beta^2) - T_3^2 ) \leq T_3^2 (\gamma b_2 - \beta b_1)^2$. However, this last condition is in contradiction with~\eqref{eq:conditions-T3} in the proof of Lemma~\ref{lem:discriminant-is-positive-if-cbf}.
    Therefore,~\eqref{eq:extra-condition-pplus-3} holds if and only if $z_{+} < 0$, which is equivalent to: $\gamma x_{c,1} + \beta x_{c,2} < 0$ and
    \begin{align}\label{eq:last-conditions}
        & |\gamma x_{c,1} + \beta x_{c,2}| > \sqrt{r^2 (\gamma^2 + \beta^2) - T_3^2}.
    \end{align}
    Note that since $r^2 (\gamma^2 + \beta^2) -T_3^2 = (x_{c,1} \gamma + x_{c,2} \beta)^2 - (\gamma^2 + \beta^2)(x_{c,1}^2 + x_{c,2}^2 - r^2) < (x_{c,1} \gamma + x_{c,2} \beta)^2$ (where in the last inequality we have used the fact that $x_{c,1}^2 + x_{c,2}^2 > r^2$), it follows that the last of the inequalities in~\eqref{eq:last-conditions} always holds.
    Therefore, $\bx_{*,+}$ is an undesired equilibrium of the closed-loop system if and only if $\gamma x_{c,1} + \beta x_{c,2} < 0$.
    The fact that $\bx_{*,+}$ is the unique undesired equilibrium and is a saddle point follows from Proposition~\ref{prop:number-of-undes-equilibria-stab-properties-finite-set-of-bounded-obstacles}. An analogous argument shows that $\bx_{*,-}$ is the unique undesired equilibrium if and only if $\gamma x_{c,1} + \beta x_{c,2} > 0$, in which case it is a saddle point.
\end{proof}

\smallskip

Note that the statement of Proposition~\ref{prop:undesired-eq-n-2} is consistent with Proposition~\ref{prop:number-of-undes-equilibria-stab-properties-finite-set-of-bounded-obstacles}, since it also states that, provided that all the undesired equilibria are not degenerate, their number is odd.

\smallskip

\begin{remark}\longthmtitle{Almost global asymptotic stability}\label{rem:almost-global-asymptotic-stability}
{\rm
The Stable Manifold Theorem~\cite[Ch. 2.7]{LP:00} ensures that if $\bx_*$ is a saddle point in $\real^2$, the local stable manifold is $1$-dimensional.
Therefore, it has measure zero. 
Moreover, the global stable manifold
must also have measure zero.
If this were not the case, solutions would have to intersect. However this is not possible 
due to the uniqueness of solutions. Hence the global stable manifold is exactly equal to $\setdef{\bx_0\in\real^n}{\lim\limits_{t\to\infty}\bx(t;\bx_0) = \mathbf{0}_n}$.
It follows that the set of initial conditions whose associated trajectory converges to $\bx_*$ has measure zero. Hence, by appropriately tuning the class $\Kc_{\infty}$ function to rule out limit cycles (cf. Proposition~\ref{prop:no-limit-cycles-general}), the Poincar\'e-Bendixson Theorem~\cite[Chapter 7, Thm. 4.1]{PH:02} implies that the origin is almost globally asymptotically stable (i.e., asymptotically stable with a region of attraction equal to $\real^2$ minus a set of measure zero).
\hfill $\Box$}
\end{remark}
%
%

\smallskip

\subsubsection{Fully actuated LTI planar systems}\label{sec:ellips-fully-actuated-LTI}

Here we consider the system~\eqref{eq:v-linear-system} and assume that $n=2$, $m=2$,
and $B\in\real^{2 \times 2}$ is invertible; in this case $h$ is a strict CBF and Assumption~\ref{as:A-B-stabilizable} is satisfied. Throughout the section, $\tilde{A}=A-BK$.
The following result summarizes the different possible undesired equilibria of~\eqref{eq:v-linear-system} in the special case where $\bx_c$, the center of the circular unsafe set, is an eigenvector of $\tilde{A}$.

\smallskip


\begin{proposition}
\longthmtitle{Characterization of undesired equilibria}\label{prop: m=2, x0 is  eigen}
    Suppose Assumption~\ref{as: interior eq} is satisfied, $B$ is invertible, and $\tilde{A}$ is Hurwitz. 
    Suppose also that $\bx_c$ is an eigenvector of $\tilde{A}$. Then one of the following is true:
    %
    %
      \begin{enumerate}
        \item  $|\mathcal{E}|=2$, $|\hat{\mathcal{E}}|=1$,  and
        $\hat{\mathcal{E}}$ consists of a degenerate equilibrium;
        %
        %

        \item $|\mathcal{E}|=2$, $|\hat{\mathcal{E}}|=1$, and 
        $\hat{\mathcal{E}}$ consists of a saddle point;

        \item  $|\mathcal{E}|=3$, $|\hat{\mathcal{E}}|=2$, and
        $\hat{\mathcal{E}}$ consists of  a saddle point and a degenerate equilibrium;
        %
        %
        \item $|\mathcal{E}|=4$, $|\hat{\mathcal{E}}|=3$, and
      $\hat{\mathcal{E}}$ consists of an asymptotically stable equilibrium and two saddle points.
    \end{enumerate}
\end{proposition}


The proof of Proposition~\ref{prop: m=2, x0 is  eigen} is provided in the Appendix. Table~\ref{tab:cases} outlines the majority of the cases discussed in the proof of Proposition~\ref{prop: m=2, x0 is  eigen}.
%
%
%
%
For different numerical examples illustrating the different cases outlined in Table~\ref{tab:cases}, we refer the reader to~\cite[Figure 1]{YC-PM-EDA-JC:24-cdc}.
\begin{table}[t!]
\vspace{-0.2cm}
    \centering
    \parbox[t]{.45\linewidth}{
    \begin{tabular}{|l|c|c|c|}
    \hline
    \rule{0pt}{3ex} $\tilde{A}$ diagonalizable,  &  \text{SP}  & \text{DE} & \text{ASE} \\ \hline 
    $ (\bv_i^T \bv_j)^2<   1-\frac{(\lambda_i-\lambda_j)^2 r^2}{\lambda_i^2\|\bx_c\|^2}$   &  1 & 0  &0\\
    $(\bv_i^T \bv_j)^2=   1-\frac{(\lambda_i-\lambda_j)^2 r^2}{\lambda_i^2\|\bx_c\|^2}$ &  1 & 1  &0\\
      $ (\bv_i^T \bv_j)^2>   1-\frac{(\lambda_i-\lambda_j)^2 r^2}{\lambda_i^2\|\bx_c\|^2}$   &  2 & 0  &1\\
      \hline
    \end{tabular}
    }
    \hfill 
    \parbox[t]{.45\linewidth}{
    \centering
    \begin{tabular}{|l|c|c|c|}
    \hline
     \rule{0pt}{3ex} $\tilde{A}$ non-diagonalizable &  \text{SP}  & \text{DE} & \text{ASE} \\  \hline
    \rule{0pt}{3.4ex} $(\bv_1^T \bv_2)^2<   1-\frac{ r^2}{{\lambda^2\|\bx_c\|^2}} $  &  1 & 0  &0 \\

   \rule{0pt}{3.4ex}  $(\bv_1^T \bv_2)^2=   1-\frac{ r^2}{{\lambda^2\|\bx_c\|^2}}$ &  1 & 1  &0\\

    \rule{0pt}{3.4ex}  $ (\bv_1^T \bv_2)^2>   1-\frac{ r^2}{{\lambda^2\|\bx_c\|^2}} $  &  2 & 0  &1\\
      \hline
    \end{tabular}
    }\vspace{0.1cm}
    (a)~~~~~~~~~~~~~~~~~~~~~~~~~~~~~~~~~~~~~~~~~~~~~~~~~~~~~~~~~~~~~~~~~~~~(b)
    \vspace{0.5cm}
    \caption{Characterization of undesired equilibria (SP: saddle point, DE: degenerate equilibrium, ASE: asymptotically stable equilibrium) when $\bx_c$ is an eigenvector of $\tilde{A}$. In (a), $\tilde{A}$ is diagonalizable, i.e., $\tilde{A}\bx_c=\lambda_i \bx_c$, $ \bv_i=\frac{\bx_c}{\|\bx_c\|}$, $\tilde{A} \bv_j=\lambda_j \bv_j$,  $\|\bv_j\|=1$, $i,j=\{1,2\}$, and $\{\bv_i, \bv_j\}$  linearly independent. We assume that $\frac{(\lambda_i-\lambda_j)^2 r^2}{\lambda_i^2\|\bx_c\|^2}\neq 1$. In (b), $\tilde{A}$ is not diagonalizable, i.e., $\tilde{A} \bv_2=\lambda \bv_2+ \bv_1$, $ \bv_1=\frac{\bx_c}{\|\bx_c\|}$, $\tilde{A}\bx_c=\lambda \bx_c$, and $\|\bv_2\|=1$.}\label{tab:cases}
\end{table}
%
%
We build on this result to show that the eigenvalues of $\tilde{A}$ do not determine the stability properties of undesired equilibria.

\smallskip

\begin{proposition}\longthmtitle{Spectrum of $\tilde{A}$ does not determine stability properties of undesired equilibria}\label{corollary: ambiguity}
  Suppose Assumption~\ref{as: interior eq} is satisfied and $B$ is invertible. Then for any given negative $\lambda_1$ and $\lambda_2$, there exists $K_1\in\real^{2\times 2}$ and $K_2\in\real^{2\times 2}$ in the set $\{K\in\real^{2\times 2}: \lambda_1,\lambda_2 = \textrm{spec}(A-BK) \}$,
  %
  %
  such that~\eqref{eq:v-linear-system} with $K = K_1$ has an undesired asymptotically stable equilibrium and~\eqref{eq:v-linear-system} with $K = K_2$, has a single undesired equilibrium, which is a saddle point.
\end{proposition}
\begin{proof}
First let us construct $K_1$.
Start by supposing that $\lambda_1 = \lambda_2$. Let $\bv_1=\frac{\bx_c}{\|\bx_c\|}$ and 
$\theta=\arccos \sqrt{Q_1},~Q_1:= \max \{0,1-\frac{r^2}{2\lambda_1^2 \|\bx_c\|}\}$. Then, $\bv_1^T \bv_2 = \cos(\theta) \norm{\bv_1}^2 = \cos(\theta)$. We let
\begin{align*}
    K_1:=B^{-1}\Big(A-\begin{bmatrix}
        \bv_1 & \bv_2
    \end{bmatrix} \begin{bmatrix}
        \lambda_1 & 1\\ 0 & \lambda_1
    \end{bmatrix} \begin{bmatrix}
        \bv_1 & \bv_2
    \end{bmatrix}^{-1}  \Big).
\end{align*}
By the third row of  Table~\ref{tab:cases}(b),~\eqref{eq:v-linear-system} has an undesired asymptotically stable equilibrium.
%
%
Next, suppose that $\lambda_1\neq\lambda_2$.
Let $\bv_1=\frac{\bx_c}{\|\bx_c\|}$, $\bv_2=\begin{bmatrix}
    \cos \theta & -\sin \theta\\
    \sin \theta & \cos \theta
\end{bmatrix}\bv_1$,  where $\theta=\arccos \sqrt{Q_2},~Q_2:= \max \{0,1-\frac{(\lambda_1-\lambda_2)^2 r^2}{2\lambda_1^2\|\bx_c\|^2}\}$. Then  $\bv_1^T \bv_2=\cos \theta \|\bv_1\|_2^2=\cos \theta$ and
\begin{align*}
    K_1:=B^{-1}\Big(A-\begin{bmatrix}
        \bv_1 & \bv_2
    \end{bmatrix} \begin{bmatrix}
        \lambda_1 & 0\\ 0 & \lambda_2
    \end{bmatrix} \begin{bmatrix}
        \bv_1 & \bv_2
    \end{bmatrix}^{-1}  \Big).
\end{align*}
By the third row of Table~\ref{tab:cases}(a), ~\eqref{eq:v-linear-system} has an undesired asymptotically stable equilibrium.

To construct  $K_2$, we assume without loss of generality that $\lambda_1\leq \lambda_2$.
Note that since both $\lambda_1$ and $\lambda_2$ are negative, 
$1-\frac{(\lambda_i-\lambda_j)^2 r^2}{\lambda_i^2\|\bx_c\|^2}>0$. Let  $\bv_i=\frac{\bx_c}{\|\bx_c\|}$, $\bv_j=\begin{bmatrix}
    0 & -1\\
    1 & 0
\end{bmatrix}\bv_i$, 
so that $\bv_i^T \bv_j = 0$
and define
\begin{align*}
    K_2:=B^{-1}\Big(A-\begin{bmatrix}
        \bv_i & \bv_j
    \end{bmatrix} \begin{bmatrix}
        \lambda_i & 0\\ 0 & \lambda_j
    \end{bmatrix} \begin{bmatrix}
        \bv_i & \bv_j
    \end{bmatrix}^{-1}  \Big).
\end{align*}
By the first row of Table~\ref{tab:cases}(a),~\eqref{eq:v-linear-system} has a single undesired equilibrium and it is a saddle point.
\end{proof}

\smallskip

Interestingly, even though one can characterize the global stability properties of the origin based on the eigenvalues of $\tilde{A}$ for the system without a safety filter, this is no longer the case
for the system with a safety filter.
%
%
On the other hand, as a consequence of Proposition~\ref{corollary: ambiguity}, we deduce that it is always possible to choose a nominal controller $\bu=-K\bx$ such that $\tilde{A}$ has negative eigenvalues and the set of trajectories of~\eqref{eq:v-linear-system} that do not converge to the origin has measure zero (cf. Remark~\ref{rem:almost-global-asymptotic-stability}).
Note that, as shown in~\cite[Proposition 10]{indep-cbf} and Table~\ref{tab:cases}, the extended class $\Kc_{\infty}$ function only affects the rate of decay in the stable manifold of the undesired equilibria and it does not affect the existence and stability of undesired equilibria. Therefore, the choice of nominal controller $\bu=-K\bx$ determines in which of the cases we fall into. Ideally, such nominal controller should be designed so that there exists only one undesired equilibrium and it is a saddle point.

We conclude this section studying the case when $\bx_c$ is not an eigenvector of $\tilde{A}$, which requires a more involved technical analysis.
The following result characterizes the number of undesired equilibria under appropriate sufficient conditions.

\smallskip

\begin{proposition}\longthmtitle{Number of undesired equilibria when $\bx_c$ is not an eigenvector}\label{prop: number-of equilibrium-xc-no-eigenvec}
 Suppose Assumption~\ref{as: interior eq} is satisfied, $B$ is invertible, $G(\bx)=B^\top B$, $\tilde{A}$ is Hurwitz and $\bx_c$ is not an eigenvector of $\tilde{A}$.
 Then  $ 1 \leq |\hat{\mathcal{E}}|\leq 3$ and $ |\mathcal{E}\setminus \hat{\mathcal{E}}|\geq 1$.
 In addition, if $\lambda_1\leq \lambda_2$, there exists $\bx_*\in\mathcal{\hat{E}}$ with indicator $\delta <\frac{\lambda_1}{2}$.
\end{proposition}
%
%

\vspace{.1cm}

The proof of Proposition~\ref{prop: number-of equilibrium-xc-no-eigenvec} is given in the Appendix.


The following result establishes the stability properties of undesired equilibria in the case where $\bx_c$ is not an eigenvector under some additional assumptions.

\smallskip

\begin{proposition}\longthmtitle{Stability properties of undesired equilibria when $\bx_c$ is not an eigenvector}\label{prop: m=2, x0 is not eigen, stab properties}
 Suppose Assumption~\ref{as: interior eq} is satisfied, $B$ is invertible, $G(\bx)=B^\top B$, $\tilde{A}$ is Hurwitz with two real eigenvalues $\lambda_1<\lambda_2$ and  $\bx_c$ is not an eigenvector of $\tilde{A}$.
 Then there is no undesired equilibrium with indicator $\delta\in\{\frac{\lambda_1}{2}, \frac{\lambda_2}{2}\}$ .
 Moreover, if $\bv_1$ and $\bv_2$ are the eigenvectors associated with $\lambda_1$ and $\lambda_2$, respectively, and $\bv_1^T \bv_2\geq 0$, $\|\bv_1\|=\|\bv_2\|=1$, and $\bx_c=\beta_1 \bv_1+\beta_2 \bv_2$, the following holds.
 \begin{itemize}
     \item[(i)] If $\beta_1^2+\beta_1 \beta_2 \bv_1^\top \bv_2  \geq  0$, then for any undesired equilibrium $\bx_*$ with indicator $\delta_{\bx_*}$ such that $\delta<\frac{\lambda_1}{2}$,  $\bx_*$ is a saddle point.

  \item[(ii)] If $\beta_1\beta_2
 \bv_2^\top \bv_1+\beta_2^2  \geq 0$, then for any undesired equilibrium $\bx_*$ with indicator $\delta_{\bx_*}$ such that $\frac{\lambda_2}{2}<\delta<0$,  $\bx_*$ is asymptotically stable.
     
     \item[(iii)] Define $F_1:\mathbb{R}\to \mathbb{R}$ as:
\begin{align} \label{eq: F_1}
    F_1(\delta)&:=-|\lambda_1-2\delta|^2 |\lambda_2-2\delta|^2 r^2+|\lambda_1\beta_1|^2 |\lambda_2-2\delta|^2+|\lambda_2\beta_2|^2 |\lambda_1-2\delta|^2
    \\
    &+\lambda_1^*\beta_1^*\lambda_2\beta_2(\lambda_2-2\delta)^*(\lambda_1-2\delta)  \bv_1^* \bv_2
    +\lambda_1\beta_1\lambda_2^*\beta_2^*(\lambda_2-2\delta)(\lambda_1-2\delta)^*  \bv_2^* \bv_1 . 
    \notag
\end{align}
If the third-order polynomial $\frac{d F_1(\delta)}{d \delta}$  has only one real root\footnote{
    For third-order polynomial $ax^3+bx^2+cx+d$, its discriminant is defined as $18 a b c d-4 b^3 d+b^2 c^2-4 a c^3-27 a^2 d^2$. If $a\neq 0$ and the discriminant is negative, the third-order polynomial only has one real root. } and $\beta_1^2+\beta_1 \beta_2 \bv_1^\top \bv_2\geq 0$, then there exists only one undesired equilibrium and it is a saddle point.
 \end{itemize}
\end{proposition}

\smallskip

The proof of Proposition~\ref{prop: m=2, x0 is not eigen, stab properties} is given in the Appendix.




\vspace{.1cm}



\section{Conclusions}\label{sec:conclusions}

This work has characterized different dynamical properties of CBF-based safety filters. We have
provided a characterization of the undesired equilibria in the corresponding closed-loop system
through the solution of an algebraic equation.
Next, we have shown that by appropriately designing the parameters of the safety filter, the closed-loop trajectories are bounded and in the planar case, no limit cycles exist.
We have shown through a counterexample that in dimension greater than $2$, limit cycles can exist and may not be removed by changing the parameters of the safety filter. 
Moreover, for general planar systems under general assumptions, we have characterized the parity of the number of undesired equilibria, as well as the number of such undesired equilibria that are saddle points.
Finally, for linear planar systems, we have shown that if the system is underactuated and the safe set is parametrizable in polar coordinates, undesired equilibria do not exist, but if any of these conditions fail, undesired equilibria may exist and might even be asymptotically stable.
In the special case of ellipsoidal obstacles, we have 
provided explicit expressions for the undesired equilibria, and characterized their stability properties.
Future work will focus on
%
%
characterizing more explicitly the regions of attraction of the origin and the different undesired equilibria by using other techniques like normal forms, utilizing Euler indices and other methods from algebraic topology to improve our understanding of undesired equilibria,
%
%
and performing a similar analysis for other CBF-based control designs from the literature, such as ones leveraging control Lyapunov functions.
%
%

%
%
\section{Acknowledgements}
This work was supported by AFOSR Award FA9550-23-1-0740.

\bibliography{bib/alias,bib/JC,bib/Main-add,bib/Main,bib/New,bib/references}

\section{Appendix}

\subsection{Auxiliary results for Section~\ref{sec:undesired-eq-bounded-trajs-limit-cycles}}

This section provides a number of supporting results for the technical treatment of Section~\ref{sec:undesired-eq-bounded-trajs-limit-cycles}.

%
%

\begin{lemma}\longthmtitle{Trajectories of globally asymptotically stable system}\label{lem:trajectories-of-gas-system}
    Let $F:\real^n\to\real^n$ be locally Lipschitz, and suppose that the origin is globally asymptotically stable for the differential equation $\dot{\bx}=F(\bx)$. Then, for any $\bx_0\in\real^n\backslash\{ \mathbf{0}_n \}$, it holds that $\lim\limits_{t\to-\infty}\norm{\bx(t;\bx_0)}=\infty$.
\end{lemma}
\begin{proof}
    Since $F$ is locally Lipschitz and the origin is globally asymptotically stable, by~\cite[Theorem 4.17]{HK:02} there exists a smooth, positive definite, and radially unbounded function $V:\real^n\to\real$ and a continuous positive definite function $W:\real^n\to\real$ 
    such that $\nabla V(\bx)^T f(\bx)\leq-W(\bx)$ for all $\bx\in\real^n$.
    Now, let $t\geq0$ and
    note that $\frac{d}{dt}V( \bx(-t;\bx_0) ) = -\nabla V( \bx(-t;\bx_0) )^\top f( \bx(-t;\bx_0) ) \geq W( \bx(-t;\bx_0) ) \geq 0$.
    Hence, $V( \bx(-t;\bx_0) ) \geq V( \bx_0 )$. By letting $\bar{w} = \min\limits_{ \setdef{ \by\in\real^n }{ V(\by) \geq V(\bx_0) } } W(\by)$, it follows that $\frac{d}{dt}V( \bx(-t;\bx_0) ) \geq \bar{w} > 0$ for all $t\geq0$.
    This implies that $\lim\limits_{t\to\infty} V( \bx(-t;\bx_0) ) = \infty$ and since $V$ is radially unbounded, the result follows.
\end{proof}

\begin{lemma}\longthmtitle{Convergence and tangency properties of stable manifold}\label{lem:stable-manif-convergence-not-tangent}
%
%
    Let $h$ be a strict CBF and suppose Assumption~\ref{as: interior eq} holds. Let  
    $\real^2\backslash\Cc$ be a bounded connected set. Consider~\eqref{eq:general-system-1} with $n=2$
    and let $\bx_*$ be an undesired equilibrium of~\eqref{eq:general-system-1} which is a saddle point.
    %
    %
    Let $\gamma$ be a subset of the one-dimensional local stable manifold of $\bx_*$
    such that $\gamma$ corresponds to a maximal trajectory of~\eqref{eq:general-system-1}
    and there exists $T>0$ with $\gamma(t)\in\text{Int}(\Cc)$ for all $t>T$.
    Then, 
    \begin{enumerate}
        \item\label{it:stab-manif-properties-first} $\gamma$ is not tangent to $\partial\Cc$ at any point;
        \item\label{it:stab-manif-properties-second} if $(\omega_{-},\infty)$ is the interval of definition of $\gamma$, then
        \begin{enumerate}
            \item\label{it:item-a} if $\omega_{-} = -\infty$, then either 
            $\lim\limits_{t\to -\infty}\norm{\gamma(t)} = \infty$,
            %
            %
            $\lim\limits_{t\to-\infty}\norm{\gamma(t)} = \bx_{*}^\prime$, with $\bx_{*}^\prime$ another equilibrium of~\eqref{eq:general-system-1},
            or $\lim\limits_{t\to-\infty}\norm{\gamma(t)}$ converges to a limit cycle;
            \item\label{it:item-b} if $\omega_{-} > -\infty$, then $\lim\limits_{t\to\omega_{-}}\gamma(t)\notin\Cc$.
        \end{enumerate}
    \end{enumerate}
\end{lemma}
\begin{proof}
    To show~\ref{it:stab-manif-properties-first}, we reason by contradiction. 
    Figure~\ref{fig:non-tangent-diagram} serves as visual aid for the different elements defined in the proof.
    Suppose that $\gamma$ is tangent to $\partial\Cc$ at a point~$\bar{\bq}$.
    Then,
    %
    %
    $\bar{\bq}=\gamma(t_{\bar{\bq}})$ for some $t_{\bar{\bq}} \in \real$, and there exists $T_{\bar{\bq}}>0$ small enough with $\gamma(t) \in \text{Int}(\Cc)$ for all $t\in ( t_{\bar{\bq}}-T_{\bar{\bq}}, t_{\bar{\bq}} )$.
    Now, note that there exists a sufficiently small neighborhood $\Nc_{\bar{\bq}}$ of $\bar{\bq}$ such that one of the arches $\zeta$
    between $\partial\Cc$ and $\gamma$ defined by $\Nc_{\bar{\bq}}$
    is such that all trajectories of~\eqref{eq:general-system-1} with initial condition in $\zeta$ stay between $\partial\Cc$ (because $\Cc$ is forward invariant) and $\gamma$ (because of uniqueness of solutions), and stay inside $\Nc_{\bar{\bq}}$ at all times (by continuity of the vector field~\eqref{eq:general-system-1}, $\gamma^\prime(t_{\bar{\bq}})\neq \mathbf{0}_n$, and because $\Nc_{\bar{\bq}}$ is sufficiently small).
    %
    %
    However, it is not possible that such trajectories stay 
    inside the region defined by $\zeta$, $\partial\Cc$ and $\gamma$ by the  Poincar\'e-Bendixson Theorem~\cite[Chapter 7, Thm. 4.1]{PH:02}, because $\bar{\bq}$ is not an equilibrium point and 
    the region defined by $\zeta$, $\partial\Cc$ and $\gamma$
    does not contain limit cycles, because such limit cycle would encircle only  points in the interior
    of $\Cc$~\cite[Corollary 6.26]{JM-book:07}.
    %
    %
    Hence, we have reached a contradiction, which means that $\gamma$ is not tangent to $\partial\Cc$ at any point.
    %
    %
    Let us now show~\ref{it:stab-manif-properties-second}.
    Part (a) follows
    by the Poincar\'e-Bendixson Theorem~\cite[Chapter 7, Thm. 4.1]{PH:02},
    whereas part (b) follows from the fact that if $\lim\limits_{t\to\omega_{-}}\gamma(t)\in\Cc$,
    since~\eqref{eq:general-system-1} is well-defined at all points in the safe set,
    then $\gamma$ is well-defined as $t\to\omega_{-}$, and the interval of definition of $\gamma$ can be increased, contradicting the assumption that $\gamma$ is a maximal solution.
    %
    %
\end{proof}

\begin{figure}[htb]
  \centering    {\includegraphics[width=0.5\textwidth]{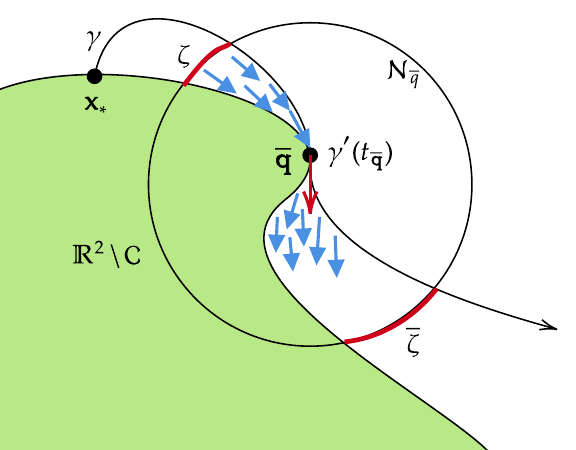}} 
  \caption{Sketch of the setting considered in the proof of Lemma~\ref{lem:stable-manif-convergence-not-tangent}. The unsafe set is depicted in green, whereas the closed-loop vector field at the point of tangency is the red arrow, the blue arrows denote the vector field elsewhere, and the arches $\zeta$ and $\bar{\zeta}$ are depicted in red.}\label{fig:non-tangent-diagram}
\end{figure}

\begin{lemma}\longthmtitle{Existence of stable manifold exiting set with no limit cycles}\label{lem:existence-stable-manifold-exiting-set-no-limit-cycles}
 Let $h$ be a strict CBF and suppose Assumption~\ref{as: interior eq} holds. Let  
    $\real^2\backslash\Cc$ be a bounded connected set. Consider~\eqref{eq:general-system-1} with $n=2$ and let  $\{ \bx_*^{(i)} \}_{i=1}^{k} \subset \partial\Cc$ 
   be its set of undesired equilibria. Denote by  $\Lc\subset[k]$ the index set of undesired equilibria that are saddle points and let $\tilde{\Phi}$ be a compact forward invariant set containing the origin and $\real^2\backslash\Cc$ and such that $\tilde{\Phi}\cap\Cc$ does not contain any limit cycles. 
   %
   %
   %
   For each $j\in\Lc$, let $\gamma_j$ be a subset of the one-dimensional local stable manifold of~\eqref{eq:general-system-1} at $\bx_*^{(j)}$ such that $\gamma_j$ corresponds to a maximal trajectory and there exists $T>0$ with $\gamma_j(t)\in\text{Int}(\Cc)$ for all $t>T$.
   Then, there exists at least one $j\in\Lc$ and $\tilde{t}_j\in\real$ such that $\gamma_j(\tilde{t}_j)\in\partial\tilde{\Phi}$ and $\gamma_j(t)\notin\tilde{\Phi}$ for all $t<t_j$.
\end{lemma}
\begin{proof}
    Let us reason by contradiction.
    If the statement does not hold, then using Lemma~\ref{lem:stable-manif-convergence-not-tangent}(ii),
    we deduce that for all $j\in\Lc$,
    one of the following holds:
    \begin{enumerate}
        \item\label{it:cases-first} there exists $t_j>0$ such that $\lim\limits_{t\to-t_j}\gamma_j(t) \in \real^2\backslash\Cc$;
        \item\label{it:cases-second} $\gamma_j(t)$ converges to a limit cycle in $\tilde{\Phi}\cap\Cc$;
        \item\label{it:cases-third} $\lim\limits_{t\to-\infty}\gamma_j(t) = \bx_*^{(\bar{j})} $ for some $\bar{j}\in[k]$.
    \end{enumerate}
    Since $\tilde{\Phi}\cap\Cc$ does not contain any limit cycles, case~\ref{it:cases-second} is not possible. If for all $j\in\Lc$, either~\ref{it:cases-first} or~\ref{it:cases-third} hold,
    there exists a compact set $\tilde{\Oc}$, whose boundary is comprised of
    $\partial(\real^2\backslash\Cc)$ and the union of the trajectories $\gamma_j$ for all $j\in\Lc$.
    %
    %
    %
    If all trajectories with initial condition in $\tilde{\Phi}\backslash\tilde{\Oc}$ converge to the origin, that contradicts~\cite[Proposition 3]{DEK:87-icra}, which shows that there cannot exist a continuous dynamical system, forward invariant in a set whose complement is compact, and with such set being the region of attraction of an asymptotically stable equilibrium.
    Note also that if $\tilde{\Phi}\backslash\tilde{\Oc}$ contains trajectories belonging to the regions of attraction of different asymptotically stable equilibria, by~\cite[Theorem 8.1]{HK:02}, there exists a trajectory of~\eqref{eq:general-system-1} in the boundary of these regions of attraction.
    By definition, this trajectory does not belong to any region of attraction and by the Poincar\'e-Bendixson Theorem~\cite[Chapter 7, Thm. 4.1]{PH:02}, this trajectory can only converge to a limit cycle. However, $\tilde{\Phi}$ is forward invariant and $\tilde{\Phi}\cap\Cc$ does not contain any limit cycle, hence reaching a contradiction.
\end{proof}

\subsection{Auxiliary results for Section~\ref{sec:dynamical-properties-safety-filters}}

This section provides a number of supporting results for the technical treatment of Section~\ref{sec:dynamical-properties-safety-filters}.
We start with an auxiliary result used in the proof of Proposition~\ref{prop:safety-filters-work-for-safe-sets-star-shaped}.

\smallskip
\begin{lemma}\longthmtitle{If a strict CBF exists, all CBFs are strict}\label{prop:if-one-strict-cbf-exists-all-cbfs-are-strict}
    Let $\Cc$ be a compact 
    set and assume  $h$ is a strict CBF of $\Cc$. Then any other CBF $\tilde{h}$ of $\Cc$ is also strict.
\end{lemma}
\begin{proof}
    By~\cite[Lemma 2.2]{indep-cbf}, there exists a function $\zeta:\partial\Cc\to\real_{>0}$
    %
    %
    such that 
    $\nabla \tilde{h}(\bx) = \zeta(\bx) \nabla h(\bx)$ for all $\bx$ in $\Cc$.
    Since $h$ is strict, for all $\bx\in\partial\Cc$, there exist $\bu_{\bx}\in\real^m$ such that 
    $\nabla h(\bx)^\top (f(\bx)+g(\bx)\bu_{\bx}) > 0$.
    This implies that 
    $\nabla \tilde{h}(\bx)^\top (f(\bx)+g(\bx)\bu_{\bx}) = \zeta(\bx) \nabla h(\bx)^\top (f(\bx)+g(\bx)\bu_{\bx}) > 0$ for all $\bx\in\partial\Cc$.
    %
    %
    Now, since $\nabla\tilde{h}$, $f$, and $g$ are continuous, there exists a neighborhood $\Nc_{\bx}$ of each $\bx\in\partial\Cc$  such that $\nabla\tilde{h}(\by)^\top (f(\by)+g(\by)\bu_{\bx} )> 0$ for all $\by\in\Nc_{\bx}$. Therefore, there exists a neighborhood $\Nc$ of $\partial\Cc$ where the CBF condition for $\tilde{h}$ is strictly feasible.
    %
    %
    Now, since $\Cc$ is compact, we can choose $\alpha$ as a linear function with a sufficiently large slope to ensure that 
    $\nabla \tilde{h}(\bx)^\top (f(\bx)+g(\bx)\bu_{\bx}) + \alpha( \tilde{h}(\bx) ) > 0$ holds for all $\bx\in\Cc\backslash\Nc$
    %
    %
    and hence for all $\bx\in\Cc$, making
    $\tilde{h}$ a strict CBF.
\end{proof}

We next give a technical result used in the proof of Proposition~\ref{prop:undesired-eq-n-2}.  We use the same notation.

\smallskip

\begin{lemma}[Conditions for 
$\beta$ and $\gamma$]\label{lem:discriminant-is-positive-if-cbf}
    Let Assumption~\ref{as:A-B-stabilizable} hold.
    Furthermore, suppose that the conditions in Proposition~\ref{prop:cbf} hold. Then, $r^2 (\gamma^2 + \beta^2) - T_3^2 > 0$, and in particular $\gamma^2 + \beta^2 > 0$.
    Moreover, if Assumption~\ref{as: interior eq}
    holds, then $\gamma x_{c,1} + \beta x_{c,2} \neq 0$.
\end{lemma}
\begin{proof}
    Let us show that $r^2 (\gamma^2 + \beta^2) - T_3^2 > 0$, which implies that $\gamma^2 + \beta^2 > 0$.
    By noting that $|T_3|+\sqrt{T_2} > 0$, and
    squaring both sides of the condition $\frac{r}{ \sqrt{b_1^2+b_2^2} } > \frac{|T_3|+\sqrt{T_2}}{2T_1}$ in Proposition~\ref{prop:cbf}, we get:
    $\Big( \frac{2T_1 r}{\sqrt{b_1^2+b_2^2}} - |T_3| \Big)^2 > T_2$, which is equivalent to
    $\frac{4T_1^2 r^2}{ b_1^2 + b_2^2 } + T_3^2 - \frac{4 T_1 r |T_3|}{ \sqrt{b_1^2 + b_2^2} } > T_3^2 + 2\alpha_0 r^2 T_1$. Rearranging terms, this yields
    \begin{align}\label{eq:conditions-T3}
        & |T_3| < \frac{(b_2\beta +b_1\gamma)r }{\sqrt{b_1^2 + b_2^2}}.
    \end{align}
    Note that~\eqref{eq:conditions-T3} requires $b_2\beta + b_1\gamma > 0$ since otherwise the conditions in~\eqref{eq:conditions-T3} would not be feasible for any $T_3$.
    %
    %
    Now, by using condition~\eqref{eq:conditions-T3} and applying the Cauchy-Schwartz inequality, we get $T_3 > -\sqrt{b_1^2+b_2^2}r$, $T_3 < \sqrt{b_1^2+b_2^2}r$,
    from which it follows that $r^2 (\gamma^2 + \beta^2) - T_3^2 > 0$. 
    Finally suppose that $\norm{\bx_c}^2 > r^2$ and $\gamma x_{c,1} + \beta x_{c,2} = 0$. Note that $T_3^2 = (-\gamma x_{c,2} + \beta x_{c,1})^2 = (-\gamma x_{c,2} + \beta x_{c,1})^2 + (\gamma x_{c,1} + \beta x_{c,2})^2 = (\gamma^2 + \beta^2)\norm{\bx_c}^2$.
    Since $\norm{\bx_c}^2 > r^2$, this implies that $r^2(\gamma^2+\beta^2)-T_3^2 < 0$, which is a contradiction.
\end{proof}

Next we add details to the Example~\ref{ex:asymp-eq-exist-all-nominal-controllers}. In particular, we elaborate further on the stability properties of undesired equilibria.

\smallskip
\begin{example}\longthmtitle{Example~\ref{ex:asymp-eq-exist-all-nominal-controllers} continued}\label{ex:asymp-eq-exist-all-nominal-controllers-continued}
{\rm
    First, since the boundary of $\Cc$ is given by a union of semicircles, by following an argument similar to that of the proof of Proposition~\ref{prop:condition-h-strict-cbf-underactuated-lti-planar}, one has that 
    the following are sufficient conditions for $h$ to be a strict CBF:
    \begin{subequations}
    \begin{align}
        T_1 > 0, & \ b_1^2 + b_2^2, \\
        \frac{r_1}{ \sqrt{b_2^2 + b_1^2} } &> \frac{  \frac{b_1}{|b_1|}(-\gamma c_2 -\beta c_1) + \sqrt{ (\gamma c_2 +\beta c_1)^2 + 2\alpha_0 r_1^2 T_1 } }{ 2 T_1 },~\label{eq:first-condition} \\
        -\frac{r}{\sqrt{b_2^2 + b_1^2}} &> \frac{ \frac{b_1}{|b_1|}(-\gamma c_2) - \sqrt{\gamma^2 c_2^2 + 2\alpha_0 r^2 T_1} }{ 2T_1 }~\label{eq:second-condition}, \\
        \frac{r_1}{ \sqrt{b_2^2 + b_1^2} } &> \frac{  \frac{b_1}{|b_1|}(-\gamma c_2 + \beta c_1) + \sqrt{ (-\gamma c_2 +\beta c_1)^2 + 2\alpha_0 r_1^2 T_1 } }{ 2 T_1 }~\label{eq:third-condition}, \\
        -\frac{R}{\sqrt{b_2^2 + b_1^2}} &< \frac{ \frac{b_1}{|b_1|}(-\gamma c_2) - \sqrt{\gamma^2 c_2^2 + 2\alpha_0 R^2 T_1} }{ 2T_1 }~\label{eq:fourth-condition}
    \end{align}
    \label{eq:feasibility-conditions-unioncircles-example}
    \end{subequations}
    where  $\beta = a_{11}b_2 - b_1 a_{21}$, $\gamma = a_{22} b_1 - b_2 a_{12}$, and $T_1 := b_2 \beta + b_1 \gamma + \frac{1}{2} \alpha_0 (b_2^2 + b_1^2)$. Further, suppose that $r^2(\gamma^2 + \beta^2) - \gamma^2 c_2^2 \geq 0$, and let 
    \begin{align*}
        z_{+,2} &= \frac{ \beta c_2 + \sqrt{ r^2 (\gamma^2 + \beta^2) - \gamma^2 c_2^2 } }{ \gamma^2 + \beta^2 }.
    \end{align*}
    It follows that the point $\bx_{*,+,2}=(\gamma z_{+,2},\beta z_{+,2})$ is in~$\partial\Cc$ and satisfies~\eqref{eq: condition-eq} for some $\delta\in\real$; this, in turn, means that $\bx_{*,+,2}\in\Ec$.
    To show whether $\bx_{*,+,2}$ is an undesired equilibrium, we need to check if that $-(\bx_{*,+,2}-\bx_c)^\top(A-BK)\bx_{*,+,2} < 0$. 
    By using the expression of $\bx_{*,+,2}$, this condition is equivalent to
    \begin{align}\label{eq:active-filter-condition-appendix}
        z_{+,2} T_4 \big( b_1\gamma z_{+,2} + b_2(\beta z_{+,2}-c_2) \big) > 0,
    \end{align}
    where $T_4 = a_{11}a_{22}-a_{12}a_{21}-k_1\gamma-k_2\beta$. Since $A-BK$ is Hurwitz, $T_4>0$, and therefore~\eqref{eq:active-filter-condition-appendix} is independent of $K$ and equivalent to 
    \begin{align}\label{eq:active-filter-condition-appendix-2}
        z_{+,2} \Big( b_1\gamma z_{+,2} + b_2(\beta z_{+,2}-c_2) \Big) > 0,
    \end{align}
    Now, note that Example~\ref{ex:asymp-eq-exist-all-nominal-controllers} satisfies~\eqref{eq:active-filter-condition-appendix-2}.  Therefore, $\bx_{*,+,2} = (3.157,7.619)$ is an undesired equilibrium for any $K$.
    To show that it is asymptotically stable, note that Example~\ref{ex:asymp-eq-exist-all-nominal-controllers} satisfies
    \begin{align}
        (-\gamma c_2) (\gamma b_2-\beta b_1) + (\gamma b_1 + \beta b_2) \sqrt{r^2 (\gamma^2 + \beta^2) - \gamma^2 c_2^2 } < 0.
        \label{eq:asymptotically-stable-conditions-unioncircles-example}
    \end{align}
    Then, by following the same argument as in the proof of~\cite[Proposition 4]{YC-PM-EDA-JC:24-cdc}, 
    the Jacobian of~\eqref{eq:v-linear-system} at $\bx_{*,+,2}$ is
    \begin{align*}
        J({\bx_{*,+,2}}) = A - \frac{B}{ (\bx_c-\bx_{*,+,2})^\top B } ( (\bx_{*,+,2}-\bx_c)^\top (A+\alpha_0 \textbf{I}_n ),
    \end{align*}
    and~\eqref{eq:asymptotically-stable-conditions-unioncircles-example} implies that $J$ has two negative eigenvalues. Moreover, since $J$ is independent of $K$, this implies that $\bx_{*,+,2} = (3.157,7.619)$ in Example~\ref{ex:asymp-eq-exist-all-nominal-controllers} is asymptotically stable for any choice of linear stabilizing nominal controller.
    \problemfinal
    }
\end{example}

\smallskip

The following results concern Section~\ref{sec:ellips-fully-actuated-LTI}, i.e., the case when the LTI system is fully actuated. We employ the same notation. 
We start by stating two auxiliary results that determine the eigenvalue other than $-\alpha^\prime(0)$ of the Jacobian.

\smallskip

\begin{lemma}\longthmtitle{Eigenvalue of the Jacobian when $\tilde{A}$ is not diagonalizable}\label{lemma: compute eigen for m=2, case 1}
    Assume that $\lambda_1=\lambda_2$ and $\tilde{A}$ is not diagonalizable. Then, there exists $\bv_1$, $\bv_2\in\mathbb{R}^2$ such that $\|\bv_1\|_2=\|\bv_2\|_2=1$, $\tilde{A}\bv_1=\lambda_1\bv_1$ and $\tilde{A}\bv_2=\lambda_2\bv_2+\bv_1$. For any $\bx_*\in\hat{\mathcal{E}}$,  if  the associated indicator $\delta_{\bx_*}\neq \frac{\lambda_1}{2}$,  $\bx_c=\beta_1 \bv_1+\beta_2\bv_2$ and $\bx_*=\beta_3 \bv_1+\beta_4 \bv_2$, then it holds that $\beta_3-\beta_1=\frac{-\lambda_1\beta_1}{\lambda_1-2\delta_{\bx_*}} +\frac{2\delta_{\bx_*} \beta_2}{(\lambda_1-2\delta_{\bx_*})^2}  $, $\beta_4-\beta_2=\frac{-\lambda_1\beta_2}{\lambda_1-2\delta_{\bx_*}} $ and the eigenvalue other than $-\alpha^\prime(0)$ of the Jacobian of~\eqref{eq:general-system-1} at $\bx_*$ is 
    \begin{align*}
  \lambda_1-2\delta_{\bx_*}-\frac{(\beta_4-\beta_2)}{r^2}((\beta_3-\beta_1)  +(\beta_4-\beta_2) \bv_2^\top \bv_1 ).
    \end{align*}
\end{lemma}
\begin{proof}
Let $J(\bx)$ be the Jacobian of~\eqref{eq:general-system-1} evaluated at $\bx$.
If we write $J({\bx_*}) \bv_1=d_{11}\bv_1+d_{21}\bv_2$ and $J({\bx_*}) \bv_2=d_{12}\bv_1+d_{22}\bv_2$, then the other eigenvalue of $J({\bx_*})$ is equal to $d_{11}+d_{22}+\alpha^\prime(0)$.

Using the expression for the Jacobian in~\cite[Proposition 11]{indep-cbf},
\begin{align*}
&d_{11}=-\frac{(\beta_3-\beta_1)\alpha^\prime(0)}{r^2}( (\beta_3-\beta_1)  +(\beta_4-\beta_2) \bv_2^\top \bv_1 )+\\
&\qquad \lambda_1 \! - \! 2\delta_{\bx_*} \! - \! \frac{(\beta_3-\beta_1)(\lambda_1-2\delta_{\bx_*})}{r^2}( (\beta_3-\beta_1) \! + \! (\beta_4-\beta_2)
 \bv_2^\top \bv_1 ), \\
&d_{22}=-\frac{(\beta_4-\beta_2)\alpha^\prime(0)}{r^2}( (\beta_3-\beta_1)\bv_1^\top \bv_2  +(\beta_4-\beta_2)  )\\
&\qquad -\frac{(\beta_4-\beta_2)(\lambda_1-2\delta_{\bx_*})}{r^2}( (\beta_3-\beta_1)\bv_1^\top \bv_2  +(\beta_4-\beta_2)  )\\
&\qquad +\lambda_1-2\delta_{\bx_*}-\frac{(\beta_4-\beta_2)}{r^2}((\beta_3-\beta_1)  +(\beta_4-\beta_2) \bv_2^\top \bv_1 )
\end{align*}
Since $(\bx_*,\delta_{\bx_*})$ is an solution of~\eqref{eq: condition-eq}, it follows that $\beta_3-\beta_1=\frac{-\lambda_1\beta_1}{\lambda_1-2\delta_{\bx_*}} +\frac{2\delta_{\bx_*} \beta_2}{(\lambda_1-2\delta_{\bx_*})^2}  $ and $\beta_4-\beta_2=\frac{-\lambda_2\beta_2}{\lambda_2-2\delta_{\bx_*}} $, by \eqref{eq: condition-eq1}; additionally, by~\eqref{eq: condition-eq2}, one has that
%
%
 \[ \left\|\left(\frac{-\lambda_1\beta_1}{\lambda_1-2\delta_{\bx_*}} +\frac{2\delta_{\bx_*} \beta_2}{(\lambda_1-2\delta_{\bx_*})^2} \right)  \bv_1+ \frac{-\lambda_2\beta_2}{\lambda_2-2\delta_{\bx_*}} \bv_2\right\|^2-r^2=0.\]
Thus, it follows that: 
\begin{align*}
&d_{11}+d_{22}+\alpha^\prime(0)=\lambda_1-2\delta_{\bx_*}-\frac{(\beta_4-\beta_2)}{r^2}((\beta_3-\beta_1)  +(\beta_4-\beta_2) \bv_2^\top \bv_1 ).
\end{align*}
\end{proof}

\smallskip

\begin{lemma}\longthmtitle{The other eigenvalue of Jacobian when $\tilde{A}$ is diagonalizable}\label{lemma: compute eigen for m=2, case 2}
    Assume that $\lambda_1\neq \lambda_2$. Then, there exists $\bv_1$, $\bv_2\in\mathbb{C}^2$ such that $\|\bv_1\|_2=\|\bv_2\|_2=1$, $\tilde{A}\bv_1=\lambda_1\bv_1$ and $\tilde{A}\bv_2=\lambda_2\bv_2$. Additionally, for any $\bx_*\in\hat{\mathcal{E}}$, if  the associated indicator $\delta_{\bx_*}\notin\{\frac{\lambda_1}{2},\frac{\lambda_2}{2}\}$, $\bx_c=\beta_1 \bv_1+\beta_2\bv_2$ and $\bx_*=\beta_3 \bv_1+\beta_4 \bv_2$, it holds that $\beta_3-\beta_1=\frac{-\lambda_1\beta_1}{\lambda_1-2\delta_{\bx_*}} $, $\beta_4-\beta_2=\frac{-\lambda_2\beta_2}{\lambda_2-2\delta_{\bx_*}} $ and the eigenvalues other than $-\alpha^\prime(0)$ of the Jacobian of~\eqref{eq:general-system-1} at $\bx_*$ is 
    $$ \lambda_1+\lambda_2-2\delta_{\bx_*}-\frac{(\beta_3-\beta_1)\lambda_1}{r^2}\Delta_1 -\frac{(\beta_4-\beta_2)\lambda_2}{r^2}\Delta_2,$$
where $\Delta_1:= (\beta_3-\beta_1)^*  +(\beta_4-\beta_2)^*
 \bv_2^* \bv_1$, $\Delta_2:=(\beta_3-\beta_1)^*\bv_1^* \bv_2  +(\beta_4-\beta_2)^*$.
  Equivalently, the eigenvalue other than $\alpha^\prime(0)$  can be expressed as
$\lambda_1-2\delta_{\bx_*}- \frac{(\beta_3-\beta_1)(\lambda_1-\lambda_2)}{r^2}\Delta_1 $ and
$\lambda_2-2\delta_{\bx_*} -\frac{(\beta_4-\beta_2)(\lambda_2-\lambda_1)}{r^2}\Delta_2$.
\end{lemma}
%
%
\begin{proof}
If we write $J({\bx_*}) \bv_1=d_{11}\bv_1+d_{21}\bv_2$ and $J({\bx_*}) \bv_2=d_{12}\bv_1+d_{22}\bv_2$, then the other eigenvalue of $J({\bx_*})$ is equal to $d_{11}+d_{22}+\alpha^\prime(0)$.
Using the expression for the Jacobian in~\cite[Proposition 11]{indep-cbf},
\begin{align*}
&d_{11}=\lambda_1-2\delta_{\bx_*}-\frac{(\beta_3-\beta_1)\alpha^\prime(0)}{r^2}( (\beta_3-\beta_1)^*  +(\beta_4-\beta_2)^* \bv_2^* \bv_1 )\\
&\qquad -\frac{(\beta_3-\beta_1)(\lambda_1-2\delta_{\bx_*})}{r^2}( (\beta_3-\beta_1)^*  +(\beta_4-\beta_2)^*
 \bv_2^* \bv_1 ),
\end{align*}
and
\begin{align*}
&d_{22}=\lambda_2-2\delta_{\bx_*}-\frac{(\beta_4-\beta_2)\alpha^\prime(0)}{r^2}( (\beta_3-\beta_1)^*\bv_1^* \bv_2  +(\beta_4-\beta_2)^*  )\\
&\qquad -\frac{(\beta_4-\beta_2)(\lambda_2-2\delta_{\bx_*})}{r^2}( (\beta_3-\beta_1)^*\bv_1^* \bv_2  +(\beta_4-\beta_2)^*  ).
\end{align*}

Note that since $(\bx_*,\delta_{\bx_*})$ is a solution of~\eqref{eq: condition-eq}, it follows that $\beta_3-\beta_1=\frac{-\lambda_1\beta_1}{\lambda_1-2\delta_{\bx_*}} $, $\beta_4-\beta_2=\frac{-\lambda_2\beta_2}{\lambda_2-2\delta_{\bx_*}}$, and
 \[ \left\|\frac{-\lambda_1\beta_1}{\lambda_1-2\delta_{\bx_*}} \bv_1+ \frac{-\lambda_2\beta_2}{\lambda_2-2\delta_{\bx_*}} \bv_2\right\|^2-r^2=0.\]

Then,
\begin{align*}
d_{11}+d_{22}+\alpha^\prime(0)&=\lambda_1+\lambda_2-2\delta_{\bx_*}-\frac{(\beta_3-\beta_1)\lambda_1}{r^2}( (\beta_3-\beta_1)^*  +(\beta_4-\beta_2)^*
 \bv_2^* \bv_1 ) \\
&\qquad -\frac{(\beta_4-\beta_2)\lambda_2}{r^2}( (\beta_3-\beta_1)^*\bv_1^* \bv_2  +(\beta_4-\beta_2)^*  ).
\end{align*}

By leveraging $\beta_3-\beta_1=\frac{-\lambda_1\beta_1}{\lambda_1-2\delta_{\bx_*}} $ and $\beta_4-\beta_2=\frac{-\lambda_2\beta_2}{\lambda_2-2\delta_{\bx_*}}$, we get the two remaining expressions.
\end{proof}

\vspace{0.5cm}

To prove Proposition~\ref{prop: m=2, x0 is eigen}, we need to determine the Jacobian evaluated at the undesired equilibrium and analyze its spectrum. Applying~\cite[Proposition 6.2]{indep-cbf} and  Lemma~\ref{lem:undesired-eq-characterization} to system~\eqref{eq:v-linear-system}, we have the following result.

\begin{lemma}\longthmtitle{The Jacobian evaluated at undesired equilibria}\label{lemma: Jacobian at undesirable eq}
  Suppose Assumption~\ref{as: interior eq} is satisfied, $B$ is invertible, and $G=B^\top B$. The Jacobian of~\eqref{eq:v-linear-system} at  any undesired equilibrium  $\bx_*$ is 
\[J({\bx_{*}})=\Tilde{A}-2\delta_{\bx_*}\mathbf{I} -\frac{ (\bx_*-\bx_c )(\bx_*-\bx_c )^\top}{\| \bx_*-\bx_c \|^2}(\Tilde{A}-(2\delta_{\bx_*}-\alpha^\prime(0))\mathbf{I}).
\]  

\end{lemma}

\vspace{0.5cm}

\emph{Proof of Proposition~\ref{prop: m=2, x0 is eigen}:}

Denote $\lambda$ the eigenvalue associated with $\bx_c$. Then $\lambda=\lambda_i$, $i=1$ or $2$; and both $\lambda_1$ and $\lambda_2$ are real. We first determine the solution for \eqref{eq: condition-eq} with $\delta\notin \{\frac{\lambda_1}{2},~\frac{\lambda_2}{2}\}$ .
Since $\delta\notin \{\frac{\lambda_1}{2},~\frac{\lambda_2}{2}\}$, by the first equation in~\eqref{eq: condition-eq}, it follows that $\bx_*=\frac{2\delta}{2\delta-\lambda}\bx_c$. Plugging this in the second equation in~\eqref{eq: condition-eq}, we can solve for $\delta$. This leads to the potential undesired equilibria $\bx_{*,-}:=(1+\frac{r}{\|\bx_c\|})\bx_c)$, with associated value of $\delta$ equal to $\frac{\lambda}{2} + \frac{\lambda \norm{x_c}}{2r}$,
and $\bx_{*,+}:= (1-\frac{r}{\|\bx_c\|})\bx_c)$, with associated value of $\delta$ equal to $\frac{\lambda}{2} - \frac{\lambda \norm{x_c} }{2r}$.
We note that the value of $\delta$ associated with $\bx_{*,-}$ is negative and the value of $\delta$ associated with $\bx_{*,-}$ is positive, so $\bx_{*,-}$ is an undesired equilibrium while $\bx_{*,+}$ is not an undesired equilibrium. By Lemma~\ref{lemma: Jacobian at undesirable eq} ,
the Jacobian at $\bx_{*,-}$ is 
\begin{align*}
J({\bx_{*,-}})=\Tilde{A}-2\delta \mathbf{I} -\frac{ \bx_c\bx_c^\top}{\| \bx_c \|^2}(\Tilde{A}-(2\delta-\alpha^\prime(0))\mathbf{I}).
\end{align*}
where $\delta=\frac{\lambda_1}{2}+\frac{\lambda_1 \|\bx_c \|}{2 r} $.

In the following, we determine if there exist solutions of \eqref{eq: condition-eq} with $\delta\in \{\frac{\lambda_1}{2},~\frac{\lambda_2}{2}\}$ and discuss the stability of the corresponding undesired equilibria case by case.




\textbf{Case 1:}  $\tilde{A}$ is not diagonalizable.

In this case, we first show that $\bx_{*,-}$ is always a saddle point. We note that we must have $\lambda_1=\lambda_2$. Let $\bv_1=\frac{\bx_c}{\|\bx_c\|}$, $\bv_2$ be a vector such that $\|\bv_2\|=1$, $\tilde{A} \bv_2=\lambda_1 \bv_2+ \bv_1$. If we write $\bx_c=\beta_1 \bv_1+\beta_2 \bv_2$, then $\beta_1=\|\bx_c\|$ and $\beta_2=0$. By Lemma \ref{lemma: compute eigen for m=2, case 1}, it follows that  the Jacobian at  $\bx_{*,-}$ has an eigenvalue equal to $\lambda_2-2\delta_{\bx_{*,-}}= \lambda_1-\lambda_1-\frac{\lambda_1 \|\bx_c \|}{ r}>0$, implying that   $\bx_{*,-}$ is a saddle point.

Next, we determine if there exists a solution with $\delta=\frac{\lambda_1}{2}$. We write $\bx_*=\beta_3 \bv_1+\beta_4 \bv_2$. Hence the first equation of \eqref{eq: condition-eq} with $\delta=\frac{\lambda_1}{2}$ can  be rewritten as $\beta_4=-\lambda_1\|\bx_c\|$.
Plugging the value of $\beta_4$ into the second equation of \eqref{eq: condition-eq}, and defining $\hat{\beta}_3:=\beta_3-\|\bx_c\|$ and $\tau_1:=\lambda_1\|\bx_c\|$, it follows that
\begin{equation}\label{eq: quadratic equation for case 2}
\hat{\beta}_3^2-2 \tau_1 \bv_1^\top \bv_2  \hat{\beta}_3+\tau_1^2-r^2=0.
\end{equation}
Note that the discriminant of the quadratic equation \eqref{eq: quadratic equation for case 2} is
\begin{align*}
\Delta:=&4\left(\tau_1^2 (\bv_1^\top \bv_2)^2-  \tau_1^2+r^2\right)=4\left(\tau_1^2 ((\bv_1^\top \bv_2)^2-1)  +r^2\right)
\end{align*}
This leads to the following three subcases.

\textbf{Case 1.1} if $(\bv_1^\top \bv_2)^2< 1-r^2/ \tau_1^2=  1-\frac{ r^2}{\lambda_1^2\|\bx_c\|^2}  $, there does not exist a solution associated with $\delta=\frac{\lambda_1}{2}$.

\textbf{Case 1.2} if $(\bv_1^\top \bv_2)^2=1-r^2/ \tau_1^2 =  1-\frac{ r^2}{\lambda_1^2\|\bx_c\|^2} $, then there exists one solution $\bx_{*,1} = (\tau_1 \bv_1^\top \bv_2 +\|\bx_c\|)\bv_1- \tau_1 \bv_2$
with $\delta_{\bx_{*,1}}=\frac{\lambda_1}{2}$, equal to
\begin{align*}
    \bx_{*,1} = (\tau_1 \bv_1^\top \bv_2 +\|\bx_c\|)\bv_1- \tau_1 \bv_2
\end{align*} 
Since $(\bx_{*,1}-\bx_c)^\top \bv_1=0$ and $(\Tilde{A}-2\delta_{\bx_{*,1}}\mathbf{I})\bv_1=0$, it follows that  $J({\bx_{*,1}})^\top \bv_1=0$, by Lemma~\ref{lemma: Jacobian at undesirable eq}.
Therefore in this case, there is another undesired equilibrium $\bx_{*,1}$, at which the Jacobian has a negative eigenvalue and a zero eigenvalue.


\textbf{Case 1.3} if $(\bv_1^\top \bv_2)^2>1-r^2/ \tau_1^2=  1-\frac{ r^2}{\lambda_1^2\|\bx_c\|^2}  $, there exist two solutions $\hat{\beta}_3=\hat{\beta}_3^{(1)}$ and $\hat{\beta}_3=\hat{\beta}_3^{(2)}$ for \eqref{eq: quadratic equation for case 2}. This implies that there exist two extra 
undesired equilibria given by 
$\bx_{*,2} = (\hat{\beta}_3^{(1)}+\|\bx_c\|)\bv_1- \tau_1 \bv_2$ and 
$\bx_{*,3} = (\hat{\beta}_3^{(2)}+\|\bx_c\|)\bv_1- \tau_1 \bv_2$.

Notice that in this sub-case, $\hat{\beta}_3^{(1)}+\hat{\beta}_3^{(2)}=2\tau_1 \bv_1^\top \bv_2$, we can assume that $\hat{\beta}_3^{(1)}<\tau_1 \bv_1^\top \bv_2$ and $\hat{\beta}_3^{(1)}>\tau_1 \bv_1^\top \bv_2$.

Using the same technique in the proof of Lemma \ref{lemma: compute eigen for m=2, case 2}, we can show that $J({\bx_{*,2}})$, has an eigenvalue $\frac{\tau_1}{r^2}(\hat{\beta}_3^{(1)}-\tau_1 \bv_1^\top \bv_2)>0$; and $J({\bx_{*,3}})$ has an eigenvalue $\frac{\tau_1}{r^2}(\hat{\beta}_3^{(2)}-\tau_1 \bv_1^\top \bv_2)<0$.

Hence in this case, there are another two undesired equilibria, one of which is stable and the other one is saddle point.



\textbf{Case 2:} $\tilde{A}$ is diagonalizable and $\lambda_1\leq\lambda_2<0$, $\tilde{A}\bx_c=\lambda_1 \bx_c$ 

In this case, we first show that $\bx_{*,-}$ is always a saddle point. Let $\bv_1=\frac{\bx_c}{\|\bx_c\|}$, $\bv_2$ be an eigenvector associated with $\lambda_2$ satisfying $\|\bv_2\|=1$, $\bv_1^\top \bv_2\geq 0$. If we write $\bx_c=\beta_1 \bv_1+\beta_2 \bv_2$, then $\beta_1=\|\bx_c\|$ and $\beta_2=0$. By Lemma \ref{lemma: compute eigen for m=2, case 2}, it follows that  the Jacobian at  $\bx_{*,-}$ has an eigenvalue equal to
$\lambda_2-\lambda_1-\frac{\lambda_1 \|\bx_c \|}{ r}>0$, implying that   $\bx_{*,-}$ is a saddle point.
Next, we determine if there exists a solution with $\delta\in \{\frac{\lambda_1}{2},~\frac{\lambda_2}{2}\}$. We write $\bx_*=\beta_3 \bv_1+\beta_4 \bv_2$ and then the first equation of \eqref{eq: condition-eq} can  be rewritten as
\begin{equation}\label{eq: equation for case 3}
\begin{aligned}  
   & (\lambda_1-2\delta)(\beta_3-\|\bx_c\|)=-\lambda_1\|\bx_c\|\\
    &(\lambda_2-2\delta)\beta_4=0
\end{aligned}    
\end{equation}
from which it follows that $\delta \neq \frac{\lambda_1}{2}$.

If $\delta= \frac{\lambda_2}{2}$, from the first equation of~\eqref{eq: equation for case 3} it follows that $\beta_3=\frac{-\lambda_2\|\bx_c\|}{\lambda_1-\lambda_2}$. Plugging the value of $\beta_3$ into the equation $h(\bx) = 0$ from~\eqref{eq: condition-eq}, and by defining $\tau_2:=\frac{\lambda_1 \|\bx_c\|}{\lambda_1-\lambda_2}$,
it follows that
\begin{equation}\label{eq: quadratic equation for case 3}    
\beta_4^2-2\tau_2 \bv_1^\top \bv_2 \beta_4+ \tau_2^2-r^2=0.
\end{equation}

Note that  the discriminant of quadratic equation \eqref{eq: quadratic equation for case 3} is
\begin{align*}
\Delta:=&4\left(\tau_2^2 (\bv_1^\top \bv_2)^2-  \tau_2^2+r^2\right)=4\left(\tau_2^2 ((\bv_1^\top \bv_2)^2-1)  +r^2\right),
\end{align*}
which leads to the following three subcases.

\textbf{Case 2.1} if $(\bv_1^\top \bv_2)^2< 1-r^2/ \tau_2^2=  1-\frac{(\lambda_1-\lambda_2)^2 r^2}{\lambda_1^2\|\bx_c\|^2}  $, there does not exist a solution associated with $\delta=\frac{\lambda_2}{2}$.

\textbf{Case 2.2} if $(\bv_1^\top \bv_2)^2=1-r^2/ \tau_2^2=  1-\frac{(\lambda_1-\lambda_2)^2 r^2}{\lambda_1^2\|\bx_c\|^2}  $, then there exists an undesired equilibrium 
\begin{align*}
    \bx_{*,4} = \frac{-\lambda_2\|\bx_c\|}{\lambda_1-\lambda_2} \bv_1+ \frac{\lambda_1\|\bx_c\|}{\lambda_1-\lambda_2} \bv_1^\top \bv_2 \bv_2
\end{align*}
with corresponding $\delta$ equal to $\frac{\lambda_2}{2}$. 
We note that $(\bx_{*,4} -\bx_c)^\top \bv_2=0$ and $(\tilde{A}-2\delta_{\bx_{*,4}}\mathbf{I})\bv_2=0$.  Hence, by Lemma~\ref{lemma: Jacobian at undesirable eq}, we have $J({\bx_{*,4}}) \bv_2=0$.
Thus in this case, the Jacobian evaluated at $\bx_{*,4}$ has a negative eigenvalue and a zero eigenvalue.


\textbf{Case 2.3}  if $(\bv_1^\top \bv_2)^2>1-r^2/ \tau_2^2=  1-\frac{(\lambda_1-\lambda_2)^2 r^2}{\lambda_1^2\|\bx_c\|^2}  $, there exist two solutions $\beta_4=\beta_4^{(1)}$ and $\beta_4=\beta_4^{(2)}$ for \eqref{eq: quadratic equation for case 2}. Then there exist two undesired equilibria $\bx_{*,5} = \frac{-\lambda_2\|\bx_c\|}{\lambda_1-\lambda_2} \bv_1+ \beta_4^{(1)} \bv_2$ and $\bx_{*,6} = \frac{-\lambda_2\|\bx_c\|}{\lambda_1-\lambda_2} \bv_1+ \beta_4^{(2)} \bv_2$ with associated value of $\delta$ equal to $\frac{\lambda_2}{2}$.
Notice that in this sub-case, $\beta_4^{(1)}+\beta_4^{(2)}=2\tau_2 \bv_1^\top \bv_2>0$ and $\beta_4^{(1)}\beta_4^{(2)}=\tau_2^2-r^2>0$, we can assume that $0<\beta_4^{(1)}<\tau_2 \bv_1^\top \bv_2$ and $\beta_4^{(2)}>\tau_2 \bv_1^\top \bv_2$. It follows that $-\tau_2 \bv_1^\top \bv_2 \beta_4^{(1)}+ \tau_2^2-r^2=-\beta_4^{(1)}\beta_4^{(1)}+\tau_2 \bv_1^\top \bv_2 \beta_4^{(1)}>0$ and $-\tau_2 \bv_1^\top \bv_2 \beta_4^{(2)}+ \tau_2^2-r^2=-\beta_4^{(2)}\beta_4^{(2)}+\tau_2 \bv_1^\top \bv_2 \beta_4^{(2)}<0$.
Using the same technique in the proof of Lemma \ref{lemma: compute eigen for m=2, case 1}, we can show that $J({\bx_{*,4}})$ has an eigenvalue $\frac{\lambda_2-\lambda_1}{r^2}(\tau_2^2-\tau_2^2 \bv_1^\top \bv_2 \beta_4^{(1)}-r^2)>0$, and $J({\bx_{*,5}})$ has an eigenvalue $\frac{\lambda_2-\lambda_1}{r^2}(\tau_2^2-\tau_2^2 \bv_1^\top \bv_2 \beta_4^{(2)}-r^2)<0$.
Hence in this case, there are two extra undesired equilibria, one of which is stable and the other one is a saddle point.




\textbf{Case 3:} $\tilde{A}$ diagonalizable, $\lambda_1<\lambda_2<0$, $\tilde{A}\bx_c = \lambda_2 \bx_c$.

Let $\bv_2=\frac{\bx_c}{\|\bx_c\|}$, $\bv_1$ be an eigenvector associated with $\lambda_1$ and $\|\bv_1\|=1$, $\bv_1^\top \bv_2\geq 0$. If we write $\bx_c=\beta_1 \bv_1+\beta_2 \bv_2$, then $\beta_2=\|\bx_c\|$ and $\beta_1=0$. By Lemma \ref{lemma: compute eigen for m=2, case 2}, it follows that  the Jacobian at  $\bx_{*,-}$ has an eigenvalue equal to
$\lambda_1-\lambda_2-\frac{\lambda_2 \|\bx_c \|}{ r}$.
We determine the sign of $\lambda_1-\lambda_2-\frac{\lambda_2 \|\bx_c \|}{ r}$ later.
First, let us determine if there exists a solution with $\delta\in \{\frac{\lambda_1}{2},~\frac{\lambda_2}{2}\}$. We write $\bx_c=\|\bx_c\| \bv_2$, $\bx=\beta_3 \bv_1+\beta_4 \bv_2$ and then from~\eqref{eq: condition-eq} it follows that
\begin{equation}\label{eq: equation for case 4}
\begin{aligned}  
   & (\lambda_1-2\delta)\beta_3=0\\
    &(\lambda_2-2\delta)(\beta_4-\|\bx_c\|)=-\lambda_2\|\bx_c\|,
\end{aligned}    
\end{equation}
from which it follows that $\delta\neq \frac{\lambda_2}{2}$.
If $\delta=\frac{\lambda_1}{2}$, it follows from~\eqref{eq: equation for case 4} that $\beta_4=\frac{-\lambda_1\|\bx_c\|}{\lambda_2-\lambda_1}$. 
Plugging the value of $\beta_4$ into 
the equation $h(\bx) = 0$ from~\eqref{eq: condition-eq}, and by letting $\tau_3:=\frac{\lambda_2 \|\bx_c\|}{\lambda_2-\lambda_1}$, it follows that
\begin{equation}\label{eq: quadratic equation for case 4}    
\beta_3^2-2\tau_3 \bv_1^\top \bv_2 \beta_3+ \tau_3^2-r^2=0. 
\end{equation}
Note that the discriminant of quadratic equation \eqref{eq: quadratic equation for case 4} is
\begin{align*}
\Delta:=&4\left(\tau_3^2 (\bv_1^\top \bv_2)^2-  \tau_3^2+r^2\right)=4\left(\tau_3^2 ((\bv_1^\top \bv_2)^2-1)  +r^2\right),
\end{align*}
which leads to the following three subcases.

\textbf{Case 3.1} if $(\bv_1^\top \bv_2)^2< 1-r^2/ \tau_3^2=  1-\frac{(\lambda_1-\lambda_2)^2 r^2}{\lambda_2^2\|\bx_c\|^2}  $, there does not exist a solution associated with $\delta=\frac{\lambda_1}{2}$.
Recall also that the eigenvalue (other than $-\alpha^\prime(0)$) of Jacobian at  $\bx_{*,-}$ is $\lambda_1-\lambda_2-\frac{\lambda_2 \|\bx_c \|}{ r}$. In this subcase, we have $  1-\frac{(\lambda_1-\lambda_2)^2 r^2}{\lambda_2^2\|\bx_c\|^2} >(\bv_1^\top \bv_2)^2\geq 0 $, which implies that $\lambda_1-\lambda_2-\frac{\lambda_2 \|\bx_c \|}{ r}>0$. Hence in this case, we only have one undesired equilibrium $\bx_{*,-}$, which is a saddle point.

\textbf{Case 3.2} if $ (\bv_1^\top \bv_2)^2= 1-r^2/ \tau_3^2=  1-\frac{(\lambda_1-\lambda_2)^2 r^2}{\lambda_2^2\|\bx_c\|^2} \neq 0$,
there exists an undesired equilibrium equal to 
\begin{align*}
    \bx_{*,7} = \frac{\lambda_2\|\bx_c\|}{\lambda_2-\lambda_1} \bv_1^T \bv_2 \bv_1- \frac{\lambda_1\|\bx_c\|}{\lambda_2-\lambda_1} \bv_2,
\end{align*}
with associated $\delta$ equal to $\delta=\frac{\lambda_1}{2}$.
We note that $(\bx_{*,7}-\bx_c)^\top \bv_1=0$ and $(\tilde{A}-2\delta_{\bx_{*,7}}\mathbf{I}) \bv_2=0$. Hence, by Lemma~\ref{lemma: Jacobian at undesirable eq}, we get that  $J({\bx_{*,7}}) \bv_1=0$. Therefore, $\bx_{*,7}$ is an undersirable equilibrium and a degenerate equilibrium.
Recall that the eigenvalue (other than $-\alpha^\prime(0)$) of the Jacobian at  $\bx_{*,-}$ is $\lambda_1-\lambda_2-\frac{\lambda_2 \|\bx_c \|}{ r}$. 
In this subcase, we still have $  1-\frac{(\lambda_1-\lambda_2)^2 r^2}{\lambda_2^2\|\bx_c\|^2} =(\bv_1^\top \bv_2)^2> 0 $, which implies that $\lambda_1-\lambda_2-\frac{\lambda_2 \|\bx_c \|}{ r}>0$. Hence in this subcase, there are two undesired equilibria $\bx_{*,-}$ and $\bx_{*,7}$, where  $\bx_{*,-}$ is an saddle point and $\bx_{*,7}$ is a degenerate equilibrium.

\textbf{Case 3.3} if $(\bv_1^\top \bv_2)^2= 1-r^2/ \tau_3^2=  1-\frac{(\lambda_1-\lambda_2)^2 r^2}{\lambda_2^2\|\bx_c\|^2} =0$ , there exists one solution associated with $\delta=\frac{\lambda_1}{2}$, which is
\begin{align*}
    \bx_{*,8} = -\frac{\lambda_1\|\bx_c\|}{\lambda_2-\lambda_1} \bv_2.
\end{align*}

Notice that $1-\frac{(\lambda_1-\lambda_2)^2 r^2}{\lambda_2^2\|\bx_c\|^2}=0$ implies that $\frac{-\lambda_2}{\lambda_2-\lambda_1}=\frac{r}{\|\bx_c\|}$, from which it follows that
\[ \bx_{*,8} =- \frac{\lambda_1\|\bx_c\|}{\lambda_2-\lambda_1} \bv_2=(1- \frac{\lambda_2}{\lambda_2-\lambda_1}) \bx_c=(1+\frac{r}{\|\bx_c\|})\bx_c=\bx_{*,-}, \]
and $\lambda_1-\lambda_2-\frac{\lambda_2 \|\bx_c \|}{ r}=0$. Thus in this case, there is only one undesired equilibrium  $\bx_{*,-}$, which is a degenerate equilibrium.

\textbf{Case 3.4} if $(\bv_1^\top \bv_2)^2> 1-r^2/ \tau_3^2=  1-\frac{(\lambda_1-\lambda_2)^2 r^2}{\lambda_2^2\|\bx_c\|^2}  $,  there exist two solutions $\beta_3=\beta_3^{(1)}$ and $\beta_3=\beta_3^{(2)}$ for \eqref{eq: quadratic equation for case 4}. Then there exist two extra undesired equilibria: $\bx_{*,9} = \beta_3^{(1)} \bv_1-\frac{\lambda_1\|\bx_c\|}{\lambda_2-\lambda_1} \bv_2$ and $\bx_{*,10} = \beta_3^{(2)} \bv_1-\frac{\lambda_1\|\bx_c\|}{\lambda_2-\lambda_1} \bv_2$.
Notice that in this sub-case, we have $\beta_3^{(1)}+\beta_3^{(2)}=2\tau_3 \bv_1^\top \bv_2<0$.  Then we can assume that $\beta_3^{(1)}<\tau_3 \bv_1^\top \bv_2$ and $\beta_3^{(2)}>\tau_2 \bv_1^\top \bv_2$.
Using the same technique in the proof of Lemma \ref{lemma: compute eigen for m=2, case 1}, we can show that the $J({\bx_{*,9}})$ has an eigenvalue
\[  \frac{\lambda_1-\lambda_2}{r^2}(\tau_3^2-\tau_3 \bv_1^\top \bv_2\beta_3^{(1)}-r^2 ); \]
 and $J({\bx_{*,10}})$ has an eigenvalue
\[  \frac{\lambda_1-\lambda_2}{r^2}(\tau_3^2-\tau_3 \bv_1^\top \bv_2\beta_3^{(2)}-r^2 ). \]
Recall that the Jacobian evaluated at $\bx_{*,-}$
has an eigenvalue
$\lambda_1-\lambda_2-\frac{\lambda_2 \|\bx_c \|}{ r}$, and then we only need to determine the sign of these three eigenvalues case by case.

\textbf{Case 3.4.1} 
If $0<\tau_3^2-r^2=\frac{\lambda_2^2 \|\bx_c\|^2}{(\lambda_2-\lambda_1)^2}-r^2$, it is easy to check that $\lambda_1-\lambda_2-\frac{\lambda_2 \|\bx_c \|}{ r}>0$. In addition, similar to \textbf{Case 3.3}, we can show that $\{\frac{\lambda_1-\lambda_2}{r^2}(\tau_3^2-\tau_3 \bv_1^\top \bv_2\beta_3^{(i)}-r^2 ):~i=1,2 \}$ contains one positive number and one negative number.
Thus in this case, there are three undesired equilibria in total, two of which are saddle points and one of which is asymptotically stable.

\textbf{Case 3.4.2} If $0=\tau_3^2-r^2=\frac{\lambda_2^2 \|\bx_c\|^2}{(\lambda_2-\lambda_1)^2}-r^2$, it follows that $\lambda_1-\lambda_2-\frac{\lambda_2 \|\bx_c \|}{ r}=0$. In addition, we have $\beta_3^{(2)}=0$ and the point $\bx_{*,10} =\beta_3^{(2)} \bv_1-\frac{\lambda_1\|\bx_c\|}{\lambda_2-\lambda_1} \bv_2$ is equal to $\bx_{*,-}$.
The point  $\bx_{*,9}=\beta_3^{(1)} \bv_1-\frac{\lambda_1\|\bx_c\|}{\lambda_2-\lambda_1} \bv_2$ is a saddle point since the eigenvalue
\begin{align*}
&\frac{\lambda_1-\lambda_2}{r^2}(\tau_3^2-\tau_3 \bv_1^\top \bv_2\beta_3^{(1)}-r^2 )=2\frac{\lambda_2-\lambda_1}{r^2}\tau_3^2 (\bv_1^\top \bv_2 )^2\\
&> 2\frac{\lambda_2-\lambda_1}{r^2}\tau_3^2 \left(1-\frac{(\lambda_1-\lambda_2)^2 r^2}{\lambda_2^2\|\bx_c\|^2} \right)>0 .
\end{align*}  
Thus in this case, there are two undesired equilibria $\bx_{*,-}$ and $\bx_{*,9}$, where  $\bx_{*,-}$ is a degenerate equilibrium  and $\bx_{*,9}$ is an saddle point .

\textbf{Case 3.4.3} If $0>\tau_3^2-r^2=\frac{\lambda_2^2 \|\bx_c\|^2}{(\lambda_2-\lambda_1)^2}-r^2$, it is easy to check that $\lambda_1-\lambda_2-\frac{\lambda_2 \|\bx_c \|}{ r}<0$, which implies that $\bx_{*,-}$ is asymptotically stable.
By $\beta_3^{(1)}\beta_3^{(2)}=\tau_3^2-r^2<0$ and
$\beta_3^{(1)}<\tau_3 \bv_1^\top \bv_2<0$, it follows that $\beta_3^{(2)}>0$.
Using the fact that $\beta_3^{(1)}<\tau_3 \bv_1^\top \bv_2<0$, we can show that \[  \frac{\lambda_1-\lambda_2}{r^2}(\tau_3^2-\tau_3 \bv_1^\top \bv_2\beta_3^{(1)}-r^2 )>0. \]
On the other hand, using the fact that $\beta_3^{(1)}>0>\tau_3 \bv_1^\top \bv_2$, we can show that \[  \frac{\lambda_1-\lambda_2}{r^2}(\tau_3^2-\tau_3 \bv_1^\top \bv_2\beta_3^{(2)}-r^2 )>0. \]
Thus in this case, there are three undesired equilibria in total, two of which are saddle points and one of which is asymptotically stable.

Table~\ref{tab:cases} summarizes  the cases discussed in the proof, except for \textbf{Case 3.3} and \textbf{Case 3.4.2}. \hfill $\square$

\vspace{0.5cm}

\emph{Proof of Proposition \ref{prop: number-of equilibrium-xc-no-eigenvec} }

Denote the eigenvalues of $\tilde{A}$ as $\lambda_1$, $\lambda_2\in\mathbb{C}$. We note that the conditions in~\eqref{eq: condition-eq} can be rewritten as follows:
\begin{align}
    ~~~~~~~ & (\Tilde{A}-2\delta\mathbf{I}_{2\times 2}) ( \bx-\bx_c)=-\Tilde{A}\bx_c ~\text{and,} \label{eq: equations for undesired eq} & \\ 
    ~~~~~~~ &  \|\bx-\bx_c \|^2-r^2=0 \, . \nonumber & ~~~~~~~~~
\end{align}
Next, we consider two cases. 
%
%

$\bullet$ \emph{Case} $\#1$ ($\tilde{A}$ is diagonalizable): Recall that $\bx_c$ is not an eigenvector, so it holds that $\lambda_1\neq \lambda_2$. Let $\bv_1$, $\bv_2\in \mathbb{C}^2$ be eigenvectors such that $\tilde{A}\bv_1=\lambda_1 \bv_1$, $\tilde{A}\bv_2=\lambda_2 \bv_2$, $\|\bv_1\|=\|\bv_2\|=1$. Write $\bx_c$ as $\bx_c=\beta_1 \bv_1+\beta_2 \bv_2$ and $\bx=\beta_3 \bv_1+\beta_4 \bv_2$. Hence, the first equation in \eqref{eq: equations for undesired eq} can  be rewritten as:
\begin{equation}
\begin{aligned}  
   & (\lambda_1-2\delta)(\beta_3-\beta_1)=-\lambda_1\beta_1\\
    &(\lambda_2-2\delta)(\beta_4-\beta_2)=-\lambda_2\beta_2 \, .
\end{aligned}    
\end{equation}
 Note that $\beta_1\neq 0$ and $\beta_2\neq 0$ as $\bx_c$ is not an eigenvector of $A-BK$; it follows that there is no solution with  $\delta\in \{\frac{\lambda_1}{2},   \frac{\lambda_2}{2}\}$. For any solution $(\bx,\delta_{\bx})$ of \eqref{eq: equations for undesired eq}, we have
  that $\beta_3-\beta_1=\frac{-\lambda_1\beta_1}{\lambda_1-2\delta_\bx}$, $\beta_4-\beta_2=\frac{-\lambda_2\beta_2}{\lambda_2-2\delta_\bx} $ and $\left\|\frac{-\lambda_1\beta_1}{\lambda_1-2\delta_\bx} \bv_1+ \frac{-\lambda_2\beta_2}{\lambda_2-2\delta_\bx} \bv_2\right\|^2-r^2=0$,
which is equivalent to $F_1(\delta)=0$, where $F_1(\delta)$ is defined in \eqref{eq: F_1}.

We first note that $F_1(\delta)=0$ can have at most $4$ solutions. Therefore, there are at most four solutions for  \eqref{eq: equations for undesired eq}. In addition, notice that $F_1(-\infty)<0$, $F_1(+\infty)<0$ and $F_1(0)=(\|\bx_c\|^2-r^2)\|\lambda_1\lambda_2\|^2>0$, it follows that there exists at least one solution of \eqref{eq: equations for undesired eq} with positive $\delta$ and at least one solution with negative $\delta$.
If $\lambda_1\leq \lambda_2$, we have $F_1(-\infty)<0$, and $F_1(\frac{\lambda_1}{2})>0$ and there exists at least one solution for \eqref{eq: equations for undesired eq} with $\delta<\frac{\lambda_1}{2}$.

$\bullet$ \emph{Case} $\#2$ ($\tilde{A}$ is not diagonalizable):  In this case, we have $\lambda_1= \lambda_2$. Note that both eigenvalues are negative and $\bx_c$ is not an eigenvector of $\tilde{A}$. Let $\bv_1$, $\bv_2\in \mathbb{R}^2$ be vectors of length $1$, such that $\tilde{A}\bv_1=\lambda_1 \bv_1$, $\tilde{A}\bv_2=\lambda_1 \bv_2+\bv_1$. We write $\bx_c=\beta_1 \bv_1+\beta_2 \bv_2$ and $\bx_*=\beta_3 \bv_1+\beta_4 \bv_2$.
Hence, the first equation in \eqref{eq: equations for undesired eq} can  be rewritten as
\begin{equation}
\begin{aligned}  
   & (\lambda_1-2\delta)(\beta_3-\beta_1)+(\beta_4-\beta_2)=-\lambda_1\beta_1-\beta_2\\
    &(\lambda_2-2\delta)(\beta_4-\beta_2)=-\lambda_2\beta_2 \, .
\end{aligned}    
\end{equation}

Note that $\beta_2\neq 0$  as $\bx_c$ is not an eigenvector of $\tilde{A}$; it follows that there is no solution with  $2\delta= \lambda_1$. For any solution $(\bx_*,\delta_{\bx_*})$ of equation \eqref{eq: equations for undesired eq}, we have
$  
\beta_3-\beta_1=\frac{-\lambda_1\beta_1}{\lambda_1-2\delta_{\bx_*}} +\frac{2\delta_{\bx_*} \beta_2}{(\lambda_1-2\delta_{\bx_*})^2}$, $\beta_4-\beta_2=\frac{-\lambda_2\beta_2}{\lambda_2-2\delta_{\bx_*}} $
and 
$ \left\|\left(\frac{-\lambda_1\beta_1}{\lambda_1-2\delta_{\bx_*}} +\frac{2\delta_\bx \beta_2}{(\lambda_1-2\delta_{\bx_*})^2} \right)  \bv_1+ \frac{-\lambda_2\beta_2}{\lambda_2-2\delta_{\bx_*}} \bv_2\right\|^2-r^2=0,$
which is equivalent to $F_2(\delta_{\bx_*})=0$, where $F_2$ is defined as 
\begin{align*} 
    &F_2(\delta):=-(\lambda_1-2\delta)^4 r^2  +(\lambda_1\beta_2)^2 (\lambda_1-2\delta)^2\\
    &+2(\lambda_1\beta_1(\lambda_1-2\delta)^2- 2\delta(\lambda_1-2\delta)\beta_2)
  \lambda_1\beta_2 \bv_1^\top \bv_2 \\
   &+( 2\delta\beta_2-\lambda_1(\lambda_1-2\delta)\beta_1  )^2.
\end{align*}

We first note that $F_2(\delta)=0$ can have at most $4$ solutions. Therefore, there are four solutions at most for  \eqref{eq: equations for undesired eq}. In addition, notice that  $F_2(+\infty)<0$ and $F_2(0)=(\|\bx_c\|^2-r^2)\lambda_1^4>0$; it follows that there exists at least a solution for \eqref{eq: equations for undesired eq} with positive $\delta$.
Similarly, we have that $\frac{1}{(\lambda_1-2\delta)^4}F_2(\delta)<0$ as $\delta\to -\infty$, and $\frac{1}{(\lambda_1-2\delta)^4} F_2(\delta)\to +\infty $ as $\delta\to \frac{\lambda_1^{-}}{2}$;
then. there exists at least one solution for \eqref{eq: equations for undesired eq}  with negative $\delta<\frac{\lambda_1}{2}$.
\hfill $\square$

\vspace{0.5cm}

\emph{Proof of Proposition \ref{prop: m=2, x0 is not eigen, stab properties}}

Let $\bx_*\in\hat{\mathcal{E}}$ with indicator $\delta_{\bx_*}<\frac{\lambda_1}{2}$, and write $\bx_*=\beta_3 \bv_1+\beta_4 \bv_2$; then, it follows  by Lemma~\ref{lemma: compute eigen for m=2, case 2} that the Jacobian evaluated at $\bx_*$ has an eigenvalue  greater than $\frac{(\lambda_2-\lambda_1)}{r^2} \frac{-\lambda_1}{\lambda_1-2\delta_{\bx_*}} ( (\beta_3-\beta_1)\beta_1  +(\beta_4-\beta_2)\beta_1
 \bv_2^\top \bv_1 )$.
Notice that $\frac{(\lambda_2-\lambda_1)}{r^2} \frac{-\lambda_1}{\lambda_1-2\delta_{\bx_*}}>0$ and
\begin{align*}
    (\beta_3-\beta_1)\beta_1  +(\beta_4-\beta_2)\beta_1
 \bv_2^\top \bv_1 &= \frac{-\lambda_1}{\lambda_1-2\delta_{\bx_*}}\beta_1^2+\frac{-\lambda_2}{\lambda_2-2\delta_{\bx_*}}\beta_1\beta_2
 \bv_2^\top \bv_1\\
 &\geq \frac{-\lambda_2}{\lambda_2-2\delta_{\bx_*}}(\beta_1^2+\beta_1\beta_2
 \bv_2^\top \bv_1)  \geq 0 \,.
\end{align*}
Hence, the Jacobian evaluated at $\bx_*$ has a positive eigenvalue and, thus, $\bx_*$ is a saddle point.
On the other hand, for any $\bx_*\in\hat{\mathcal{E}}$ with indicator $\frac{\lambda_2}{2}<\delta_{\bx_*}<0$, write $\bx_*=\beta_3 \bv_1+\beta_4 \bv_2$; then, by Lemma \ref{lemma: compute eigen for m=2, case 1}, it follows that the Jacobian evaluated at $\bx_*$ has an eigenvalue less than $\frac{(\lambda_2-\lambda_1)}{r^2} \frac{2\lambda_2}{\lambda_2-2\delta_{\bx_*}} ( (\beta_3-\beta_1)\beta_2 \bv_1^\top \bv_2  +(\beta_4-\beta_2)\beta_2 )$.
Notice that $\frac{(\lambda_2-\lambda_1)}{r^2} \frac{2\lambda_2}{\lambda_1-2\delta_{\bx_*}}>0$ and
\begin{align*}
    (\beta_3-\beta_1)\beta_2 \bv_2^\top \bv_1  +(\beta_4-\beta_2)\beta_2 &= \frac{-\lambda_1}{\lambda_1-2\delta_{\bx_*}}\beta_1\beta_2
 \bv_2^\top \bv_1+\frac{-\lambda_2}{\lambda_2-2\delta_{\bx_*}}\beta_2^2
 \\
 &\leq \frac{-\lambda_1}{\lambda_1-2\delta_{\bx_*}}(\beta_1\beta_2
 \bv_2^\top \bv_1+\beta_2^2)   \leq 0.
\end{align*}
Besides, by~\cite[Proposition 10]{indep-cbf}, $-\alpha^{\prime}(0)$ is another eigenvalue. Hence, all the eigenvalues of the Jacobian evaluated at $\bx_*$ are negative, which means that $\bx_*$ is an undesired asymptotically stable equilibrium.

To prove the last claim, let $\delta_0$ denote the only real root of the third-order polynomial $\frac{d F_1(\delta)}{d \delta}$. It follows that $F_1(\delta)$ is monotonically increasing on $(-\infty,\delta_0)$ and monotonically decreasing on $(\delta_0,+\infty)$; this  implies that $F_1(\delta)=0$ only has two solutions. By Lemma \ref{lemma: compute eigen for m=2, case 1}, there is only one undesired equilibrium and its indicator satisfies $\delta<\frac{\lambda_1}{2}$. Since $\beta_1^2+\beta_1 \beta_2 \bv_1^\top \bv_2\geq 0$,  there is only one undesired equilibrium and it is a saddle point. \hfill $\square$

\vspace{0.5cm}

\smallskip
%
%
\begin{remark}\longthmtitle{Connection between the eigenvector-eigenvalue structure}
{\rm Propositions~\ref{prop: m=2, x0 is eigen} and~\ref{prop: m=2, x0 is not eigen, stab properties} describe the number and  stability properties of undesired equilibria based on the specific eigenvector-eigenvalue structure. Proposition~\ref{prop: m=2, x0 is eigen}  (cf. Table \ref{tab:cases}) consists of the cases where $\bx_c$ is an eigenvector of $\tilde{A}$, while Proposition \ref{prop: m=2, x0 is not eigen, stab properties} focuses on the cases where $\bx_c$ is not. We clarify the relationship between them here.
  Proposition \ref{prop: m=2, x0 is not eigen, stab properties}(i) and (ii) correspond to
  %
  %
  the first row  and last row of Table \ref{tab:cases}(a), respectively, but apply to cases where 
$\bx_c$ is not an eigenvector.
When the two eigenvalues are ``highly distinct'', $\lambda_1\ll \lambda_2$, Proposition \ref{prop: m=2, x0 is not eigen, stab properties}(i)  aligns with the first row  of Table \ref{tab:cases}(a).  Specifically,  suppose that $\tilde{A}\bx_c=\lambda_i\bx_c$ and $ (\bv_i^T \bv_j)^2<   1-\frac{(\lambda_i-\lambda_j)^2 r^2}{\lambda_i^2\|\bx_c\|^2}$, then we have $|\beta_1|=0$ or $ |\beta_2|=0$. However, $|\beta_1|=0$ cannot hold with $\lambda_1\ll \lambda_2$, since this implies that $1-\frac{(\lambda_2-\lambda_1)^2 r^2}{\lambda_2^2\|\bx_c\|^2}<0$, contradicting $(\bv_1^T \bv_2)^2<1-\frac{(\lambda_2-\lambda_1)^2 r^2}{\lambda_2^2\|\bx_c\|^2}$.
%
%
Hence if $\lambda_1\ll \lambda_2$, the condition of the first row of Table \ref{tab:cases}(a) says that  $ |\beta_2|=0$ and $\bv_1^\top \bv_2$ is small enough.  If $|\beta_1/\beta_2|\geq1$ (i.e., $\bx_c$ is ``essentially'' an eigenvector associated with $\lambda_1$), then  $\beta_1^2+\beta_1 \beta_2 \bv_1^\top \bv_2\geq |\beta_1|(|\beta_1|- |\beta_2| | \bv_1^\top \bv_2|)\geq|\beta_1|(|\beta_1|- |\beta_2|) \geq 0$. Thus the condition  $\beta_1^2+\beta_1 \beta_2 \bv_1^\top \bv_2\geq 0$ can be viewed as a counterpart to the first row of Table \ref{tab:cases}(a),  when $\lambda_1\ll \lambda_2$.

When the two eigenvectors $\bv_1$ and $\bv_2$ are ``highly distinct'', i.e., $\bv_1^\top\bv_2$ is small enough,  Proposition \ref{prop: m=2, x0 is not eigen, stab properties}(ii) corresponds to the last row of Table \ref{tab:cases}(a). Similarly, we have  $|\beta_1|=0$ or $ |\beta_2|=0$ if $\bx_c$ is an eigenvector. However, if $|\beta_2|=0$, one has that  $1-\frac{(\lambda_1-\lambda_2)^2 r^2}{\lambda_1^2\|\bx_c\|^2}>0$ together with small enough $\bv_1^\top\bv_2$ violate the inequality  $(\bv_1^T \bv_2)^2>1-\frac{(\lambda_1-\lambda_2)^2 r^2}{\lambda_1^2\|\bx_c\|^2}$.
%
%
Hence if $\bv_1^\top\bv_2$ is small enough, $\tilde{A}\bx_c=\lambda_i\bx_c$ and $ (\bv_i^T \bv_j)^2>   1-\frac{(\lambda_i-\lambda_j)^2 r^2}{\lambda_i^2\|\bx_c\|^2}$  imply that  $ |\beta_1|=0$ and $\frac{|\lambda_1|}{|\lambda_2|}>\frac{\| \bx_c\|}{r}\sqrt{1-(\bv_1^T \bv_2)^2}+1$, i.e.,  $\lambda_1\ll \lambda_2$. If $|\beta_2/\beta_1|\geq1$ (i.e., $\bx_c$ is ``essentially'' an eigenvector associated with $\lambda_2$), then  $\beta_2^2+\beta_1 \beta_2 \bv_2^\top \bv_1\geq |\beta_2|(|\beta_2|- |\beta_1| | \bv_1^\top \bv_2|)\geq  |\beta_2|(|\beta_2|- |\beta_1| ) \geq 0$.  Hence, provided that $\bv_1^\top\bv_2$ is small enough, the condition  $\beta_1\beta_2
 \bv_2^\top \bv_1+\beta_2^2  \geq  0$ can be viewed as a counterpart to the last row of Table \ref{tab:cases}(a).   
 \hfill $\Box$
 }
\end{remark}
%

\end{document}